\documentclass[12pt, twoside]{article}

\usepackage{amsmath, amssymb, amsthm, physics, url, bbm, mathrsfs, cite, xcolor, mathtools, dsfont, esint}

\usepackage{graphicx}
\usepackage[shortlabels]{enumitem}
\usepackage[title]{appendix}

\mathtoolsset{showonlyrefs}
\usepackage[utf8]{inputenc}

\definecolor{darkblue}{RGB}{32,57,231}
\usepackage[colorlinks,
	citecolor=darkblue, 
	linkcolor=darkblue, 
	urlcolor=darkblue,
	bookmarks=true,
	bookmarksopen=true,
	pdfauthor={Fabrice Baudoin, Aobo Chen, Li Chen}, 
	pdftitle={Riesz transform on eventually Gaussian local trees}
]{hyperref}

\allowdisplaybreaks

\numberwithin{equation}{section}
\numberwithin{figure}{section}

\newtheorem{theorem}{Theorem}[section]
\newtheorem{lemma}[theorem]{Lemma}
\newtheorem{proposition}[theorem]{Proposition}
\newtheorem{corollary}[theorem]{Corollary}

\theoremstyle{definition}
\newtheorem{definition}[theorem]{Definition}
\newtheorem{example}[theorem]{Example}

\newtheorem{remark}[theorem]{Remark}
\newtheorem{framework}[theorem]{Framework}
\newtheorem{notation}[theorem]{Notation}

\def\dist{{\mathop {{\rm dist}}}}
\def\diam{{\mathop{{\rm diam }}}}
\newcommand*{\dif}{\mathop{}\!\mathrm{d}}
\newcommand{\supp}{\mathrm{supp}}
\newcommand{\one}{\mathds{1}}

 \def\sE {{\mathcal E}} 
\def\sG {{\mathcal G}} \def\sH {{\mathcal H}} 
  \def\sL {{\mathcal L}}

  \def\sX {{\mathcal X}}

 \def\bN {{\mathbb N}} 
  \def\bR {{\mathbb R}}

 \def\bZ {{\mathbb Z}}

\def\bfa {{\mathbf a}}  \def\bfc {{\mathbf c}} 
\def\bfk {{\mathbf k}}

\def\ol{\overline}
\def\wt{\widetilde}

\newcommand{\loc}[0]{\operatorname{loc}}
\newcommand{\set}[1]{\left\{ #1 \right\}}
\newcommand{\Sett}[2]{\left\{ #1  : \, #2 \right\}}

\newcommand\restr[2]{{
		\left.\kern-\nulldelimiterspace 
		#1 
		\vphantom{\big|}
		\right|_{#2}
}}

\newcommand{\ambient}{\mathcal{T}} 
\newcommand{\metric}{d} 
\newcommand{\meas}{m} 
\newcommand{\skeleton}{\mathcal{S}} 
\newcommand{\medm}{\nu} 
\newcommand{\wgrad}{\partial} 
\newcommand{\abscon}{\mathrm{AC}}
\DeclareMathOperator*{\einf}{einf}
\DeclareMathOperator*{\esup}{esup}
\DeclareMathOperator*{\osc}{osc}

\newcommand{\form}{\mathcal{E}} 
\newcommand{\domain}{\mathcal{F}} 
\newcommand{\gen}{\Delta} 
\newcommand{\Dom}{\mathrm{Dom}} 
\newcommand{\proj}{{E}} 

\newcommand{\riesz}{\mathcal{R}} 
\newcommand{\high}{\sharp} 
\newcommand{\low}{\flat} 

\newcommand{\volume}{V}
\newcommand{\scale}{\Psi}
\newcommand{\graph}{\mathsf{G}} 
\newcommand{\vertex}{\mathsf{V}}
\newcommand{\edge}[0]{\mathsf{E}}

\newcommand{\UVG}[0]{\hyperlink{UVG}{\mathrm{UVG}(\volume)}}
\newcommand{\HKE}[0]{\hyperlink{HKE}{\mathrm{HKE}(\scale)}}
\newcommand{\UHK}[0]{\hyperlink{UHK}{\mathrm{UHK}(\scale)}}

\newcommand{\Grad}[0]{\hyperlink{Grad}{\mathrm{Grad}(\scale)}}
\newcommand{\RH}[0]{\hyperlink{RH}{\mathrm{RH}}}
\newcommand{\R}[1]{\hyperlink{R}{\mathrm{R}_{ #1 }}}
\newcommand{\RR}[1]{\hyperlink{RR}{\mathrm{RR}_{ #1 }}}
\newcommand{\DN}[1]{\hyperlink{DN}{\mathrm{DN}_{ #1 }}}

\newcommand{\Ker}{\operatorname{Ker}}
\newcommand{\Leb}[0]{\mathcal{L}^{1}}
\newcommand{\Length}[0]{\mathrm{Length}}

\font\titlefont=cmbx12 scaled 1400
	\title{\titlefont Riesz transform on eventually Gaussian local trees}
\date{\today}
\author{Fabrice Baudoin, Aobo Chen, Li Chen}

\begin{document}

\maketitle

\vspace{-22pt}
\begin{abstract}
We study the Riesz transform $\riesz=\wgrad(-\gen)^{-\frac{1}{2}}$ on uniform local trees, metric measure spaces that are locally real trees and whose canonical Dirichlet form is built from weak derivatives along the skeleton. The reference measure $\meas$ may be singular with respect to the length measure $\medm$, so boundedness of $\riesz$ is understood from $L^{p}(\ambient,\meas)$ to $L^{p}(\skeleton,\medm)$. In contrast with fractal-like manifolds and cable systems, the diffusion is sub-Gaussian at small scales and Gaussian at large scales. Under uniform volume growth, two-sided heat kernel estimates and a pointwise gradient estimate for the heat kernel, we prove that a local Dini condition on the scale function implies boundedness of $\riesz$ on $L^{p}$ for every $p\in[2,\infty)$, and hence the reverse Riesz inequality for every $p\in(1,2]$. Conversely, boundedness of $\riesz$ for some $p<2$, or a reverse Riesz inequality for some $p>2$, forces the space to be one-dimensional at small scales. We show that a reverse H\"older inequality for harmonic functions yields the gradient estimate, and verify all hypotheses for spaces carrying a geometric group action whose generators have bounded displacement. As an application, for the alternating Vicsek fractafold in $\bZ^{d}$ we determine the exact ranges of $p$ for which the Riesz and reverse Riesz inequalities hold.
\end{abstract}

\tableofcontents

\section{Introduction}
In Euclidean space $\bR^{n}$, the Riesz transform $\nabla(-\Delta)^{-\frac{1}{2}}$ is a singular integral operator that is bounded on $L^p(\bR^n)$ for every $p\in(1,\infty)$. It compares the gradient with the square root of the Laplacian through the estimate $\norm{\nabla f}_{L^p}\lesssim\norm{(-\Delta)^{\frac{1}{2}}f}_{L^p}$. Strichartz \cite{Str83} raised the question of boundedness of the Riesz transform on complete Riemannian manifolds. If the Riesz inequality and the reverse Riesz inequality hold at the same exponent $p$, then $\norm{\nabla f}_{L^p}\asymp\norm{(-\Delta)^{\frac{1}{2}}f}_{L^p}$ for functions in their common domain.

Bakry \cite{Bak87} proved that the Riesz transform is bounded on $L^p$ for every $p\in(1,\infty)$ on complete Riemannian manifolds with non-negative Ricci curvature. Coulhon and Duong \cite{CD99} proved that, on complete Riemannian manifolds, volume doubling and a Gaussian heat kernel upper bound imply boundedness of the Riesz transform for $p\in(1,2]$. These assumptions do not suffice when $p>2$. For example, the connected sum $\bR^2\#\bR^2$ satisfies both assumptions, but its Riesz transform is unbounded for every $p>2$ \cite[Section~5]{CD99}. The conjecture of Coulhon and Duong \cite{CD03} has recently been resolved by R.~Chen, Jiang, Li and Li \cite{CJLL26}. Their result establishes $L^p$ boundedness of the Riesz transform for $1<p\leq2$, together with a weak-type $(1,1)$ estimate, on arbitrary complete non-compact Riemannian manifolds.

For $p>2$, finer regularity properties of the heat kernel and harmonic functions become relevant, notably gradient estimates; see \cite{ACDH04,CJKS20}. More precisely, under volume doubling and a scale-invariant $L^2$-Poincar\'e inequality, Auscher, Coulhon, Duong and Hofmann \cite[Theorem~1.2]{ACDH04} proved that a gradient estimate for the heat semigroup of the form
\begin{equation}\label{e.intro-background-gradient}
\norm{\nabla e^{t\Delta}f}_{L^p}\leq Ct^{-\frac{1}{2}}\norm{f}_{L^p},\ \text{ for all }t\in(0,\infty)\text{ and }f\in L^p,
\end{equation}
for some $p>2$ implies boundedness of the Riesz transform on $L^q$ for every $q\in(2,p)$. Here $C$ is independent of $t$ and $f$. Reverse H\"older inequalities for gradients of harmonic functions are also used to prove boundedness of the Riesz transform; see \cite{AC05,CJKS20}.

The development of analysis on fractals led to the study of Riesz transforms on \emph{fractal-like manifolds and graphs}. These spaces exhibit \emph{sub-Gaussian diffusion at large scales}, while the manifolds and the corresponding cable systems retain \emph{Gaussian behaviour at small scales}. Thus, at large scales, diffusion is slower than in the Gaussian case. L.~Chen, Coulhon, Feneuil and Russ \cite{CCFR17} showed that, on complete manifolds satisfying volume doubling, Gaussian small-time and sub-Gaussian large-time heat kernel upper bounds still imply boundedness of the Riesz transform for $p\in(1,2]$. This range is sharp on Vicsek manifolds and graphs, where the Riesz transform is unbounded for every $p>2$. Feneuil \cite[Theorems~1.3--1.4]{Fen26a} established this phenomenon for broader classes of graphs and manifolds under sub-Gaussian estimates and additional regularity assumptions. These results show that the underlying geometry plays an important role in determining the boundedness of the Riesz transform.

We also recall the following results on quasi-Riesz transforms and gradient estimates for the heat kernel.
\begin{itemize}[leftmargin=*,topsep=5pt,parsep=0pt,itemsep=4pt]
\item L.~Chen \cite{Che15} introduced quasi-Riesz transforms and proved that, on complete non-compact Riemannian manifolds, $\nabla e^{\Delta}(-\Delta)^{-\alpha}$ is bounded on $L^p$ for every $p\in(1,2]$ and $\alpha\in(0,\frac{1}{2})$, without additional geometric assumptions.
\item Devyver, Russ and Yang \cite{DRY23} proved pointwise gradient estimates for the heat kernel on fractal-like cable systems and used them to establish boundedness of quasi-Riesz transforms. Their examples include the Vicsek and Sierpi\'nski gasket cable systems. On the Vicsek cable system, Devyver and Russ \cite{DR26} studied the values of $\alpha>0$ and $p\in(1,\infty)$ for which $\norm{(-\Delta)^\alpha e^{\Delta}f}_{L^p}\lesssim\norm{\nabla f}_{L^p}$ holds.
\end{itemize}

In this paper, we study the Riesz transform on metric measure spaces whose diffusion is \emph{sub-Gaussian at small scales} and \emph{Gaussian at large scales}. To define a gradient operator, we consider spaces that are \emph{locally trees}. On such spaces, absolutely continuous functions admit weak derivatives, which also define a canonical regular Dirichlet form. More precisely, unless otherwise specified, we work under the following framework.

\begin{framework}\label{f.treeMMD} We impose the following assumptions on the quintuple $(\ambient,\metric,\meas,\form,\domain)$.
\begin{enumerate}[label=\textup{(F{\arabic*})},align=right,leftmargin=*,topsep=5pt,parsep=0pt,itemsep=2pt]
	\item $(\ambient,\metric)$ is an unbounded, proper, geodesic, and separable uniform local tree with constant $\iota$, skeleton $\skeleton$, and length measure $\medm$; see Definitions~\ref{d.tree1} and \ref{d.tree2}. Let $\meas$ be a Radon measure on $(\ambient,\metric)$ satisfying a uniform volume growth $\UVG$.
\item Fix an arbitrary orientation $\mathfrak{O}$ of $(\ambient,\metric)$ (see Definition~\ref{d.tree3}). Let $(\form,\domain)$ be the strongly local regular symmetric Dirichlet form on $L^2(\ambient,\meas)$ constructed in Proposition~\ref{p.energy}, which is independent of the choice of orientation $\mathfrak{O}$. Let $(-\gen,\Dom(-\gen))$ be its non-negative self-adjoint generator on $L^2(\ambient,\meas)$.
\item We write $\wgrad:=\wgrad_{\mathfrak{O}}$ for the weak gradient defined in Definition~\ref{d.tree3}. The Riesz transform $\riesz:=\riesz_{\mathfrak{O}}:L^{2}(\ambient,\meas)\to L^{2}(\skeleton,\medm)$ denotes the $L^2$-isometric extension of the operator $\wgrad(-\gen)^{-\frac{1}{2}}$ constructed in Lemma~\ref{l.defriesz}-\ref{it.dres2}.
\end{enumerate}
\end{framework}

The length measure $\medm$, against which we measure the weak gradient $\wgrad$, need not be a Radon measure and may be mutually singular with the reference measure $\meas$. Thus, boundedness of the Riesz transform means an estimate from $L^p(\ambient,\meas)$ to $L^p(\skeleton,\medm)$. In the singular case, the theory based on a \emph{carr\'e du champ} with respect to $\meas$ does not apply directly. 

We focus on the following two inequalities.

\begin{definition}
		Let the quintuple $(\ambient,\metric,\meas,\form,\domain)$ be as in Framework~\ref{f.treeMMD}. Let $p\in(1,\infty)$.
\begin{enumerate}[label=\textup{({\arabic*})}, align=right, leftmargin=*, topsep=5pt, parsep=0pt, itemsep=2pt]
\item We say that the \hypertarget{R}{\emph{$L^{p}$-Riesz inequality}} $\R{p}$ holds if there exists a constant $C\in(0,\infty)$ such that
\begin{equation}\label{e.intro-riesz}
\norm{\riesz f}_{L^{p}(\skeleton,\medm)}\leq C\norm{f}_{L^{p}(\ambient,\meas)},\ \text{ for all }f\in L^{p}(\ambient,\meas)\cap L^{2}(\ambient,\meas).
  \end{equation}
\item We say that the \hypertarget{RR}{\emph{$L^{p}$-reverse Riesz inequality}} $\RR{p}$ holds if there exists a constant $C\in(0,\infty)$ such that
 \begin{equation}\label{e.reverse}
\norm{(-\gen)^{\frac{1}{2}}f}_{L^p(\ambient,\meas)}\leq C\norm{\wgrad f}_{L^p(\skeleton,\medm)},\ \text{ for all }f\in\Dom((-\gen)^{\frac{1}{2}})=\domain.
\end{equation}
Here the operator $((-\gen)^{\frac{1}{2}},\Dom((-\gen)^{\frac{1}{2}}))$ is defined by the spectral calculus.
\end{enumerate}
\end{definition}

We describe the heat kernel estimates using a \emph{scale function} $\scale$; see Definitions \ref{d.HKE} and \ref{d.Grad}. Let $\scale:(0,\infty)\to(0,\infty)$ be a continuous, increasing bijection such that, for some $C\in(1,\infty)$ and $1<\beta_{1}\leq\beta_{2}<\infty$,
\begin{equation}\label{e.scale}
		C^{-1}\left(\frac{R}{r}\right)^{\beta_{1}}\leq \frac{\scale(R)}{\scale(r)}\leq C \left(\frac{R}{r}\right)^{\beta_{2}},\ \text{ for all }0<r\leq R<\infty.
	\end{equation}
For $p\in(2,\infty)$, we say that the \hypertarget{DN}{\emph{local Dini condition}} $\DN{p}$ holds for $\scale$ if
\begin{equation}\label{e.near0int}\tag{$\mathrm{DN}_{p}$}
\int_0^1\left(\frac{\scale(r)}{r^2}\right)^{\frac{1}{2}-\frac{1}{p}}\frac{\dif r}{r}<\infty.
\end{equation}

Our first main result is the following.
\begin{theorem}\label{t.main}
Assume $\UVG$, $\HKE$, and $\Grad$. Suppose that $\scale$ is \emph{eventually Gaussian} \ref{e.EventGau}: there exist $R_*,C\in(0,\infty)$ such that
\begin{equation}\label{e.EventGau}\tag{$\mathrm{G}_{\infty}$}
\scale(r)\leq Cr^2,\qquad\text{for all }r\in[R_*,\infty).
\end{equation}
Assume, in addition, that one of the following conditions holds:
\begin{enumerate}[label=\textup{(\roman*)},leftmargin=*,topsep=5pt,parsep=0pt,itemsep=2pt]
\item\label{it.DNp} the condition $\DN{q}$ holds for some exponent $q\in(2,\infty)$;
\item\label{it.Lin} $\scale$ is \emph{locally Gaussian} \ref{e.linear}: there exist $r_*,C\in(0,\infty)$ such that
\begin{equation}\label{e.linear}\tag{$\mathrm{G}_{0}$}
C^{-1}r^2\leq\scale(r)\leq Cr^2,\qquad\text{for all }r\in(0,r_*].
\end{equation}
\end{enumerate}
Then the $L^{p}$-Riesz inequality $\R{p}$ holds for every $p\in[2,\infty)$, and the $L^{q}$-reverse Riesz inequality $\RR{q}$ holds for every $q\in(1,2]$.
\end{theorem}
\begin{remark}
We make the following remarks on Theorem~\ref{t.main}.
	\begin{enumerate}[label=\textup{({\alph*})}, align=right, leftmargin=*, topsep=5pt, parsep=0pt, itemsep=2pt]
  \item By Proposition~\ref{p.swap}, the validities of the $L^{p}$-Riesz inequality $\R{p}$ and the $L^{p}$-reverse Riesz inequality $\RR{q}$ are independent of the choice of the orientation $\mathfrak{O}$.
\item For $p=2$, the Riesz transform $\riesz$ is an $L^{2}$-isometry by Lemma~\ref{l.defriesz}.
\item By \cite[Theorem~2.6]{Mur25}, under Framework~\ref{f.treeMMD}, $\UVG$, and $\HKE$, we have
\begin{equation}\label{e.intro-scale-lower}
\frac{\scale(R)}{\scale(r)}\gtrsim\frac{R^2}{r^2},\qquad 0<r\leq R.
\end{equation}
See also Lemma~\ref{l.growthPsi}. Consequently, under $\UVG$ and $\HKE$, eventual Gaussianity \ref{e.EventGau} is equivalent to $\scale(r)\asymp r^2$ for all sufficiently large $r$. This estimate also gives $\scale(r)/r^2\lesssim1$ for small $r$. Thus $\DN{p}$ implies $\DN{q}$ for every $q\geq p>2$.
  \item If $\scale(r)\lesssim r^\beta$ near zero for some $\beta>2$, then $\DN{p}$ holds for every $p>2$. In the locally Gaussian case \ref{e.linear}, $\DN{p}$ fails for all $p\in(2,\infty)$. The two alternatives do not exhaust all possible scales. For example, $\scale(r)\asymp r^2(\log(e/r))^{-1}$ near zero satisfies neither. Theorem~\ref{t.main} does not settle the boundedness of $\riesz$ in this case.
\end{enumerate}
\end{remark}

We briefly outline the proof of Theorem~\ref{t.main}. Since $\R{p}$ implies $\RR{p/(p-1)}$ (see Lemma~\ref{l.reverse-duality}), it suffices to prove boundedness of the Riesz transform $\riesz$. In the case of Theorem~\ref{t.main}-\ref{it.Lin}, we prove that $\meas\asymp\medm$ and apply \cite[Theorem~1.9]{CJKS20}; see Proposition~\ref{p.onedim}. In the case of Theorem~\ref{t.main}-\ref{it.DNp}, we decompose $\riesz$ into two parts:
\begin{equation}\label{e.intro-frequency}
\riesz=\riesz_{\high}+\riesz_{\low},\ \text{ where }\ \riesz_{\high}:=\riesz\circ(I-P_1),\ \text{ and }\ \riesz_{\low}:=\riesz\circ P_1,
\end{equation}
where $P_t=\exp(t\gen)$ is the heat semigroup.  This decomposition also appears in the third author's work on quasi-Riesz transforms \cite[Section~2.1]{Che15}. We call $\riesz_{\high}$ and $\riesz_{\low}$ the \emph{high-} and \emph{low-} frequency parts of the Riesz transform $\riesz$, respectively, in accordance with the terminology in \cite[Remark~2.1]{Che15}.
\begin{enumerate}[label=\textup{({\alph*})}, align=right, leftmargin=*, topsep=5pt, parsep=0pt, itemsep=2pt]
\item For $\riesz_{\low}$, we discretize the space using a maximal $1$-separated net $\vertex$ with cells $K_{v}$. The associated graph is roughly isometric to $\ambient$ and is doubling (Proposition~\ref{p.count}). For fixed $s\in(0,\infty)$, we estimate the sublinear map $Tf(v):=\norm{\wgrad(-\gen+\kappa)^{-\frac{1}{2}}P_{s}f}_{L^{2}(\skeleton_{v},\medm)}$ uniformly in $\kappa\in(0,\infty)$. The argument adapts the good-$\lambda$ criterion of Auscher, Coulhon, Duong and Hofmann \cite[Theorem~2.1 and Lemmas~2.2, 2.3]{ACDH04}. This is where \ref{e.EventGau} is used.
\item For $\riesz_{\high}$, interpolating the pointwise bound $\Grad$ with an $L^{2}$ estimate gives $\norm{\wgrad P_{t}f}_{L^{p}(\skeleton,\medm)}\lesssim t^{-\frac{1}{p}}\scale^{-1}(t)^{-1+\frac{2}{p}}\norm{f}_{L^{p}(\ambient,\meas)}$ (Proposition~\ref{p.migrad}). We then use spectral calculus and duality to estimate $\riesz_{\high}$. This is where $\DN{p}$ is used.
\end{enumerate}

Our second main result gives a necessary geometric condition for $\R{q}$ to hold for some $q\in(1,2)$ or for $\RR{p}$ to hold for some $p\in(2,\infty)$.
\begin{theorem}\label{t.smallp}
Assume $\UVG$ and $\HKE$. If $\R{q}$ holds for some $q\in(1,2)$, or if $\RR{p}$ holds for some $p\in(2,\infty)$, then there exist $r_0\in(0,\infty)$ and $C\in(1,\infty)$ such that
\begin{equation}\label{e.intro-rigidity}
C^{-1}r\leq\volume(r)\leq Cr,\qquad C^{-1}r^2\leq\scale(r)\leq Cr^2,\qquad\text{for all }r\in(0,r_0],
\end{equation}
and
\begin{equation}\label{e.intro-measure-comparison}
	C^{-1}\medm(A)\leq \meas(A)\leq C\medm(A),\ \text{ for every Borel set }A\subset\ambient.
\end{equation}
Here we regard $\medm$ as a measure on $\ambient$ concentrated on $\skeleton$.
\end{theorem}

Heat kernel estimates $\HKE$ on metric measure spaces admit several equivalent characterizations; see, for example, \cite{GT12,Lie15,GH14,GHL15}. Baudoin, L.~Chen and Yang \cite{BCY25} proved local pointwise gradient estimates for the heat kernel for uniform local trees. To establish a \emph{global} gradient estimate $\Grad$ for a uniform local tree, we use the \emph{reverse H\"older inequality} $\RH$:

\begin{equation}\label{e.intro-harmonic-rh}
\norm{\wgrad h}_{L^{\infty}(B(x,r),\medm)}\leq\frac{C}{r}\fint_{B(x,Ar)}\abs{h}\dif\meas,\ \text{for functions $h$ that are $\form$-harmonic on $B(x,Ar)$}.
\end{equation}
 This is the analogue in our setting of the reverse H\"older inequalities in \cite{AC05,BF16,CJKS20}. Theorem~\ref{t.RH>Gd} establishes the implication
\begin{equation}\label{e.RH=>Gad}
\UVG+\UHK+\RH\;\Longrightarrow\;\Grad.
\end{equation}
It therefore remains to verify $\RH$. Motivated by \cite{CJKS20}, we show in Theorem~\ref{t.Gp} that, under $\UVG$, a geometric group action that preserves $(\form,\domain)$ and whose generators have uniformly bounded displacement gives $\HKE$, $\RH$, $\Grad$, and \ref{e.EventGau}. Theorem~\ref{t.smallp} then applies, and Theorem~\ref{t.main} applies when its additional local Dini or Gaussian condition holds for $\scale$.

As an application, for every integer $d\geq2$, Example~\ref{ex.Vicsek} provides an alternating Vicsek fractafold with
\begin{equation}\label{e.intro-vicsek-scales}
\volume(r)=r^{\alpha_d}\vee r^d,\ \ \scale(r)=r^{\alpha_d+1}\wedge r^2,\ \ r\in[0,\infty), \ \text{ where }\alpha_d=\log_3(2^d+1).
\end{equation}
Since $\alpha_d>1$, the condition $\DN{p}$ holds for every $p>2$. For $p\in(1,\infty)$, Theorems~\ref{t.main} and~\ref{t.smallp} therefore give that, 
\begin{equation}\label{e.intro-exact-ranges}
\text{for the alternating Vicsek fractafold, we have }\begin{cases}
\R{p}\text{ holds if and only if }p\in[2,\infty),\\
\RR{p}\text{ holds if and only if }p\in(1,2].
\end{cases}
\end{equation}

This paper is organized as follows. In Section~\ref{s.prelim}, we review local trees, define orientations, the canonical Dirichlet form and the Riesz transform, and discretize the ambient space. In Section~\ref{s.heat}, we establish consequences of the heat kernel estimates that will be used in the proofs. Section~\ref{s.riesz} proves Theorems~\ref{t.main} and~\ref{t.smallp}. In Section~\ref{s.RHG} , we prove \eqref{e.RH=>Gad}, verify the required assumptions using geometric group actions in Theorem~\ref{t.Gp}, and prove \eqref{e.intro-exact-ranges} for the alternating Vicsek fractafold. Appendix~\ref{s.appendix} collects some useful estimates.

\begin{notation}
In this paper, we use the following notation and conventions.
\begin{enumerate}[label=\textup{(\roman*)},align=right,leftmargin=*,topsep=5pt,parsep=0pt,itemsep=2pt]
\item $\bN:=\{1,2,\ldots\}$. Thus $0\notin \bN$.
\item Let $U$ and $V$ be open subsets of a topological space. If $U$ is precompact and its closure is contained in $V$, then we write $U\Subset V$.
\item Let $X$ be a non-empty set. For $A\subset X$, define $\one_A\in\bR^X$ by
\begin{equation}\label{e.intro-indicator}
\one_A(x):=
\begin{dcases}
1,&x\in A,\\
0,&x\notin A.
\end{dcases}
\end{equation}
\item Let $X$ be a topological space. We set
\begin{equation}\label{e.intro-continuous-spaces}
\begin{aligned}
C(X)&:=\Sett{f\in\bR^{X}}{f\text{ is continuous and real-valued}},\\
C_c(X)&:=\Sett{f\in C(X)}{X\setminus f^{-1}(0)\text{ has compact closure in }X}.
\end{aligned}
\end{equation}
For $f\in C(X)$, define $\norm{f}_{\sup}:=\sup_{x\in X}\abs{f(x)}$.
\item Let $X$ be a topological space equipped with a Borel measure $\mu$, and let $f$ be Borel measurable. We denote by $\supp_{\mu}[f]$ the support of the measure $\abs{f}\dif\mu$.
\item Let $(X,d)$ be a metric space. For $(x,r)\in X\times[0,\infty]$, define
\begin{equation}\label{e.intro-ball}
B(x,r):=\Sett{y\in X}{d(x,y)<r}.
\end{equation}
In particular, $B(x,0)=\emptyset$ and $B(x,\infty)=X$.
\item The letters $C,c$, and their variants denote positive constants whose values are inessential and may change from line to line.
\item For $A,B\in\bR\cup\{-\infty,\infty\}$, set $A\wedge B:=\min(A,B)$ and $A\vee B:=\max(A,B)$.
\item For non-negative quantities $f$ and $g$, we write $f\lesssim g$ if there exists a constant $C\geq1$, depending only on inessential parameters, such that $f\leq Cg$.
\end{enumerate}
\end{notation}

\section{Preliminaries}\label{s.prelim}
In this section, we review the definitions of local trees, their geometric properties, and their canonical
Dirichlet forms.
\subsection{Local trees and their orientation}
We first recall the definitions of a local tree, its skeleton, and its length measure, following \cite[Definitions~2.1, 2.2 and Section~2.1]{BCY25}.
\begin{definition}\label{d.tree1}
Let $(\ambient,\metric)$ be a metric space.
\begin{enumerate}[label=\textup{({\arabic*})},align=right,leftmargin=*,topsep=5pt,parsep=0pt,itemsep=2pt]
    	\item The space $(\ambient,\metric)$ is called a \emph{real tree} if it satisfies the following two properties:
\begin{enumerate}[label=\textup{(\roman*)},align=right,leftmargin=*,topsep=5pt,parsep=0pt,itemsep=2pt]
        \item for every $u,v\in\ambient$, there exists a unique isometric embedding $ \phi_{u,v}\colon [0,\metric(u,v)]\rightarrow \ambient$
        such that $\phi_{u,v}(0)=u$ and $\phi_{u,v}(\metric(u,v))=v$;
        \item for every injective continuous map
        $\kappa\colon[0,1]\to\ambient$, one has
        \[ \kappa([0,1])=\phi_{\kappa(0),\kappa(1)} \bigl([0,\metric(\kappa(0),\kappa(1))]\bigr).\]
    \end{enumerate}
    	\item The space $(\ambient,\metric)$ is called a \emph{local tree} if, for every $x\in\ambient$, there exists $\iota(x)\in(0,\infty]$ such that
$\bigl(B(x,\iota(x)),\restr{\metric}{B(x,\iota(x))\times B(x,\iota(x))}\bigr)$
    is a real tree.

    \item The space $(\ambient,\metric)$ is called a \emph{uniform local tree} if there exists
    $\iota\in(0,\diam(\ambient,\metric)]$ such that, for every $x\in\ambient$, the metric space $\bigl(B(x,\iota),\restr{\metric}{B(x,\iota)\times B(x,\iota)}\bigr)$ is a real tree.
\end{enumerate}
\end{definition}

\begin{definition}\label{d.tree2}
Let $(\ambient,\metric)$ be a local tree.\begin{enumerate}[label=\textup{({\arabic*})},align=right,leftmargin=*,topsep=5pt,parsep=0pt,itemsep=2pt]
	\item For each $x\in\ambient$, fix $\iota(x)\in(0,\infty]$ such that $B(x,\iota(x))$ is a real tree. Whenever $\metric(x,y)<\iota(x)$, write \[[x,y]:=\phi_{x,y}\bigl([0,\metric(x,y)]\bigr),\text{ and }(x,y):=\phi_{x,y}\bigl((0,\metric(x,y))\bigr).\]
	\item The \emph{skeleton} of $\ambient$ is defined as $\skeleton:=\bigcup_{x\in\ambient}\bigcup_{y\in B(x,\iota(x))}(x,y)$.
	\item There exists a unique measure $\medm$ on the $\sigma$-field generated by these open arcs such that \[\text{ $\medm((x,y))=\metric(x,y)$ whenever $x,y\in\skeleton$ and $\metric(x,y)\in(0,\iota(x))$.}\] The measure $\medm$ is called the \emph{length measure} on $\skeleton$. 
\end{enumerate}

\end{definition}
\begin{remark}
\begin{enumerate}[label=\textup{({\arabic*})},align=right,leftmargin=*,topsep=5pt,parsep=0pt,itemsep=2pt]
	\item  If $(\ambient,\metric)$ is a separable uniform local tree, and $\iota(x)=\iota$ for every $x\in\ambient$, then $\skeleton=\bigcup_{\substack{x,y\in D\\ \metric(x,y)<\iota}}(x,y)$ for any countable dense subset $D\subset\skeleton$.
	\item The length measure $\medm$ is supported on $\skeleton$, and need not be a Radon measure on $(\ambient,\metric)$.
\end{enumerate}

\end{remark}

We will define an orientation of a local tree and the weak gradient of an absolutely continuous function using \emph{local arc charts}.
\begin{definition}\label{d.tree3}
	Let $(\ambient,\metric)$ be a connected and separable local tree with skeleton $\skeleton$ and length measure $\medm$.
	\begin{enumerate}[label=\textup{({\arabic*})},align=right,leftmargin=*,topsep=5pt,parsep=0pt,itemsep=2pt]
	\item We say that a map $\gamma:[0,\ell]\to\ambient$ is a \emph{local arc chart} of $(\ambient,\metric)$ if $\ell\in(0,\infty)$, $\gamma$ is an isometric embedding such that $\ell=\Length(\gamma)$, and there exists $x\in \ambient$ such that $\gamma([0,\ell])\subset B(x,\iota(x))$.
	\item We define the collection of \emph{absolutely continuous} functions by \begin{equation}
	\abscon(\ambient,\metric):=\Biggl\{u\in C(\ambient)\Biggm|\begin{minipage}{270pt}
{for every {local arc chart} $\gamma:[0,\ell]\to\ambient$, the function ${(u\circ \gamma)}:[0,\ell]\to\bR$ is absolutely continuous on $[0,\ell]$}
\end{minipage}
\Biggr\}.
\end{equation}
\item We say that $\mathfrak{O}=(\gamma_{n})_{n\in\bN}$ is an \emph{orientation} of $(\ambient,\metric)$ if, for each $n\in\bN$, $\gamma_{n}:[0,\ell_{n}]\to\ambient$ is a local arc chart and $\skeleton=\bigcup_{n\in\bN}\gamma_{n}((0,\ell_{n}))$.
\item For every {orientation} $\mathfrak{O}=(\gamma_{n})_{n\in\bN}$ of $(\ambient,\metric)$ and every $u\in\abscon(\ambient,\metric)$, we define the \emph{weak gradient of $u$ with respect to $\mathfrak{O}$} by \begin{equation}\label{e.dwgrad}
	\wgrad_{\mathfrak{O}}u(z):=\sum_{n\in\bN}\one_{A_{n}}(z)(u\circ \gamma_{n})^{\prime}(\gamma_{n}^{-1}(z)),\ \medm\text{-a.e. }z\in\skeleton,
\end{equation}
where \begin{equation}\label{e.dfan}
	A_{1}:=\gamma_{1}((0,\ell_{1})),\ A_{n}:=\gamma_{n}((0,\ell_{n}))\setminus\left(\bigcup_{j=1}^{n-1}\gamma_{j}((0,\ell_{j}))\right).
\end{equation}
\end{enumerate}
\end{definition}
\begin{remark}
	\begin{enumerate}[label=\textup{({\arabic*})},align=right,leftmargin=*,topsep=5pt,parsep=0pt,itemsep=2pt]
	\item Clearly, the collection $\abscon(\ambient,\metric)$ does not depend on the choice of orientation of the local tree $(\ambient,\metric)$.
	\item We emphasize that the ordering of the local arc charts in an orientation $\mathfrak{O}=(\gamma_{n})_{n\in\bN}$ matters in the definition of weak gradient \eqref{e.dwgrad}.
	\item For every {orientation} $\mathfrak{O}=(\gamma_{n})_{n\in\bN}$ of $(\ambient,\metric)$ and every $u\in\abscon(\ambient,\metric)$, the weak derivative $\wgrad_{\mathfrak{O}}u$ of $u$ with respect to $\mathfrak{O}$ is $\medm$-measurable (that is, measurable with respect to the $\sigma$-algebra generated by all open arcs) and is uniquely defined up to a $\medm$-null set.
	 \end{enumerate}
\end{remark}
The following proposition describes how the weak gradient changes when the orientation is changed.
\begin{proposition}\label{p.swap}
Let $\mathfrak O=(\gamma_n)_{n\in\mathbb N}$ and $\mathfrak O'=(\eta_m)_{m\in\mathbb N}$ be two orientations of $(\ambient,\metric)$. There exists a $\medm$-measurable function $\varepsilon_{\mathfrak O,\mathfrak O'}:\skeleton\rightarrow\{-1,1\}$, depending only on $\mathfrak O$ and $\mathfrak O'$, such that 
\begin{equation}\label{e.swap1}
\text{ for every $u\in\abscon(\ambient,\metric)$, }\wgrad_{\mathfrak O'}u=\varepsilon_{\mathfrak O,\mathfrak O'}\wgrad_{\mathfrak O}u
\ \medm\text{-a.e. on }\skeleton, \text{ and }|\wgrad_{\mathfrak O'}u|=|\wgrad_{\mathfrak O}u|\ \medm\text{-a.e. on }\skeleton.
\end{equation}
Consequently, for every $u,v\in\abscon(\ambient,\metric)$,
\begin{equation}\label{e.swap2}
\wgrad_{\mathfrak O'}u\,\wgrad_{\mathfrak O'}v
=\wgrad_{\mathfrak O}u\,\wgrad_{\mathfrak O}v
\ \medm\text{-a.e. on }\skeleton.
\end{equation}
\end{proposition}
\begin{proof}
	For $n,m\in\bN$, define $I_{n}:=\gamma_{n}((0,\Length(\gamma_{n})))$, $J_{m}:=\eta_{m}((0,\Length(\eta_{m})))$, $A_{1}:=I_{1}$, $A_{n}:=I_{n}\setminus \bigcup_{j=1}^{n-1}I_{j}$ for $n\geq2$, $B_{1}:=J_{1}$, and $B_{m}:=J_{m}\setminus \bigcup_{k=1}^{m-1}J_{k}$ for $m\geq2$. Then the sets $\{A_{n}\cap B_{m}\}_{(n,m)\in\bN^{2}}$ are pairwise disjoint, and $\bigcup_{(n,m)\in\bN^{2}}A_{n}\cap B_{m}=\skeleton$. Let \[E_{n,m}:=\gamma_n^{-1}(A_{n}\cap B_{m})\subset(0,\Length(\gamma_{n})),\]which is a Borel measurable subset of $\bR$. Observe that $(A_{n}\cap B_{m})\subset I_n\cap J_m$. Hence the transition map $\tau_{n,m}:=\eta_m^{-1}\circ\gamma_n: E_{n,m}\rightarrow(0,\Length(\eta_{m}))$ is well-defined. For $s,t\in E_{n,m}$, both parametrizations are isometries, so
\begin{equation}
|\tau_{n,m}(s)-\tau_{n,m}(t)|=
\metric(\eta_m(\tau_{n,m}(s)),\eta_m(\tau_{n,m}(t)))=
\metric(\gamma_n(s),\gamma_n(t))=|s-t|.
\end{equation}
Thus $\tau_{n,m}$ is an isometry on the subset $E_{n,m}$ of the real line.

If $E_{n,m}$ has Lebesgue measure zero, define
$\varepsilon_{n,m}:=1$. If $E_{n,m}$ has positive Lebesgue measure, then it contains two distinct points, so Lemma~\ref{l.isom} gives constants
$\varepsilon_{n,m}\in\{-1,1\}$ and $c_{n,m}\in\mathbb R$ such that $\tau_{n,m}(t)=c_{n,m}+\varepsilon_{n,m}t$ for all $t\in E_{n,m}$. Define
\begin{equation}
\varepsilon_{\mathfrak O,\mathfrak O'}(z)
:=\sum_{(n,m)\in\bN^{2}}
\varepsilon_{n,m}\one_{A_{n}\cap B_{m}}(z),\ z\in\skeleton.
\end{equation}
Then the function $\varepsilon_{\mathfrak O,\mathfrak O'}:\skeleton\to\{1,-1\}$ is well-defined because the sets $A_{n}\cap B_{m}$ form a partition of $\skeleton$. It is $\medm$-measurable because it is constant on each member of a countable Borel partition. Fix $u\in\abscon(\ambient,\metric)$ and $(n,m)\in\bN^{2}$. Put $f:=u\circ\gamma_n$ and $g:=u\circ\eta_m$. Both $f$ and $g$ are absolutely continuous on their respective parameter intervals. 

Assume first that $E_{n,m}$ has positive Lebesgue measure. Define the open interval
\begin{equation}
D_{n,m}
:=
(0,\Length(\gamma_{n}))\cap
\Sett{t\in\mathbb R}{c_{n,m}+\varepsilon_{n,m}t\in(0,\Length(\eta_{m}))}.
\end{equation}
It contains $E_{n,m}$. The function $h(t):=f(t)-g(c_{n,m}+\varepsilon_{n,m}t)$, $t\in D_{n,m}$, is absolutely continuous. For $t\in E_{n,m}$, we have $\gamma_n(t)=\eta_m(\tau_{n,m}(t))=\eta_m(c_{n,m}+\varepsilon_{n,m}t)$, and hence $h(t)=0$. By \cite[Exercise~8.10]{Bre11}, it follows that $f'(t)=\varepsilon_{n,m}
 g'(c_{n,m}+\varepsilon_{n,m}t)$ for $\Leb$-a.e. $t\in E_{n,m}$. Since $\varepsilon_{n,m}^2=1$, this is equivalent to $g'(\tau_{n,m}(t))=\varepsilon_{n,m}f'(t)$ for $\Leb$-a.e. $t\in E_{n,m}$. For $z=\gamma_n(t)\in A_{n}\cap B_{m}$, the two weak gradients are given by
\begin{equation}
\wgrad_{\mathfrak O}u(z)=f'(t),\ 
\wgrad_{\mathfrak O'}u(z)=g'(\eta_m^{-1}(z))=g'(\tau_{n,m}(t)).
\end{equation}
Since the length measure $\medm$ on $I_n$ is the push-forward of Lebesgue measure under $\gamma_n$, the preceding derivative identity yields $\wgrad_{\mathfrak O'}u=\varepsilon_{n,m}\wgrad_{\mathfrak O}u$ $\medm$-a.e. on $A_{n}\cap B_{m}$.

 If $\sL^{1}(E_{n,m})=0$, then $A_{n}\cap B_{m}=\gamma_n(E_{n,m})$ is $\medm$-null, so the same assertion is vacuous on $A_{n}\cap B_{m}$. 
 
 Taking the countable union over $(n,m)\in\mathbb N^2$ gives $\wgrad_{\mathfrak O'}u=\varepsilon_{\mathfrak O,\mathfrak O'}\wgrad_{\mathfrak O}u$ $\medm$-a.e. on $\skeleton$. Since $|\varepsilon_{\mathfrak O,\mathfrak O'}|=1$, taking absolute values gives $|\wgrad_{\mathfrak O'}u|=|\wgrad_{\mathfrak O}u|$ $\medm$-a.e. on $\skeleton$. This proves \eqref{e.swap1}. Applying the same identity to $u$ and $v$ gives \eqref{e.swap2}.
\end{proof}

The following proposition is a version of the fundamental theorem of calculus for local trees.
\begin{proposition}\label{p.arc}
Let $\mathfrak O$ be an orientation of $(\ambient,\metric)$. Let $z\in\ambient$, $\ell\in(0,\infty)$, and let $\alpha:[0,\ell]\rightarrow B(z,\iota(z))$ be an isometric embedding. Then there exists a measurable function $\sigma_{\alpha,\mathfrak O}:\alpha((0,\ell))\rightarrow\{-1,1\}$ depending only on $\alpha$ and $\mathfrak O$ such that
\begin{equation}
u(\alpha(\ell))-u(\alpha(0))
=
\int_{\alpha((0,\ell))}
\sigma_{\alpha,\mathfrak O}(z)\,
\wgrad_{\mathfrak O}u(z)\dif\medm(z),\ \text{for every $u\in\abscon(\ambient,\metric)$}.
\end{equation}
\end{proposition}

\begin{proof}
Let $\mathfrak O=(\gamma_n)_{n\in\bN}$. For each $n\in\bN$, define $\{A_{n}\}_{n\in\bN}$ by \eqref{e.dfan}, and define $E_n:=\alpha^{-1}\bigl(A_n\cap\alpha((0,\ell))\bigr)$. Then $(0,\ell)=\bigcup_{n\in\bN}E_{n}$. On $E_n$, the map $\tau_n:=\gamma_n^{-1}\circ\alpha$ is an isometry between subsets of $\bR$. Hence, if $E_n$ contains at least two points, there exist $c_n\in\bR$ and $\varepsilon_n\in\{-1,1\}$ such that $\tau_n(t)=c_n+\varepsilon_n t$ for all $t\in E_n$. When $E_n$ contains at most one point, set $\varepsilon_n:=1$. Define
\begin{equation}\label{e.arc1}
	\sigma_{\alpha,\mathfrak O}(x):=\sum_{n\in\bN}\varepsilon_n\one_{E_n}(\alpha^{-1}(x)),
\  x\in \alpha((0,\ell)).
\end{equation}
This function $\sigma_{\alpha,\mathfrak O}$ is measurable and depends only on $\alpha$ and $\mathfrak O$. Fix $u\in\abscon(\ambient,\metric)$ and write $f:=u\circ\alpha$ and $f_n:=u\circ\gamma_n$. Then $f$ and $f_n$ are absolutely continuous on their respective parameter intervals. If $E_{n}$ contains at least two points, then we have $f(t)=f_n(c_n+\varepsilon_n t)$ for all $t\in E_n$. Since two absolutely continuous functions that agree on a measurable set have equal derivatives almost everywhere on that set by \cite[Exercise~8.10]{Bre11}, it follows that $f^{\prime}(t)=\varepsilon_n f_n'(\gamma_n^{-1}(\alpha(t)))=\varepsilon_n\wgrad_{\mathfrak O}u(\alpha(t))$ for $\Leb$-a.e. $t\in E_n$. Therefore,
\[
u(\alpha(\ell))-u(\alpha(0))=
\int_0^\ell f^{\prime}(t)\dif t
=
\sum_{n\in\bN}\int_{E_n}
\varepsilon_n
\wgrad_{\mathfrak O}u(\alpha(t))\dif t\overset{\eqref{e.arc1}}{=}\int_{\alpha((0,\ell))}
\sigma_{\alpha,\mathfrak O}(z)\,
\wgrad_{\mathfrak O}u(z)\dif\medm(z),
\]
since $\alpha_{\#}\mathcal L^1=\medm$ on $\alpha((0,\ell))$. This proves the formula.
\end{proof}

We present some basic properties of the weak gradient.
\begin{lemma}\label{l.wgrad}
	Let $(\ambient,\metric)$ be a connected and separable local tree with skeleton $\skeleton$ and length measure $\medm$. Fix an orientation $\mathfrak O$.
	\begin{enumerate}[label=\textup{({\arabic*})},align=right,leftmargin=*,topsep=5pt,parsep=0pt,itemsep=2pt]
\item\label{it.line} If $u,v\in\abscon(\ambient,\metric)$ and $a,b\in\mathbb R$, then $au+bv\in\abscon(\ambient,\metric)$ and
\begin{equation}\label{e.line}
\wgrad_{\mathfrak O}(au+bv)
=a\wgrad_{\mathfrak O}u+b\wgrad_{\mathfrak O}v
\ \ \medm\text{-a.e. on } \skeleton.
\end{equation}
\item\label{it.chain} \textup{(Chain rule)} If $u\in\abscon(\ambient,\metric)$ and $F:\mathbb R\to\mathbb R$ is Lipschitz with $\operatorname{Lip}(F)<\infty$, then $F\circ u\in\abscon(\ambient,\metric)$ and
\begin{equation}\label{e.chain}
\wgrad_{\mathfrak O}(F\circ u)
=F^{\prime}(u)\wgrad_{\mathfrak O}u
\ \ \medm\text{-a.e. on } \skeleton.
\end{equation}
\item\label{it.leb} \textup{(Leibniz rule)} If $u,v\in\abscon(\ambient,\metric)$, then $uv\in\abscon(\ambient,\metric)$ and
\begin{equation}\label{e.leb}
\wgrad_{\mathfrak O}(uv)
=u\wgrad_{\mathfrak O}v+v\wgrad_{\mathfrak O}u
\ \ \medm\text{-a.e. on } \skeleton.
\end{equation}
\item\label{it.clos} \textup{(Closedness)} Let $m$ be a Radon measure on $(\ambient,\metric)$ with full support. Suppose that $u_n\in\abscon(\ambient,\metric)\cap L^2(\ambient,\meas)$,
that $\wgrad_{\mathfrak O}u_n\in L^2(\skeleton,\medm)$, and that
\begin{equation}\label{e.clos1}
u_n\to u
\quad\text{in }L^2(\ambient,\meas),
\ \text{ and }\
 \wgrad_{\mathfrak O}u_n\to g
\quad\text{in }L^2(\skeleton,\medm).
\end{equation}
Then there exists $\widetilde u\in\abscon(\ambient,\metric)$ such that $\wt{u}=u$ $\meas$-a.e. on $\ambient$, and $\wgrad_{\mathfrak O}\widetilde u=g$ $\medm$-a.e. on $\skeleton$.
\end{enumerate}
\end{lemma}
\begin{proof}
Write $\mathfrak O=(\gamma_n)_{n\in\bN}$ and let
$\gamma_n:(0,\Length(\gamma_{n}))\to \ambient$ be the corresponding isometric parametrization. \begin{enumerate}[label=\textup{({\arabic*})},align=right,leftmargin=*,topsep=5pt,parsep=0pt,itemsep=2pt]
	\item[\ref{it.line}] For each $n\in\bN$, set $u_n:=u\circ\gamma_n$ and $v_n:=v\circ\gamma_n$. By definition, $u_n$ and $v_n$ are absolutely continuous on $(0,\Length(\gamma_{n}))$. For $a,b\in\bR$, the function $au_n+bv_n$ is absolutely continuous and $(au_n+bv_n)^{\prime}=au_n^{\prime}+bv_n^{\prime}$ $\Leb$-a.e. on $(0,\Length(\gamma_{n}))$. Hence $au+bv\in\abscon(\ambient,\metric)$, and \eqref{e.line} follows from the definition of the weak gradient.
	\item[\ref{it.chain}] Now let $F:\bR\to\bR$ be Lipschitz. Then $F\circ u_n$ is locally absolutely
continuous. Extend $F^{\prime}$ arbitrarily, for instance by zero, on the Lebesgue-null
set on which $F$ is not differentiable. The one-dimensional chain rule for
absolutely continuous functions gives $(F\circ u_n)^{\prime}=F^{\prime}(u_n)u_n^{\prime}$ $\Leb$-a.e. on $(0,\Length(\gamma_{n}))$. Thus $F\circ u\in\abscon(\ambient,\metric)$ and \eqref{e.chain} holds.
\item[\ref{it.leb}] Since $u_n$ and $v_n$ are bounded on every compact subinterval, their
product is locally absolutely continuous, and the one-dimensional Leibniz rule gives $(u_nv_n)^{\prime}=u_nv_n^{\prime}+v_nu_n^{\prime}$ $\Leb$-a.e. on $(0,\Length(\gamma_{n}))$. Consequently, $uv\in\abscon(\ambient,\metric)$ and \eqref{e.leb} holds.
\item[\ref{it.clos}]
After passing to a subsequence, we may assume that $u_n\to u$
$\meas$-a.e. on $\ambient$. Let $E:=\Sett{x\in \ambient}{\lim_{n\to\infty}u_{n}(x)=u(x)}$, so $\meas(\ambient\setminus E)=0$. Choose a countable cover $(U_j)_{j\in\bN}$ of $\ambient$ by open balls, such that $(U_j,\restr{\metric}{U_j\times U_{j}})$ is a tree, and choose $x_j\in E\cap U_j$. Fix an arbitrary $j\in\bN$ and let $\sigma_{x_j,y,\mathfrak O}$ be the sign function supplied by Proposition~\ref{p.arc} for the arc from $x_j$ to $y$, when $y\neq x_j$. For $y=x_j$, set $\sigma_{x_j,y,\mathfrak O}=1$. Define\begin{equation}
	v_j(y):=u(x_j)+
\int_{[x_j,y]}
\sigma_{x_j,y,\mathfrak O}(z)\,g(z)\dif\medm(z),\ \text{ for all }y\in U_j.\label{e.clos2}
\end{equation}

Since every arc $[x_j,y]\subset U_{j}$ has finite length and $g\in L^2(\skeleton,\medm)$,
this integral is finite. Moreover, for $y,z\in U_j$, by H\"{o}lder's inequality,
\[
|v_j(y)-v_j(z)|
\le \metric(y,z)^{\frac{1}{2}}
\left(\int_{[y,z]}|g|^2\dif\medm\right)^{\frac{1}{2}},
\]
so $v_j:U_{j}\to\bR$ is absolutely continuous. For $y\in E\cap U_j$, Proposition \ref{p.arc} gives that
\begin{equation}
	u_n(y)-u_n(x_j)
=
\int_{[x_j,y]}
\sigma_{x_j,y,\mathfrak O}\; \wgrad_{\mathfrak O}u_n\dif\medm.\label{e.clos2+}
\end{equation}
Since $u_n(y)\to u(y)$, $u_n(x_j)\to u(x_j)$, and
\begin{equation}
	\lim_{n\to\infty}\left|
\int_{[x_j,y]}
\sigma_{x_j,y,\mathfrak O}(\wgrad_{\mathfrak O}u_n-g)\dif\medm
\right|
\leq \lim_{n\to\infty} \metric(x_j,y)^{\frac{1}{2}}\norm{\wgrad_{\mathfrak O}u_n-g}_{L^2(\skeleton,\medm)}\overset{\eqref{e.clos1}}{=}0,\label{e.clos3}
\end{equation}
we obtain $v_j(y)=u(y)$ by \eqref{e.clos2}. Hence $v_j=u$ on $U_j\cap E$.

If $U_j\cap U_k\neq\emptyset$, then $v_j-v_k$ is continuous and vanishes
$\meas$-a.e. on $U_j\cap U_k$. Since $\meas$ has full support, it follows
that $v_j=v_k$ everywhere on $U_j\cap U_k$. Thus the functions $v_j$
patch together to form a continuous function
$\widetilde u\in\abscon(\ambient,\metric)$ satisfying
$\widetilde u=u$ $\meas$-a.e. Finally, the identity \eqref{e.clos2} shows that $\wgrad_{\mathfrak O}\widetilde u=g$ $\medm$-a.e. on $\skeleton$. Therefore, $u$ admits a locally absolutely continuous representative and
$g=\wgrad_{\mathfrak O}u\in L^2(\skeleton,\medm)$.
\qedhere
\end{enumerate}
\end{proof}

\subsection{Large-scale geometry of a uniform local tree}\label{s.graph}

Let $(\ambient,\metric)$ be a uniform local tree. Fix $\epsilon\in(0,(20)^{-1}\iota)$ and let $\vertex\subset\ambient$ be a \emph{maximal $\epsilon$-separated set} of $(\ambient,\metric)$; that is, $\metric(x,y)\geq\epsilon$ for all $x,y\in\vertex$ with $x\neq y$, and $\ambient=\bigcup_{v\in\vertex}B(v,\epsilon)$. The existence of such a maximal $\epsilon$-separated set follows from Zorn's lemma. We define the \emph{edge set} by \begin{equation}\label{e.edges}
	\edge:=\Sett{(x,y)\in \vertex\times \vertex}{d(x,y)\in(0,3\epsilon]}.
\end{equation}
Then $\graph:=(\vertex,\edge)$ is a graph. Let $\metric_{\graph}:\vertex\times \vertex\to[0,\infty]$ denote the graph distance associated with $\edge$, and let
\begin{equation}
 B_{\graph}(v,R):=\Sett{w\in\vertex}{\metric_{\graph}(v,w)<R},\ \text{for all }(v,R)\in\vertex\times(0,\infty).
\end{equation}

\begin{definition}
	Let $\volume:[0,\infty)\to [0,\infty)$ be an increasing function. Let $\meas$ be a Radon measure on $(\ambient,\metric)$. We say that $(\ambient,\metric,\meas)$ satisfies the \hypertarget{UVG}{\emph{uniform volume growth}} condition $\UVG$, or the measure $m$ satisfies $\UVG$, if there exists $C_{\volume}\in[1,\infty)$ such that \begin{equation}\label{e.UVG}
	0<C_{\volume}^{-1}\volume(r)\leq \meas(B(x,r))\leq C_{\volume}\volume(r)<\infty,\ \text{ for all }(x,r)\in \ambient\times(0,\diam(\ambient,\metric)),
\end{equation}
and
\begin{equation}\label{e.VD}
	\volume(2r)\leq C_{\volume}\volume(r),\ \text{ for all }r\in(0,\infty).
\end{equation}
\end{definition}
\begin{remark}
	\begin{enumerate}[label=\textup{({\arabic*})}, align=right, leftmargin=*, topsep=5pt, parsep=0pt, itemsep=2pt]
  \item If $\meas$ satisfies $\UVG$, then $\meas$ has full support; if, in addition, $(\ambient,\metric)$ is unbounded, then $\meas(\ambient)=\infty$. In fact, unboundedness provides infinitely many pairwise disjoint balls with radii $1$, each of measure at least $C_{\volume}^{-1}\volume(1)>0$. 
  \item By \eqref{e.VD}, we know that there exist $C\in(0,\infty)$ and $d_{2}\in[1,\infty)$ such that \begin{equation}\label{e.VD1}
	\frac{V(R)}{V(r)}\leq C\left(\frac{R}{r}\right)^{d_{2}},\ \text{ for all }0<r\leq R<\infty.
\end{equation}
\end{enumerate}
\end{remark}

\begin{lemma}\label{l.covering}
Let $(\ambient,\metric)$ be a $\sigma$-compact unbounded uniform local tree, and let $\meas$ be a Borel measure on $(\ambient,\metric)$ satisfying $\UVG$. Then $\vertex$ is locally finite and countable. Moreover, there exists $C\in(1,\infty)$ such that
\begin{equation}\label{e.pack1}
\#(\vertex\cap B(x,r))
 \leq C\frac{\volume(r)}{\volume(\epsilon)},\ \text{ for all }(x,r)\in\ambient\times [\epsilon,\diam(\ambient,\metric)),
\end{equation}
and
\begin{equation}\label{e.pack2}
 \#(\vertex\cap B(x,r+\epsilon))
 \geq C^{-1}\frac{\volume(r)}{\volume(\epsilon)},\ \text{ for all }(x,r)\in\ambient\times (0,\diam(\ambient,\metric)).
\end{equation}
\end{lemma}

\begin{proof}
For $(x,r)\in\ambient\times [\epsilon,\diam(\ambient,\metric))$ and for every finite set $F\subset\vertex\cap B(x,r)$, since $r\geq\epsilon$ and the balls $B(v,\epsilon/3)$, $v\in\vertex$, are pairwise disjoint, we have
\begin{equation}
 C_{\volume}^{-1} \volume(\epsilon/3)\cdot \#F\leq\sum_{v\in F}\meas(B(v,\epsilon/3))
 \leq \meas(B(x,r+\epsilon/3))\leq \meas(B(x,4r/3))\overset{\eqref{e.UVG}}{\leq} C_{\volume} \volume(4r/3).
\end{equation}
Combining this with \eqref{e.VD} and then taking the supremum over all finite sets $F$ proves \eqref{e.pack1}.
In particular, $\vertex$ is locally finite. Since $\ambient$ is $\sigma$-compact and $\vertex$ is locally finite, $\vertex$ is countable.

Fix $(x,r)\in\ambient\times (0,\diam(\ambient,\metric))$. By maximality, for every $y\in B(x,r)$, there exists $v\in\vertex$ such that $\metric(y,v)<\epsilon$. Such a point satisfies $v\in B(x,r+\epsilon)$, and therefore
$B(x,r)\subset\bigcup_{v\in\vertex\cap B(x,r+\epsilon)}B(v,\epsilon)$. Using \eqref{e.UVG} and subadditivity, we obtain $\volume(r)
 \leq C_{\volume}\meas(B(x,r))
 \leq C\volume(\epsilon)
 \#(\vertex\cap B(x,r+\epsilon))$, which proves \eqref{e.pack2}.
\end{proof}

\begin{proposition}\label{p.count}
	Let $(\ambient,\metric)$ be a $\sigma$-compact geodesic unbounded uniform local tree, and let $\meas$ be a Borel measure on $(\ambient,\metric)$ satisfying $\UVG$. The graph $\graph=(\vertex,\edge)$ is connected and has uniformly bounded degree. Moreover, \begin{equation}
		  \frac{\epsilon}{2}\metric_{\graph}(v,w)
 \leq \metric(v,w)
 \leq 3\epsilon\metric_{\graph}(v,w),\ \text{ for every }v,w\in\vertex,\label{e.quaiso}
	\end{equation}
	and there exists $C\in(1,\infty)$ such that \begin{equation}\label{e.count}
		C^{-1}\frac{\volume(\epsilon R)}{\volume(\epsilon)}
 \leq \#B_{\graph}(v,R)
 \leq C\frac{\volume(\epsilon R)}{\volume(\epsilon)},\ \text{ for every } (v,R)\in\vertex\times\big[1,\frac{1}{3\epsilon}\diam(\ambient,\metric)\big).
	\end{equation}
\end{proposition}
\begin{proof}
Let $v,w\in\vertex$ with $v\neq w$, set $L:=\metric(v,w)\geq\epsilon$, and let $\gamma:[0,L]\to\ambient$ be a unit-speed geodesic from $v=\gamma(0)$ to $w=\gamma(L)$. Set $N=\left\lceil{L}{\epsilon}^{-1}\right\rceil$ and $x_j=\gamma\left({jL}{N}^{-1}\right)$ for $j\in\{0,\ldots, N\}$. For $j\in\{1,\ldots, N-1\}$, choose $v_j\in\vertex$ such that $\metric(v_j,x_j)<\epsilon$, and set $v_0=v$ and $v_N=w$. Then $\metric(v_j,v_{j+1}) <\epsilon+{L}{N}^{-1}+\epsilon\leq3\epsilon$. So $(v_{j},v_{j+1})\in\edge$ whenever $v_{j}\neq v_{j+1}$. After deleting repeated elements, $(v_0,\ldots,v_N)$ is a path in $\graph$. Hence
\begin{equation}
 \metric_{\graph}(v,w)\leq N\leq\frac{L}{\epsilon}+1\leq\frac{2L}{\epsilon}=\frac{2}{\epsilon}d(v,w).
\end{equation}
Conversely, by the triangle inequality, every path in $\graph$ of graph length $n$ has length in $\metric$ at most $3\epsilon n$. This proves connectivity and \eqref{e.quaiso}.

By Lemma~\ref{l.covering}, for any $v\in\vertex$, we have $\deg_{\graph}(v)+1
 \leq\#(\vertex\cap B(v,4\epsilon)) \leq C$, so the degree of $\graph$ is uniformly bounded. The triangle inequality and \eqref{e.quaiso} give
\begin{equation}
 \vertex\cap B(v,{\epsilon R}/{2})
 \subset B_{\graph}(v,R)
 \subset\vertex\cap B(v,3\epsilon R).
\end{equation}
The upper bound in \eqref{e.count} follows from \eqref{e.pack1}. If $R\geq4$, apply \eqref{e.pack2} with $r=\epsilon\left(\frac{R}{2}-1\right)$ and use \eqref{e.VD} to compare $\volume(r)$ with $\volume(\epsilon R)$. This gives the lower bound. If $1\leq R<4$, then $\#B_{\graph}(v,R)\geq1$ and $\volume(\epsilon R)\leq C\volume(\epsilon)$ by \eqref{e.VD}.
\end{proof}

Enumerate $\vertex=\set{v_1,v_2,\ldots}$. Since $\vertex$ is $\epsilon$-separated, it is closed. Since it is locally finite, the distance (with respect to $\metric$) from any $x\in\ambient$ to $\vertex$ is attained. Define \begin{align}
		K_{v_1}&:=\Sett{x\in\ambient}{\metric(x,v_1)=\dist(x,\vertex)}\\
		K_{v_n}&:=\Sett{x\in\ambient}{\metric(x,v_n)=\dist(x,\vertex)}\setminus\bigcup_{j<n}K_{v_{j}},\ \text{ for }n\in\bN\cap[2,\#\vertex+1).
\end{align} 
Each $K_v$ is Borel because the function $x\mapsto\dist(x,\vertex)$ is $1$-Lipschitz. Therefore, $\ambient=\bigsqcup_{v\in\vertex}K_v$. Set $\skeleton_v:=\skeleton\cap K_v$ for $v\in\vertex$. Then $(\skeleton_v)_{v\in\vertex}$ is a partition of $\skeleton$. For any $A\subset\vertex$, define $K_{A}:=\bigcup_{v\in A}K_{v}$.

\begin{lemma}\label{l.sanwich}
	Let $(\ambient,\metric)$ be a $\sigma$-compact geodesic unbounded uniform local tree, and let $\meas$ be a Borel measure on $(\ambient,\metric)$ satisfying $\UVG$. Then \begin{equation}\label{e.sanwich}
		B(v,\epsilon/3)\subset K_v\subset B(v,\epsilon),\ \text{ for every }v\in\vertex.
	\end{equation}
\end{lemma}
\begin{proof}
If $x\in B(v,\epsilon/3)$ and $w\in\vertex\setminus\set{v}$, then \[\metric(x,w)
 \geq\metric(v,w)-\metric(x,v)
 \geq\frac{2\epsilon}{3}
 >\metric(x,v),\] and hence $x\in K_v$. If $x\in K_v$, maximality of $\vertex$ gives $\metric(x,v)=\dist(x,\vertex)<\epsilon$, and hence $x\in B(v,\epsilon)$.
\end{proof}
For $A\subset\vertex$ and a nonempty subset $E\subset\ambient$, define
\begin{equation}\label{e.defdG}
 \metric_{\graph}(A,E):=\inf\Sett{\metric_{\graph}(v,w)}
 {v\in A,\ w\in\vertex,\ K_w\cap E\neq\emptyset},
\end{equation}
and $\metric_{\graph}(v,E):=\metric_{\graph}(\{v\},E)$ for $v\in\vertex$.

\begin{lemma}\label{l.c-dist}
Let $(\ambient,\metric)$ be a $\sigma$-compact geodesic unbounded uniform local tree, and let $\meas$ be a Borel measure on $(\ambient,\metric)$ satisfying $\UVG$. Let $A\subset\vertex$ be a non-empty subset, $v\in\vertex$, let $E\subset\ambient$ be a non-empty Borel subset, and let $R\in(0,\infty)$. Then
\begin{equation}\label{e.d-ball}
 \frac{\epsilon}{2}\metric_{\graph}(v,E)-(R+1)\epsilon\leq \dist(B(v,R\epsilon),E),
\end{equation}
\begin{equation}\label{e.d-cell}
 \dist(K_A,E)
 \geq\frac{\epsilon}{2}\metric_{\graph}(A,E)-2\epsilon,
\end{equation}
and
\begin{equation}\label{e.d-nbhd}
 \dist(N_{R\epsilon}(K_A),E)
 \geq\frac{\epsilon}{2}\metric_{\graph}(A,E)
 -(R+2)\epsilon.
\end{equation}
Here $N_{r}(S):=\Sett{x\in\ambient}{\dist(x,S)<r}$.
\end{lemma}

\begin{proof}
For any $x\in E$, choose $w\in\vertex$ such that $x\in E\cap K_w$, and let $y\in B(v,R\epsilon)$. By the triangle inequality, \eqref{e.quaiso}, and Lemma~\ref{l.sanwich},
\begin{align}
 \metric(y,x)
 \geq\metric(v,w)-\metric(v,y)-\metric(w,x)
&\overset{\eqref{e.quaiso},\eqref{e.sanwich}}{\geq}\frac{\epsilon}{2}\metric_{\graph}(v,w)
 -(R+1)\epsilon\\
&\overset{\eqref{e.defdG}}{\geq} \frac{\epsilon}{2}\metric_{\graph}(v,E)-(R+1)\epsilon.
\end{align}
Taking the infimum over all $x\in E$ and all $y\in B(v,R\epsilon)$ gives the lower bound in \eqref{e.d-ball}. 
For any $x\in E$, there exists $w\in\vertex$ such that $x\in E\cap K_{w}$. For any $y\in K_{A}$, there exists $u\in A$ with $y\in K_u$. By the triangle inequality, \eqref{e.quaiso}, and Lemma~\ref{l.sanwich},
\begin{equation}
 \metric(x,y)
 \geq\metric(u,w)-\metric(u,y)-\metric(w,x)\overset{\eqref{e.quaiso},\eqref{e.sanwich}}{\geq}\frac{\epsilon}{2}\metric_{\graph}(u,w)-2\epsilon\overset{\eqref{e.defdG}}{\geq} \frac{\epsilon}{2}\metric_{\graph}(A,E)-2\epsilon
\end{equation}
which proves \eqref{e.d-cell} by taking the infimum over all $x\in E$ and all $y\in K_{A}$.

Finally, if $z\in N_{R\epsilon}(K_A)$, choose $y\in K_A$ such that $\metric(z,y)<R\epsilon$. Applying \eqref{e.d-cell} and the triangle inequality then gives \eqref{e.d-nbhd}.
\end{proof}

\subsection{Canonical Dirichlet form and Riesz transform}

\begin{proposition}\label{p.energy}
Let $(\ambient,\metric)$ be a proper, connected, and separable uniform local tree with skeleton $\skeleton$ and length measure $\medm$. Let $\meas$ be a Radon measure satisfying $\UVG$, and choose an orientation $\mathfrak O$. Define
\begin{equation}
	\domain:=\Biggl\{u\in L^2(\ambient,\meas)\Biggm|
	\begin{minipage}{180pt}
		there exists $\widetilde u\in\abscon(\ambient,\metric)$ such that $\wt{u}=u$ $\meas$-a.e. and $\wgrad_{\mathfrak O}\widetilde u\in L^2(\skeleton,\medm)$.
	\end{minipage}
	\Biggr\},
\end{equation}
and \begin{equation}
	\form(u,v):=\int_{\skeleton}\wgrad_{\mathfrak O}\wt{u}\,\wgrad_{\mathfrak O}\wt{v}\dif\medm,\ \text{ for all }(u,v)\in \domain\times\domain.
\end{equation}
Then $(\form,\domain)$ is a strongly local regular symmetric Dirichlet form on $L^2(\ambient,\meas)$, independent of the orientation.
\end{proposition}
\begin{proof}
By $\UVG$, the properness of $(\ambient,\metric)$ and the uniform local tree property, we may use the partition of unity and the argument in \cite[Proof of Lemma 4.1]{BCY25} to conclude that $(\form,\domain)$ is a strongly local regular Dirichlet form on $L^2(\ambient,\meas)$. The independence of $(\form,\domain)$ from the chosen orientation follows directly from \eqref{e.swap2}.
\end{proof}

By \cite[Theorem~1.3.1 and Corollary~1.3.1]{FOT11}, there exists a unique non-positive self-adjoint operator $(\gen,\Dom(\gen))$ on $L^{2}(\ambient,\meas)$ corresponding to $(\form,\domain)$. The associated symmetric Markov semigroup $\{P_{t}\}_{t\in(0,\infty)}$ is provided by \cite[Theorem~1.4.1]{FOT11}. Let $\proj$ be the spectral measure on the Borel $\sigma$-algebra of $\bR$ associated with $(-\gen,\Dom(-\gen))$, so that $\{\proj_{\lambda}:=\proj(( -\infty,\lambda])\}_{\lambda\in[0,\infty)}$ is the spectral family of projection operators associated with $(-\gen,\Dom(-\gen))$. For every Borel-measurable function $f:[0,\infty)\to\bR\cup\{-\infty,\infty\}$ such that
\[
\proj(\Sett{\lambda\in[0,\infty)}{\abs{f(\lambda)}=\infty})=0,
\]
the spectral calculus defines an (unbounded) linear operator $(f(-\gen),\Dom(f(-\gen)))$ by
\begin{equation}
	\begin{dcases}
		\Dom(f(-\gen)):=\Sett{u\in L^{2}(\ambient,\meas)}{\int_{[0,\infty)}\abs{f(\lambda)}^{2}\dif\langle \proj_{\lambda}u,u\rangle<\infty};\\
		\langle f(-\gen)u,v\rangle_{L^{2}(\ambient,\meas)} :=\int_{[0,\infty)}f(\lambda)\dif\langle \proj_{\lambda}u,v\rangle,
		\ u\in \Dom(f(-\gen)) \text{ and } v\in L^{2}(\ambient,\meas).
	\end{dcases}
\end{equation}
Here and throughout the remainder of the paper, all inner products involving the projection operators $\{\proj_{\lambda}\}_{\lambda\in[0,\infty)}$ are taken in the Hilbert space $L^{2}(\ambient,\meas)$. In particular, the Dirichlet form $(\form,\domain)$ admits the representation
\begin{equation}
	\begin{dcases}
		\domain =\Dom(( -\gen)^{\frac{1}{2}}),\\
		\form(u,v) =\langle(-\gen)^{\frac{1}{2}}u,(-\gen)^{\frac{1}{2}}v\rangle
		=\int_{[0,\infty)}\lambda\dif\langle \proj_{\lambda}u,v\rangle,
		\quad (u,v)\in\domain\times\domain,
	\end{dcases}
\end{equation}
and $\{P_{t}=\exp(t\gen)\}_{t\in (0,\infty)}$ is the unique strongly continuous symmetric Markovian semigroup on $L^{2}(\ambient,\meas)$ corresponding to $(\form,\domain)$. We say that a family of measurable functions $\{p_{t}:\ambient\times\ambient\to[0,\infty)\}_{t\in(0,\infty)}$ is a \emph{heat kernel} of $\{P_{t}\}_{t\in(0,\infty)}$, or of $(\ambient,\meas,\form,\domain)$, if, for every $f\in L^{2}(\ambient,\meas)$,
\begin{equation}
	P_{t}f(x)=\int_{\ambient}p_{t}(x,y)f(y)\dif \meas(y),
	\quad\text{for }\meas\text{-a.e. }x\in\ambient.
\end{equation}

\begin{lemma}\label{l.meab}
		Let $(\ambient,\metric)$ be a proper, connected, and separable uniform local tree with skeleton $\skeleton$ and length measure $\medm$. Choose an orientation $\mathfrak{O}$ of $(\ambient,\metric)$. Let $\meas$ be a Radon measure on $(\ambient,\metric)$ satisfying $\UVG$. Let $(\form,\domain)$ be the strongly local regular symmetric Dirichlet form on $L^2(\ambient,\meas)$ constructed in Proposition \ref{p.energy}. Then
\begin{enumerate}[label=\textup{({\arabic*})},align=right,leftmargin=*,topsep=5pt,parsep=0pt,itemsep=2pt]
	\item\label{it.meab0} The metric measure Dirichlet space $(\ambient,\metric,\meas,\form,\domain)$ admits a jointly continuous heat kernel $p=p_{t}(x,y):(0,\infty)\times\ambient\times\ambient\to[0,\infty)$.
	\item\label{it.meab1} For every $(t,x)\in(0,\infty)\times \ambient$, the function $p_{t,x}:=p_{t}(x,\cdot)$ belongs to $\Dom(-\gen)$, and $\wgrad_{\mathfrak{O}} p_{t,x}\in L^{2}(\skeleton,\medm)$. For each $t\in(0,\infty)$, there exists a jointly measurable function $h_{t}:\ambient\times\skeleton\to \bR$ such that, for every $x\in\ambient$, $h_{t}(x,y)=(\wgrad_{\mathfrak{O}} p_{t,x})(y)$ for $\medm$-a.e. $y\in \skeleton$. 
	\item\label{it.meab2} For every $(x,y)\in \ambient\times\ambient$, the function $t\mapsto p_{t}(x,y)$ is real-analytic on $(0,\infty)$. The function $(t,x,y)\mapsto \frac{\dif}{\dif t}p_{t}(x,y)$ is jointly measurable on $(0,\infty)\times \ambient\times\ambient$.
	\item\label{it.meab3} For every $(t,x)\in(0,\infty)\times \ambient$, we have
\begin{equation}\label{e.dif-gen}
		\frac{\dif}{\dif t}p_{t}(x,y)=(\gen p_{t,x})(y),
		\quad \meas\text{-a.e. } y\in\ambient.
	\end{equation}
\end{enumerate}
\end{lemma}
\begin{proof}
	\begin{enumerate}[label=\textup{({\arabic*})},align=right,leftmargin=*,topsep=5pt,parsep=0pt,itemsep=2pt]
	\item[\ref{it.meab0}] Note that $(-\gen,\Dom(-\gen))$ is a self-adjoint operator that is bounded below. Since for each $t\in(0,\infty)$, $P_{t}$ is Markovian, thus positivity preserving, and $P_{t}(L^{2}(\ambient,\meas))\subset\domain\subset\abscon(\ambient,\metric)$. By \cite[Proposition~3.3]{KLVW15}, we obtain a measurable function $p=p_{t}(x,y):(0,\infty)\times\ambient\times\ambient\to[0,\infty)$ that is a heat kernel of $\{P_{t}\}_{t\in(0,\infty)}$, is continuous separately in each $\ambient$-variable, and for all $t,s\in(0,\infty)$ and all $x,y\in\ambient$, we have
\begin{equation}\label{e.meab1}
	p_{t}(x,y)=p_{t}(y,x)
	\quad\text{and}\quad
	p_{t+s}(x,y)=\int_{\ambient}p_{t}(x,z)p_{s}(z,y)\dif\meas(z).
\end{equation}
For any $t\in(0,\infty)$ and any $x,x'\in\ambient$ such that $\metric(x,x')<\iota$, we have \begin{align}
	&\phantom{\ \leq}\norm{p_{t}(x,\cdot)-p_{t}(x',\cdot)}_{L^{2}(\ambient,\meas)}\\
	&=\sup_{\substack{f\in L^{2}(\ambient,\meas)\\ \norm{f}_{L^{2}(\ambient,\meas)}\leq1}}\abs{P_{t}f(x)-P_{t}f(x')}\leq \sup_{\substack{f\in L^{2}(\ambient,\meas)\\ \norm{f}_{L^{2}(\ambient,\meas)}\leq1}}\int_{[x,x']}\abs{\wgrad_{\mathfrak O}P_{t}f} \dif\medm\\
	&\leq \metric(x,x')^{\frac{1}{2}}\sup_{\substack{f\in L^{2}(\ambient,\meas)\\ \norm{f}_{L^{2}(\ambient,\meas)}\leq1}}\form(P_{t}f,P_{t}f)^{\frac{1}{2}}\leq \left(\frac{\metric(x,x')}{2et}\right)^{\frac{1}{2}}.\label{e.measb1+}
\end{align}
Fix $(t,x,y)\in(0,\infty)\times\ambient\times\ambient$ and let $(t_n,x_n,y_n)\to(t,x,y)$. For all sufficiently large $n$, we have $t_n>t/2$ and $\max(\metric(x_n,x),\metric(y_n,y))<\iota$. If we denote $p_{t,z}:=p_{t}(z,\cdot)$, then $p_{t+s,z}=P_{t}p_{s,z}$ for all $t,s\in(0,\infty)$. By the Cauchy--Schwarz inequality and the strong $L^{2}$-continuity of $\{P_{t}\}_{t\in(0,\infty)}$,
\begin{align}
&\phantom{\ \leq}|p_{t_n}(x_n,y_n)-p_t(x,y)|\\
&\overset{\eqref{e.meab1}}{\leq}\big|\langle p_{\frac{t}{4},x_{n}}-p_{\frac{t}{4},x},p_{t_{n}-\frac{t}{4},y_{n}}\rangle_{L^2(\ambient,\meas)}+\langle p_{\frac{t}{4},x}, p_{t_{n}-\frac{t}{4},y_{n}}-p_{t_{n}-\frac{t}{4},y}\rangle_{L^2(\ambient,\meas)}\\
&\ \qquad+ \langle p_{\frac{t}{4},x}, p_{t_{n}-\frac{t}{4},y}-p_{\frac{3}{4}t,y}\rangle_{L^2(\ambient,\meas)}\big|\\
&\leq\|p_{t/4,x_n}-p_{t/4,x}\|_{L^2(\ambient,\meas)}\|p_{t/4,y_n}\|_{L^2(\ambient,\meas)}+\|p_{t/4,x}\|_{L^2(\ambient,\meas)}\|p_{t/4,y_n}-p_{t/4,y}\|_{L^2(\ambient,\meas)}\\
&\ \qquad+\|p_{t/4,x}\|_{L^2(\ambient,\meas)}\|(P_{t_n-t/2}-P_{t/2})p_{t/4,y}\|_{L^2(\ambient,\meas)}\\
&\overset{\eqref{e.measb1+}}{\to }0\ \text{ as $n\to\infty$}.
\label{e.hk.joint.difference}
\end{align}
This proves the joint continuity of $p$.
	\item[\ref{it.meab1}] By the spectral theorem, $\gen P_{t/2}$ is bounded on $L^2(\ambient,\meas)$, since $\sup_{\lambda\in[0,\infty)}\lambda e^{-t\lambda/2}<\infty$. Consequently, $p_{t,x}\in\Dom(\gen)\subset\domain$. By the definition of the weak gradient, $\wgrad_{\mathfrak O}p_{t,x}\in L^2(\skeleton,\medm)$.
	
	Note that the operator $\wgrad_{\mathfrak O}P_{t/2}:L^2(\ambient,\meas)\to L^2(\skeleton,\medm)$ is bounded. Indeed, for $f\in L^2(\ambient,\meas)$, the spectral theorem gives
	\[
	\|\wgrad_{\mathfrak O}P_{t/2}f\|_{L^2(\skeleton,\medm)}^2
	=\|(-\gen)^{\frac{1}{2}}P_{t/2}f\|_{L^2(\ambient,\meas)}^2
	\leq (et)^{-1}\|f\|_{L^2(\ambient,\meas)}^2.
	\]
	Therefore, $x\mapsto\wgrad_{\mathfrak O}p_{t,x}=\wgrad_{\mathfrak O}P_{t/2}p_{t/2,x}$ is continuous as an $L^2(\skeleton,\medm)$-valued map.
	
	Since $(\ambient,\metric)$ is proper and separable, it is Polish. The skeleton is a countable union of finite-length arcs, and hence $(\skeleton,\medm)$ is a $\sigma$-finite measure space. By Lemma \ref{l.measu}, we may define $h_t(x,y):=\mathsf M(\wgrad_{\mathfrak O}p_{t,x},y)$. Then $h_{t}$ is measurable because it is the composition of the continuous map $(x,y)\mapsto (\wgrad_{\mathfrak O}p_{t,x},y)\in L^{2}(\skeleton,\medm)\times \skeleton$ with the measurable function $\mathsf{M}:L^{2}(\skeleton,\medm)\times \skeleton\to\bR$.
	\item[\ref{it.meab2}] The analyticity of $t\mapsto p_{t}(x,y)$ follows from \cite[Theorem~3]{Dav97} and \eqref{e.meab1}. By \cite[Proposition~3.1]{GHH21}, the function $(t,x,y)\mapsto p_{t}(x,y)$ is jointly measurable. Therefore,
	\begin{equation}
		\frac{\dif}{\dif t}p_{t}(x,y)=\lim_{n\to\infty}n\left(p_{t+n^{-1}}(x,y)-p_{t}(x,y)\right)
	\end{equation}
	is measurable as the pointwise limit of a sequence of measurable functions.
	\item[\ref{it.meab3}] Fix $t\in(0,\infty)$ and $x\in\ambient$. Since $p_{t,x}\in\Dom(\gen)$, by the definition of generator, $s^{-1}(p_{t+s,x}-p_{t,x})\to\gen p_{t,x}$ in $L^2(\ambient,\meas)$ as $s\downarrow0$. Taking $s=1/n$ and then passing to a subsequence, the difference quotients converge to $\gen p_{t,x}$ for $\meas$-a.e. $y$. On the other hand, by \ref{it.meab2}, the full sequence converges at every $y$ to $\frac{\dif}{\dif t}p_t(x,y)$. Therefore, the two limits agree for $\meas$-a.e. $y$.
\qedhere
\end{enumerate}
\end{proof}

\begin{lemma}\label{l.cutoff}
Let $(\ambient,\metric)$ be a proper, connected, and separable uniform local tree with skeleton $\skeleton$ and length measure $\medm$. Choose an orientation $\mathfrak{O}$ of $(\ambient,\metric)$. Let $\meas$ be a Radon measure on $(\ambient,\metric)$ satisfying $\UVG$. Let $(\form,\domain)$ be the strongly local regular symmetric Dirichlet form on $L^2(\ambient,\meas)$ constructed in Proposition \ref{p.energy}. There exists $C\in(0,\infty)$ such that, for any $(x,\epsilon)\in\ambient\times(0,6^{-1}\iota)$, there exists $\eta_{x,\epsilon}\in \domain\cap C_{c}(\ambient)$ with $0\leq\eta_{x,\epsilon}\leq1$ such that $\restr{\eta_{x,\epsilon}}{B(x,\epsilon)}=1$,  $\restr{\eta_{x,\epsilon}}{\ambient\setminus B(x,4\epsilon)}=0,$\begin{equation}\label{e.cutoff1}
		\abs{\wgrad_{\mathfrak{O}}\eta_{x,\epsilon}}\leq \epsilon^{-1}\ \text{ for $\medm$-a.e. $y\in\skeleton$},\ \ \abs{\wgrad_{\mathfrak{O}}\eta_{x,\epsilon}}=0\ \text{ for $\medm$-a.e. $y\in B(x,2\epsilon)^{c}$},
	\end{equation}
	and \begin{equation}\label{e.cutoff2}
		\medm(\supp_{\medm}[\abs{\wgrad_{\mathfrak{O}}\eta_{x,\epsilon}}])\leq C\epsilon.
	\end{equation}
\end{lemma}
\begin{proof}
	By $\UVG$, we know that the hypotheses of \cite[Lemma~3.3]{BCY25} hold. Apply \cite[Lemma~3.3]{BCY25} with parameter $2\epsilon$, we obtain a function $\Psi_{x}^{2\epsilon}\in\domain$ such that $\restr{\Psi_{x}^{2\epsilon}}{B(x,\epsilon)}\geq \frac{1}{2}$, $\restr{\Psi_{x}^{2\epsilon}}{\ambient\setminus B(x,4\epsilon)}=0$, and $\medm(\supp_{\medm}[\abs{\wgrad_{\mathfrak{O}}\Psi_{x}^{2\epsilon}}])\leq C\epsilon$ for some constant $C$ independent of $(x,\epsilon)$. Define $\eta_{x,\epsilon}(z):=0\vee((2\Psi_{x}^{2\epsilon}(z))\wedge1)$, $z\in\ambient$. Since the function $s\mapsto 0\vee((2s)\wedge1)$ is Lipschitz, the chain rule in Lemma~\ref{l.wgrad}-\ref{it.chain} gives $\abs{\wgrad_{\mathfrak{O}}\eta_{x,\epsilon}}\leq 2\abs{\wgrad_{\mathfrak{O}}\Psi_{x}^{2\epsilon}}$. Applying \cite[Lemma~3.3-(v), (vi)]{BCY25}, we obtain \eqref{e.cutoff1} and \eqref{e.cutoff2}.
\end{proof}

\subsubsection{Homogeneous Riesz transform}

In this section, we define the Riesz transform operator $\riesz$.

\begin{lemma}\label{l.defriesz}
	Let $(\ambient,\metric)$ be an unbounded, proper, connected, and separable uniform local tree with skeleton $\skeleton$ and length measure $\medm$. Choose an orientation $\mathfrak{O}$ of $(\ambient,\metric)$. Let $\meas$ be a Radon measure on $(\ambient,\metric)$ that satisfies $\UVG$. Let $(\form,\domain)$ be the strongly local regular symmetric Dirichlet form on $L^2(\ambient,\meas)$ constructed in Proposition \ref{p.energy}.\begin{enumerate}[label=\textup{({\arabic*})},align=right,leftmargin=*,topsep=5pt,parsep=0pt,itemsep=2pt]
	\item\label{it.dres1} $\mathrm{Ker}(-\gen)=\{0\}$, and the subspace $\Dom((-\gen)^{-\frac{1}{2}})$ defined by \begin{equation}
	\Dom((-\gen)^{-\frac{1}{2}})=\Sett{f\in L^{2}(\ambient,\meas)}{\int_{[0,\infty)}\lambda^{-1}\dif\langle \proj_{\lambda}f,f\rangle<\infty}.
\end{equation}
 is dense in $(L^{2}(\ambient,\meas),\norm{\cdot}_{L^{2}(\ambient,\meas)})$. Moreover, for each $f\in \Dom((-\gen)^{-\frac{1}{2}})$, we have $(-\gen)^{-\frac{1}{2}}f\in\domain$ and
\begin{equation}\label{e.iso1}
	\norm{\wgrad_{\mathfrak{O}}(-\gen)^{-\frac{1}{2}}f}_{L^{2}(\skeleton,\medm)}=\norm{f}_{L^{2}(\ambient,\meas)},\ \text{for all }f\in \Dom((-\gen)^{-\frac{1}{2}}).
\end{equation}
	\item\label{it.dres2} There exists a unique linear isometry $\riesz_{\mathfrak{O}}:L^{2}(\ambient,\meas)\to L^{2}(\skeleton,\medm)$ such that $\riesz f=\wgrad_{\mathfrak{O}}(-\gen)^{-\frac{1}{2}}f$ for all $f\in \Dom((-\gen)^{-\frac{1}{2}})$.
	\item\label{it.dres3} For every $\kappa\in(0,\infty)$ and every $f\in \domain$, we have $(-\gen+\kappa)^{-\frac{1}{2}}f\in \domain$ and 
	\begin{equation}
		\lim_{\kappa \downarrow 0}\wgrad_{\mathfrak{O}}(-\gen+\kappa)^{-\frac{1}{2}}f=\riesz_{\mathfrak{O}} f\ \text{ in }L^{2}(\skeleton,\medm).
	\end{equation}
\end{enumerate}
\end{lemma}
\begin{proof}
\begin{enumerate}[label=\textup{({\arabic*})},align=right,leftmargin=*,topsep=5pt,parsep=0pt,itemsep=2pt]
	\item[\ref{it.dres1}] Let $h\in\Dom(-\gen)\subset \domain$ such that $-\gen h=0$. We represent $h$ by its $\meas$-version in $\abscon(\ambient,\metric)$. Then, \begin{equation}
		\norm{\wgrad_{\mathfrak O}h}^{2}_{L^{2}(\skeleton,\medm)}=\form(h,h)=\langle-\gen h,h\rangle_{L^{2}(\ambient,\meas)}=0,
	\end{equation}
	that is, $\wgrad_{\mathfrak O}h=0$ $\medm$-a.e. on $\skeleton$. By Proposition \ref{p.arc}, $h$ is locally constant over $\ambient$. Since $(\ambient,\metric)$ is connected, $h$ must be a constant. Since $\meas(\ambient)=\infty$ and $h\in L^2(\ambient,\meas)$ is constant, we obtain $h=0$. This shows that $\mathrm{Ker}(-\gen)=\{0\}$. In particular, $\proj(\{0\})=0$.
	
	 For any $f\in L^{2}(\ambient,\meas)$, denote $f_{n}:=\one_{[n^{-1},n]}(-\gen)f\in L^{2}(\ambient,\meas)$ for $n\in\bN$. By spectral calculus and the dominated convergence theorem, \begin{align}
		\lim_{n\to\infty}\norm{f-f_{n}}_{L^{2}(\ambient,\meas)}^{2}&=\lim_{n\to\infty}\left(\int_{[0,n^{-1})}\dif\langle \proj_{\lambda}f,f\rangle+\int_{(n,\infty)}\dif\langle \proj_{\lambda}f,f\rangle\right)\\
		&=\langle \one_{\{0\}}(-\gen)f,f\rangle_{L^{2}(\ambient,\meas)}.\label{e.dresc1}
	\end{align}
	Since $\one_{\{0\}}(-\gen)f\in\Ker(-\gen)$, we have $\one_{\{0\}}(-\gen)f=0$. By \eqref{e.dresc1}, we obtain that $f_{n}\to f$ in $L^{2}(\ambient,\meas)$. This proves the density of $\Dom((-\gen)^{-\frac{1}{2}})$. For each $f\in \Dom((-\gen)^{-\frac{1}{2}})$, spectral calculus gives
\begin{equation}
	\int_{[0,\infty)}\lambda\dif\langle \proj_{\lambda}(-\gen)^{-\frac{1}{2}}f,(-\gen)^{-\frac{1}{2}}f\rangle=\int_{[0,\infty)}\dif\langle \proj_{\lambda}f,f\rangle=\norm{f}_{L^{2}(\ambient,\meas)}^{2}<\infty.
\end{equation}
Therefore, \begin{equation}
\norm{\wgrad_{\mathfrak{O}}(-\gen)^{-\frac{1}{2}}f}_{L^{2}(\skeleton,\medm)}^{2}=\form((-\gen)^{-\frac{1}{2}}f,(-\gen)^{-\frac{1}{2}}f)=\norm{f}_{L^{2}(\ambient,\meas)}^{2}.
\end{equation}

	\item[\ref{it.dres2}]	By \eqref{e.iso1}, the map $\wgrad_{\mathfrak{O}}(-\gen)^{-\frac{1}{2}}:\Dom((-\gen)^{-\frac{1}{2}})\to L^{2}(\skeleton,\medm)$ is an isometry. Since $\Dom((-\gen)^{-\frac{1}{2}})$ is dense in $(L^{2}(\ambient,\meas),\norm{\cdot}_{L^{2}(\ambient,\meas)})$, we may extend $\wgrad_{\mathfrak{O}}(-\gen)^{-\frac{1}{2}}$ uniquely to all of $L^{2}(\ambient,\meas)$, and we denote the resulting operator by $\riesz_{\mathfrak{O}}$. This proves existence, while uniqueness follows from the density of $\Dom((-\gen)^{-\frac{1}{2}})$. 	\item[\ref{it.dres3}]For each $f\in\domain$ and $\kappa\in(0,\infty)$, define $f_{\kappa}:=(-\gen)^{\frac{1}{2}}(-\gen+\kappa)^{-\frac{1}{2}}f$. By spectral calculus and the fact that $\Ker(-\gen)=\{0\}$ in \ref{it.dres1}, we have  $f_{\kappa}\in\Dom((-\gen)^{-\frac{1}{2}})$ and $f_{\kappa}\to f$ in $L^{2}(\ambient,\meas)$ as $\kappa\downarrow 0$. Therefore, by the continuity of $\riesz_{\mathfrak{O}}$, we have
\begin{equation}
		\lim_{\kappa \downarrow 0}\wgrad_{\mathfrak{O}}(-\gen+\kappa)^{-\frac{1}{2}}f=\lim_{\kappa \downarrow 0}\wgrad_{\mathfrak{O}}(-\gen)^{-\frac{1}{2}}f_{\kappa}=\lim_{\kappa\downarrow 0}\riesz_{\mathfrak{O}} f_{\kappa}=\riesz_{\mathfrak{O}} f.
	\end{equation}	
	This completes the proof.
\qedhere
\end{enumerate}
\end{proof}

\begin{lemma}
\label{l.funint}
Let $\phi:[0,\infty)\to\bR$ be a measurable function such that $\int_{0}^{\infty}\abs{\phi(t)}\dif t<\infty$. Define
\begin{equation}
 m_{\phi}(\lambda):=\int_{0}^{\infty}\phi(t)e^{-t\lambda}\dif t, \text{ and }n_{\phi}(\lambda):=\lambda^{\frac{1}{2}}m_{\phi}(\lambda), \text{ for all }\lambda\in[0,\infty).
\end{equation}
Let $f\in L^{2}(\ambient,\meas)$. Then the following statements hold.
\begin{enumerate}[label=\textup{({\arabic*})},align=right,leftmargin=*,topsep=5pt,parsep=0pt,itemsep=2pt]
	\item\label{it.scalar1} $m_{\phi}(-\gen)f\in L^{2}(\ambient,\meas)$ is the unique element satisfying
\begin{equation}\label{e.scamtp}
 \langle m_{\phi}(-\gen)f,g\rangle_{L^2(\ambient,\meas)} = \int_{0}^{\infty}\phi(t)\langle P_tf,g\rangle_{L^2(\ambient,\meas)}\dif t,\ \text{for all }g\in L^2(\ambient,\meas).
\end{equation}
\item\label{it.scalar2} If $\sup_{\lambda\in[0,\infty)}\abs{n_{\phi}(\lambda)}<\infty$, then $m_{\phi}(-\gen)f\in\domain$, and $\wgrad_{\mathfrak{O}} m_{\phi}(-\gen)f=\riesz_{\mathfrak{O}} n_{\phi}(-\gen)f$ $\medm$-a.e. on $\skeleton$.
\item\label{it.scalar3} If $\int_{0}^{\infty}|\phi(t)|t^{-\frac{1}{2}}\dif t<\infty$, then for all $h\in L^{2}(\skeleton,\medm)$,
 \begin{equation}\label{e.sccgradf}
 \langle \riesz_{\mathfrak{O}} n_{\phi}(-\gen)f,h\rangle_{L^2(\skeleton,\nu)}
 =
 \int_{0}^{\infty}\phi(t)
 \langle\wgrad_{\mathfrak{O}} P_tf,h\rangle_{L^2(\skeleton,\nu)}\dif t.
\end{equation}
\end{enumerate}

\end{lemma}

\begin{proof}

\begin{enumerate}[label=\textup{({\arabic*})},align=right,leftmargin=*,topsep=5pt,parsep=0pt,itemsep=2pt]
	\item[\ref{it.scalar1}] Since $\norm{m_{\phi}}_{L^{\infty}([0,\infty))}\leq \norm{\phi}_{L^{1}([0,\infty))}$, the function $m_{\phi}$ defines a bounded linear operator $m_{\phi}(-\gen)$ on $L^{2}(\ambient,\meas)$, and hence $m_{\phi}(-\gen)f\in L^{2}(\ambient,\meas)$. By spectral calculus,
\begin{align}
		\langle m_{\phi}(-\gen)f,g\rangle_{L^2(\ambient,\meas)}& =\int_{[0,\infty)}m_{\phi}(\lambda)\dif\langle\proj_{\lambda}f,g\rangle=\int_{[0,\infty)}\int_{0}^{\infty}\phi(t)e^{-t\lambda}\dif t\dif\langle\proj_{\lambda}f,g\rangle\\
		&=\int_{0}^{\infty}\phi(t)\left(\int_{[0,\infty)}e^{-t\lambda}\dif\langle\proj_{\lambda}f,g\rangle\right)\dif t \ \text{(by Fubini's theorem)}\\
		&=\int_{0}^{\infty}\phi(t)\langle P_tf,g\rangle_{L^2(\ambient,\meas)}\dif t.
\end{align}
This proves \eqref{e.scamtp}. The uniqueness of $m_{\phi}(-\gen)f\in L^{2}(\ambient,\meas)$ follows from the Riesz representation theorem.
	\item[\ref{it.scalar2}] By spectral calculus,
\begin{align}
		&\phantom{\ \leq}\int_{[0,\infty)}\lambda\dif\langle\proj_{\lambda}m_{\phi}(-\gen)f,m_{\phi}(-\gen)f\rangle =\int_{[0,\infty)}\lambda m_{\phi}(\lambda)^{2}\dif\langle\proj_{\lambda}f,f\rangle\\
		&=\int_{[0,\infty)}n_{\phi}(\lambda)^{2}\dif\langle\proj_{\lambda}f,f\rangle\leq \Big(\sup_{[0,\infty)}n_{\phi}^{2}\Big)\norm{f}_{L^{2}(\ambient,\meas)}^{2}.
\end{align}
Therefore, $m_{\phi}(-\gen)f\in\domain$. Similarly, $n_{\phi}(-\gen)f\in \Dom((-\gen)^{-\frac{1}{2}})$, and therefore,
\begin{equation}
		\riesz_{\mathfrak{O}} n_{\phi}(-\gen)f=\wgrad_{\mathfrak{O}} (-\gen)^{-\frac{1}{2}}n_{\phi}(-\gen)f=\wgrad_{\mathfrak{O}} m_{\phi}(-\gen)f.
\end{equation}
	\item[\ref{it.scalar3}] Since $\sup_{\lambda\in[0,\infty)}\lambda^{1/2}e^{-t\lambda}=(2et)^{-1/2}$ for $t\in(0,\infty)$, we have 
	\begin{equation}
		\sup_{\lambda\in[0,\infty)}\lambda^{\frac{1}{2}}\abs{\int_{0}^{\infty}\phi(t)e^{-t\lambda}\dif t}\leq (2e)^{-\frac{1}{2}}\int_{0}^{\infty}|\phi(t)|t^{-\frac{1}{2}}\dif t<\infty.\label{e.sgdl1}
	\end{equation} 
	Thus the assumption in \ref{it.scalar2} is satisfied. For $t\in(0,\infty)$, spectral calculus gives
\begin{equation}\label{e.sgdl2}
 \|\wgrad_{\mathfrak{O}} P_tf\|_{L^2(\skeleton,\medm)}
 =\|(-\gen)^{\frac{1}{2}}P_tf\|_{L^2(\ambient,\meas)}
 \leq(2et)^{-\frac{1}{2}}\|f\|_{L^2(\ambient,\meas)}.
\end{equation}
Thus, by the Cauchy--Schwarz inequality,
\begin{align}
&\phantom{\ \leq} \int_{0}^{\infty}|\phi(t)|
 |\langle\wgrad_{\mathfrak{O}} P_tf,h\rangle_{L^{2}(\skeleton,\medm)}|\dif t\\
& \overset{\eqref{e.sgdl2}}{\leq}(2e)^{-\frac{1}{2}}\int_{0}^{\infty}|\phi(t)|t^{-\frac{1}{2}}\dif t\|f\|_{L^{2}(\ambient,\meas)}\|h\|_{L^{2}(\skeleton,\medm)}\overset{\eqref{e.sgdl1}}{<}\infty.
\end{align}

Since $\wgrad_{\mathfrak{O}} P_{t}f=\riesz_{\mathfrak{O}}(-\gen)^{\frac{1}{2}}P_{t}f$ by \ref{it.scalar2}, if we denote the adjoint of $\riesz_{\mathfrak{O}}$ by $\riesz_{\mathfrak{O}}^*:L^{2}(\skeleton,\medm)\to L^{2}(\ambient,\meas)$, then Fubini's theorem gives
\begin{align}
	 &\phantom{\ \leq}\int_{0}^{\infty}\phi(t)
 \langle\wgrad_{\mathfrak{O}} P_tf,h\rangle_{L^2(\skeleton,\nu)}\dif t = \int_{0}^{\infty}\phi(t)
 \langle\riesz_{\mathfrak{O}}(-\gen)^{\frac{1}{2}}P_{t}f,h\rangle_{L^2(\skeleton,\nu)}\dif t\\
 &=\int_{0}^{\infty}\phi(t)
 \langle (-\gen)^{\frac{1}{2}}P_{t}f,\riesz_{\mathfrak{O}}^*h\rangle_{L^{2}(\ambient,\meas)}\dif t\\
 &=\langle n_{\phi}(-\gen)f,\riesz_{\mathfrak{O}}^*h\rangle_{L^{2}(\ambient,\meas)}=\langle \riesz_{\mathfrak{O}} n_{\phi}(-\gen)f,h\rangle_{L^{2}(\skeleton,\medm)}.
\end{align}
This proves \eqref{e.sccgradf}.
\qedhere
\end{enumerate}
\end{proof}
\begin{lemma}\label{l.reverse-duality}
Let $p\in(1,\infty)$ and $p'=p/(p-1)$. Then \begin{equation}
	\R{p'}\ \Longrightarrow \ \RR{p}.
\end{equation}
\end{lemma}
\begin{proof}
Let $u\in\domain=\Dom((-\gen)^{\frac12})$. By Lemma~\ref{l.defriesz}-\ref{it.dres1}, $\Ker(-\gen)=\{0\}$. Therefore,
\begin{equation}\label{e.sharpdual1}
\int_{(0,\infty)}\lambda^{-1}\dif\langle\proj_\lambda (-\gen)^{\frac12}u,(-\gen)^{\frac12}u\rangle=\int_{(0,\infty)}\dif\langle\proj_\lambda u,u\rangle=\norm{u}_{L^2(\ambient,\meas)}^2<\infty.
\end{equation}
Thus $(-\gen)^{\frac12}u\in\Dom((-\gen)^{-\frac12})$, and $(-\gen)^{-\frac12}((-\gen)^{\frac12}u)=u$. By Lemma~\ref{l.defriesz}-\ref{it.dres2}, $\riesz ((-\gen)^{\frac12}u)=\wgrad u$, and $\riesz:L^{2}(\ambient,\meas)\to L^{2}(\skeleton,\medm)$ is a linear isometry. Therefore, for every $u\in\domain$ and every $g\in L^2(\ambient,\meas)$, we have
\begin{align}
\langle\wgrad u,\riesz g\rangle_{L^2(\skeleton,\medm)}&=\langle\riesz ((-\gen)^{\frac12}u),\riesz g\rangle_{L^2(\skeleton,\medm)}\\
&=\langle(-\gen)^{\frac12}u,g\rangle_{L^2(\ambient,\meas)}\ \text{ (since $\riesz$ is an isometry).}\label{e.reverse-pairing}
\end{align}
Suppose that $\wgrad u\in L^p(\skeleton,\medm)$. For $g\in L^2(\ambient,\meas)\cap L^{p'}(\ambient,\meas)$, H\"older's inequality and $\R{p'}$ give
\begin{equation}
\abs{\langle (-\gen)^{\frac12}u,g\rangle_{L^2(\ambient,\meas)}}\overset{\eqref{e.reverse-pairing}}{\leq}\norm{\wgrad u}_{L^p(\skeleton,\medm)}\norm{\riesz g}_{L^{p'}(\skeleton,\medm)}\overset{\R{p'}}{\leq} C_{p'}\norm{\wgrad u}_{L^p(\skeleton,\medm)}\norm{g}_{L^{p'}(\ambient,\meas)}.\label{e.sharpdual2}
\end{equation}
Since \eqref{e.sharpdual2} holds for all $g\in L^2(\ambient,\meas)\cap L^{p'}(\ambient,\meas)$. Since $L^2(\ambient,\meas)\cap L^{p'}(\ambient,\meas)$ is dense in $L^{p'}(\ambient,\meas)$, by \eqref{e.sharpdual2} and the duality, we know that \begin{equation}
	\norm{(-\gen)^{\frac12}u}_{L^p(\ambient,\meas)}\leq C_{p'}\norm{\wgrad u}_{L^p(\skeleton,\medm)},\ \text{ whenever }u\in \domain \ \text{ and }\wgrad u\in L^p(\skeleton,\medm).
\end{equation}
This is exactly $\RR{p}$.
\end{proof}

\section{Heat kernel estimates and their consequences}\label{s.heat}
In this section and in the rest of this paper, we will work under Framework \ref{f.treeMMD}. After rescaling the metric, we also assume that $\iota>40$. We choose $\epsilon=1$, let $\vertex$ be a maximal $1$-separated set of $(\ambient,\metric)$, and let $\edge$ be defined by \eqref{e.edges}.

We first recall some basic facts on heat kernel estimates.

\begin{definition}\label{d.HKE}
We say that $(\form,\domain)$ on $L^{2}(\ambient,\meas)$ satisfies the \hypertarget{HKE}{\emph{heat kernel estimates}} $\HKE$ if there exist
		$C_{1},c_{2}, \delta\in(0,\infty)$ and a heat kernel $\set{p_t}_{t\in(0,\infty)}$ of its semigroup $\{P_{t}\}_{t\in(0,\infty)}$ such that, for every $t\in(0,\infty)$,
  		\begin{align}
		p_{t}(x,y) &\leq \frac{C_{1}}{\meas(B(x,\scale^{-1}(t)))} \exp\left(-c_{2}\Phi\left({\metric(x, y)},{t}\right)\right)
		\quad \mbox{for $\meas\otimes \meas$-a.e.\ $(x,y) \in \ambient\times \ambient$}\\
		\text{and }\ p_{t}(x,y) &\geq \frac{C_{1}^{-1}}{\meas(B(x,\scale^{-1}(t)))}
		\quad \mbox{for $\meas\otimes \meas$-a.e.\ $(x,y) \in \ambient\times \ambient$ with $\metric(x,y) \leq \delta \scale^{-1}(t)$},
		\end{align}
		where  \begin{equation}\label{e.phi}
		\Phi(R,t):=\sup_{r\in(0,\infty)}\left({\frac{R}{r}-\frac{t}{\scale(r)}}\right),\ (R,t)\in[0,\infty)\times(0,\infty).
		\end{equation}
		We say that the \hypertarget{UHK}{\emph{heat kernel upper estimates}} $\UHK$ hold if the upper bound in $\HKE$ holds.
\end{definition}

\begin{remark}
	Suppose $(\ambient,\metric,\meas,\form,\domain)$ satisfies $\UVG$ and $\HKE$. \begin{enumerate}[label=\textup{({\arabic*})},align=right,leftmargin=*,topsep=5pt,parsep=0pt,itemsep=2pt]
	  \item Since $(\ambient,\metric)$ is geodesic, it satisfies the \emph{chain condition} \cite[Definition~2.10-(a)]{KM20}. By \cite[Remark~2.5-(d)]{KM20}, the heat kernel estimates $\HKE$ can be strengthened to the \hypertarget{HKEf}{\emph{full heat kernel estimates}}, that is, the upper bound of $p_{t}$ in $\HKE$ remains valid, while the lower bound of $p_{t}$ in $\HKE$ is replaced by \begin{equation}
		p_{t}(x,y) \geq \frac{C_{3}^{-1}}{\meas(B(x,\scale^{-1}(t)))}\exp\left(-c_{4}\Phi\left({\metric(x, y)},{t}\right)\right),
	\end{equation}
		for $\meas\otimes \meas$-a.e. $(x,y) \in \ambient\times \ambient$, for some constants $C_{3},c_{4}\in(0,\infty)$. Conversely, full off-diagonal estimates imply the chain condition by \cite[Theorem~2.11]{Mur20}; under precompactness of balls, \cite[Proposition~A.1]{KM20} identifies this with bi-Lipschitz equivalence to a geodesic metric.
		\item By $\UVG$, $\HKE$, the unboundedness of $(\ambient,\metric)$ and \cite[Theorem~7.4 and Lemma~7.3-(b)]{GT12} (alternatively \cite[Theorem~3.2]{Lie15}), we know that the heat kernel is \emph{conservative}, i.e., \begin{equation}\label{e.conserv}
		1=P_{t}\one_{\ambient}(x):=\int_{\ambient}p_{t}(x,y)\dif \meas(y),\ \text{ for all }(t,x)\in(0,\infty)\times\ambient.
	\end{equation}
\end{enumerate}
\end{remark}

\begin{lemma}\label{l.growthPsi}
	Suppose that $(\ambient,\metric,\meas,\form,\domain)$ satisfies $\UVG$ and $\HKE$. Then there exists $C\in(1,\infty)$ such that \begin{equation}\label{e.growth1}
		C^{-1}\frac{R^{2}}{r^{2}}\leq\frac{\scale(R)}{\scale(r)}\leq C\frac{R\volume(R)}{r\volume(r)},\ \text{ for all }0<r\leq R<\infty,
	\end{equation}
	and \begin{equation}\label{e.growth2}
		C^{-1}r\volume(r)\leq\scale(r)\leq Cr\volume(r),\ \text{ for all }r\in(0,\iota].
	\end{equation}
\end{lemma}
\begin{proof}
We first regularize the volume function $\volume$. By $\UVG$ and the geodesic property of $(\ambient,\metric)$, we have a constant $c\in(0,\infty)$ such that
\begin{equation}\label{e.volreverse}
\frac{\volume(R)}{\volume(r)}\geq c\frac{R}{r},\ \text{ for all } 0<r\leq R<\infty.
\end{equation}
Define
\[
\widetilde{\volume}(r):=\int_0^r\frac{\volume(s)}{s}\dif s,\ \text{ for } r\in[0,\infty).
\]
By \eqref{e.volreverse} and $\UVG$, we know that there exists $C_{1}\in(1,\infty)$ such that
\[
C_{1}^{-1}\volume(r)\leq\widetilde{\volume}(r)\leq C_{1}\volume(r),\ \text{ for all } r\in(0,\infty).
\]
Consequently, $\widetilde{\volume}$ is a doubling homeomorphism of $[0,\infty)$ onto itself, and $\UVG$ also holds with $\widetilde{\volume}$ in place of $\volume$. The full heat-kernel estimates and \cite[Theorem~2.6]{Mur25} imply \eqref{e.growth1}.

Set $\Theta(r):=r\widetilde{\volume}(r)$ for $r\in[0,\infty)$. Then $\Theta$ is an increasing homeomorphism of $[0,\infty)$ onto itself. The estimates \eqref{e.volreverse} and \eqref{e.VD1} verify the assumptions on the volume growth in \cite[Section~3]{BCY25}. Choose the continuous heat kernel provided by \cite[Proposition~4.2]{BCY25}. By \cite[Theorem~4.3, Lemma~4.6]{BCY25}, there exist $C_{2}\in(1,\infty)$ and $t_{0}\in(0,\infty)$ such that \begin{equation}
		\frac{C_{2}^{-1}}{\wt{\volume}(\Theta^{-1}(t))}\leq p_{t}(x,x)\leq \frac{C_{2}}{\wt{\volume}(\Theta^{-1}(t))},\ \text{ for all }(t,x)\in (0,t_{0}]\times\ambient,
	\end{equation}
which, combined with the continuity of $p_{t}$, $\UVG$, and $\HKE$, gives
\begin{equation}\label{e.pgrowth1}
C_{3}^{-1}
\widetilde{\volume}(\Theta^{-1}(t))\leq \widetilde{\volume}(\scale^{-1}(t))\leq C_{3}
\widetilde{\volume}(\Theta^{-1}(t)),\ \text{ for all } t\in(0,t_0].
\end{equation}
Let $t\in(0,t_0]$. Then \eqref{e.volreverse} and \eqref{e.pgrowth1} give
\begin{equation}\label{e.review_inverse_comparison}
\max\left(\frac{\scale^{-1}(t)}{\Theta^{-1}(t)},\frac{\Theta^{-1}(t)}{\scale^{-1}(t)}\right)\leq C\max\left(\frac{\widetilde\volume(\scale^{-1}(t))}{\widetilde\volume(\Theta^{-1}(t))},\frac{\widetilde\volume(\Theta^{-1}(t))}{\widetilde\volume(\scale^{-1}(t))}\right)\leq C.
\end{equation}
For $r\in(0,\scale^{-1}(t_{0})]$, apply \eqref{e.review_inverse_comparison} at $t=\scale(r)$ and then apply the doubling property of $\Theta$. This yields
\begin{equation}\label{e.review_scale_comparison}
\Theta(C^{-1}r)\leq\scale(r)\leq\Theta(Cr),\ \text{ and } \scale(r)\asymp\Theta(r)\asymp r\volume(r).
\end{equation}
 This proves \eqref{e.growth2}.
\end{proof}
Recall that, by Lemma \ref{l.meab}-\ref{it.meab0}, for the canonical Dirichlet form on a uniform local tree, a jointly continuous heat kernel always exists.
\begin{definition}\label{d.Grad}
	We say that $(\form,\domain)$ on $L^{2}(\ambient,\meas)$ satisfies the \hypertarget{Grad}{\emph{heat kernel gradient estimates}} $\Grad$ if there exist $C_{1},c_{2}\in(0,\infty)$ such that, for every $t\in(0,\infty)$, the following statement holds for $\meas$-a.e. $x\in\ambient$, \begin{equation}
		\abs{\wgrad p_{t,x}(y)}\leq \frac{C_{1}}{\scale^{-1}(t)\meas(B(x,\scale^{-1}(t)))}\exp\left(-c_{2}\Phi\left({\metric(x, y)},{t}\right)\right)\  \mbox{for $\medm$-a.e.\ $y \in \skeleton$.}
	\end{equation}
\end{definition}
\begin{remark}\label{r.Grad}
Assume $\UVG$ and $\Grad$. By the joint continuity of $p_{t}$ for each $t\in(0,\infty)$ and Lemma \ref{l.meab}, for each $t\in(0,\infty)$, there exists a jointly measurable function $h_{t}:\ambient\times \skeleton\to\bR$, such that, for each $x\in\ambient$, we have $h_{t}(x,\cdot)=\wgrad p_{t,x}$ $\medm$-a.e. on $\skeleton$. Choose a jointly measurable representative $h_t(x,y)$ as in Lemma~\ref{l.meab}, and set
\[
H_t(x,y):=\frac{C}{\scale^{-1}(t)\volume(\scale^{-1}(t))}\exp({-c\Phi(\metric(x,y),t)}),\]
and \begin{equation}
	N_t:=\Sett{(x,y)\in \ambient\times\skeleton}{|h_t(x,y)|>H_t(x,y)},\ t\in(0,\infty).
\end{equation}
	Then $N_{t}$ is a measurable subset of $\ambient\times \skeleton$ since $h_{t}$ is measurable. For every $x\in\ambient$, by $\Grad$, the continuity of the map $\ambient\ni x\mapsto \wgrad p_{t,x}\in L^{2}(\skeleton,\medm)$ and the lower semicontinuity of $\Phi$ in Lemma \ref{l.phi}, we know that the section of $N_t$ is $\medm$-null,
\begin{equation}\label{e.measexc}
\medm(\Sett{y\in\skeleton}{(x,y)\in N_t})=0.
\end{equation}
	Let $\wt{h}_{t}(x,y):=h_{t}(x,y)\one_{(\ambient\times \skeleton)\setminus N_{t}}(x,y)$. Then $\wt{h}_{t}$ is jointly measurable, for each $x\in\ambient$, $\wt{h}_{t}(x,\cdot)=\wgrad p_{t,x}$ $\medm$-a.e. on $\skeleton$ by \eqref{e.measexc}, and \begin{equation}
		\abs{\wt{h}_{t}(x,y)}\leq \frac{C_{1}}{\scale^{-1}(t)\meas(B(x,\scale^{-1}(t)))}\exp\left(-c_{2}\Phi\left({\metric(x, y)},{t}\right)\right),\ \text{ for all }(x,y)\in\ambient\times\skeleton.
	\end{equation}
	For each $x\in\ambient$, we always choose $\wt{h}_{t}(x,\cdot)$ as the representative of $\wgrad p_{t,x}$ in its $\medm$-equivalence class. Therefore, the upper bound in $\Grad$ holds \emph{in the pointwise sense}.
\end{remark}

\begin{lemma}\label{l.intexp}
Assume $\UVG$. For every $c\in(0,\infty)$, there exists a constant $C\in(0,\infty)$ such that, for all $t\in(0,\infty)$, $y\in\ambient$, and $R\in[0,\infty)$,
\begin{equation}\label{e.intexp}
 \int_{\ambient\setminus B(y,R)}
 \exp(-c\Phi(\metric(x,y),t))\dif\meas(x)
 \leq C\volume(\scale^{-1}(t))\exp(-c\Phi(R,t)/2).
\end{equation}
\end{lemma}
\begin{proof}

Set $A_{0}:=B(y,\scale^{-1}(t))$. For $k\in\bN$, set $A_k:=B(y,2^{k}\scale^{-1}(t))\setminus B(y,2^{k-1}\scale^{-1}(t))$. Then by Lemma \ref{l.phi}-\ref{it.phi3} and $\UVG$,\begin{align}
	&\phantom{\ \leq}\int_{\ambient}\exp(-\frac{c}{2}\Phi(\metric(x,y),t))\dif\meas(x)=\sum_{k=0}^{\infty}\int_{A_{k}}\exp(-\frac{c}{2}\Phi(\metric(x,y),t))\dif\meas(x)\\
	&\lesssim \volume(\scale^{-1}(t))+\sum_{k=1}^{\infty}\volume(2^{k}\scale^{-1}(t))\exp(-c_{1}2^{c_{2}k})\ \text{ (by $\UVG$ and \eqref{e.phi2})}\\
	&\lesssim \sum_{k=0}^{\infty}\volume(\scale^{-1}(t))2^{kd_{2}}\exp(-c_{1}2^{c_{2}k})\lesssim \volume(\scale^{-1}(t)).\label{e.intexp1}
\end{align}
If $\metric(x,y)\geq R$, then by Lemma \ref{l.phi}-\ref{it.phi1},\begin{align}
	\exp(-c\Phi(\metric(x,y),t))&=\exp(-\frac{c}{2}\Phi(\metric(x,y),t))\exp(-\frac{c}{2}\Phi(\metric(x,y),t))\\
	&\leq \exp(-\frac{c}{2}\Phi(R,t))\exp(-\frac{c}{2}\Phi(\metric(x,y),t)),\label{e.intexp2}
\end{align}
and hence \begin{align}
	&\phantom{\ \leq}\int_{\ambient\setminus B(y,R)}
 \exp(-c\Phi(\metric(x,y),t))\dif\meas(x)\\
 &\overset{\eqref{e.intexp2}}{\leq} \exp(-\frac{c}{2}\Phi(R,t))\int_{\ambient} \exp(-\frac{c}{2}\Phi(\metric(x,y),t))\dif\meas(x)\overset{\eqref{e.intexp1}}{\lesssim}\volume(\scale^{-1}(t)) \exp(-\frac{c}{2}\Phi(R,t)).
\end{align}
This proves \eqref{e.intexp}.
\end{proof}
\begin{lemma}\label{l.time_derivative}
Assume $\UVG$ and $\UHK$. There exist $C,c_1\in(0,\infty)$ such that
\begin{equation}\label{e.ddt<}
	\abs{\frac{\dif}{\dif t}p_{t}(x,y)}\leq \frac{C}{t\volume(\scale^{-1}(t))} \exp\left(-c_{1}\Phi\left({\metric(x, y)},{t}\right)\right),\ \text{ for all }(t,x,y)\in(0,\infty)\times\ambient\times\ambient.
\end{equation}
\end{lemma}
\begin{proof}
Fix $t\in(0,\infty)$ and $x,y\in\ambient$. By \eqref{e.scale} and $\UVG$, $\volume(\scale^{-1}(s))\asymp\volume(\scale^{-1}(t))$ for $s\in[t/2,3t/2]$. By \eqref{e.phi},
\[
 \Phi(\metric(x,y),at)=a\Phi(\metric(x,y)/a,t),\ \text{ for all }a\in(0,\infty),
\]
which, combined with \eqref{e.phi2}, gives $\Phi(\metric(x,y),s)\ge c_0\Phi(\metric(x,y),t)$ for $s\in[t/2,3t/2]$. Thus, by $\UHK$, we have
\[
\max( p_{t/2}(x,x),p_{t/2}(y,y))\le\frac K{\volume(\scale^{-1}(t))},
  \text{ and } \sup_{s\in[t/2,3t/2]}p_s(x,y)\le\frac{K\exp(-c\Phi(\metric(x,y),t))}{\volume(\scale^{-1}(t))}.
\]
Applying \cite[Theorem~4]{Dav97} with $\delta=1/2$, $\varepsilon=1/16$, $a=b=K/\volume(\scale^{-1}(t))$ and $c=\exp({-c\Phi(\metric(x,y),t)})$, we have
\[
 \left|\frac{\dif}{\dif t}p_t(x,y)\right|
 \le\frac{32K}{t\volume(\scale^{-1}(t))}\exp({-\frac{13c}{16}\Phi(\metric(x,y),t)}).
\]
This proves the assertion.
\end{proof}

\begin{lemma}\label{l.deri<}
	Assume $\UVG$, $\UHK$ and $\Grad$.
	\begin{enumerate}[label=\textup{({\arabic*})},align=right,leftmargin=*,topsep=5pt,parsep=0pt,itemsep=2pt]
	\item\label{it.deri<1}  There exists $C\in(1,\infty)$ such that
	\begin{equation}\label{e.wl2<}
		\int_{\ambient}(\wgrad p_{t,x}(y))^{2}\dif\meas(x)\leq \frac{C}{\scale^{-1}(t)^{2}\volume(\scale^{-1}(t))},\ \text{ for all }(t,y)\in(0,\infty)\times\skeleton.
	\end{equation}
\item\label{it.deri<2} For $g\in L^2(\ambient,\meas)$ and $t\in(0,\infty)$, one has
$P_tg\in\Dom(-\gen)\subset\domain$. Moreover,
\begin{equation}\label{e.kernel1}
	(\gen P_tg)(y)
 =\int_{\ambient}\frac{\dif}{\dif t}p_t(x,y)g(x) \dif\meas(x)
\  \text{ for $\meas$-a.e. }y\in\ambient,
\end{equation}
and
\begin{equation}\label{e.kernel2}
	(\wgrad P_tg)(y)
 =\int_{\ambient}\wgrad p_{t,x}(y)g(x)\dif\meas(x)\ \text{ for $\medm$-a.e. }y\in\skeleton.
\end{equation}
Both integrals converge absolutely.
\end{enumerate}
\end{lemma}

\begin{proof}
\begin{enumerate}[label=\textup{({\arabic*})},align=right,leftmargin=*,topsep=5pt,parsep=0pt,itemsep=2pt]
\item[\ref{it.deri<1}] The estimate \eqref{e.wl2<} is a direct consequence of $\Grad$ and \eqref{e.intexp} in Lemma \ref{l.intexp} with $R=0$.
\item[\ref{it.deri<2}]  The inclusion $P_tg\in\Dom(-\gen)$ follows from spectral calculus. By \eqref{e.ddt<}, \eqref{e.wl2<}, Cauchy--Schwartz inequality and Lemma \ref{l.intexp}, the two integrals in \eqref{e.kernel1} and \eqref{e.kernel2} are absolutely convergent. The identity \eqref{e.kernel1} follows by differentiating $P_tg$ in $L^2(\ambient,\meas)$, \eqref{e.dif-gen} and the dominated convergence theorem.  Let $\gamma:[0,\Length(\gamma)]\to\ambient$ be a unit-speed geodesic, whose image is contained in a local tree, with $a:=\gamma(0)$ and $b=\gamma(\Length(\gamma))$. By the symmetry of the heat kernel and Proposition \ref{p.arc}, 
\begin{equation}
 \begin{aligned}
 &\phantom{\ \leq}\int_{\gamma([0,\Length(\gamma)])}\sigma_{\gamma,\mathfrak{O}}(y)
 (\wgrad P_tg)(y) \dif\medm(y)=P_tg(b)-P_tg(a)\\
 &=\int_{\ambient}\left(p_t(x,b)-p_t(x,a)\right)g(x)\dif\meas(x)\\
 &=\int_{\gamma([0,\Length(\gamma)])}\sigma_{\gamma,\mathfrak{O}}(y)
 \left(\int_{\ambient}(\wgrad p_{t,x})(y)g(x)\dif\meas(x)\right)\dif\medm(y)\ \text{(by Fubini's theorem)}.
 \end{aligned}
\end{equation}
By the Lebesgue differentiation theorem, we obtain \eqref{e.kernel2}.
\qedhere
\end{enumerate}
\end{proof}

\subsection{Gradient estimates of the semigroup}
Define \begin{equation}\label{e.deflamb}
	\Lambda(t):=(\scale^{-1}(t)\wedge t)^{-\frac{1}{2}}\scale^{-1}(t)^{-\frac{1}{2}}\volume(\scale^{-1}(t))^{-\frac{1}{2}}, \ t\in(0,\infty).
\end{equation}
Then $\Lambda$ is Borel measurable on $(0,\infty)$, and decreasing.
\begin{proposition}
\label{p.migrad}
Assume $\UVG$, $\UHK$ and $\Grad$. For
$p\in[2,\infty]$, there exists $C_{p}\in(0,\infty)$ such that, for every $t\in(0,\infty)$,
\begin{equation}\label{e.migrad}
 \left\lVert\wgrad P_tf\right\rVert_{L^p(\skeleton,\medm)}
 \leq C_pt^{-\frac{1}{p}}\scale^{-1}(t)^{-1+\frac{2}{p}}\lVert f\rVert_{L^p(\ambient,\meas)},\ \text{ for all $f \in L^p(\ambient,\meas)\cap L^{2}(\ambient,\meas)$.}
\end{equation}
\end{proposition}

\begin{proof}
By spectral calculus, for every $f\in L^2(\ambient,\meas)$,
\begin{equation}
 \left\lVert\wgrad P_tf\right\rVert_{L^2(\skeleton,\medm)}=\form(P_tf,P_{t}f)^{\frac{1}{2}}=\left(\int_{[0,\infty)}\lambda e^{-2t\lambda}
 \dif\langle\proj_{\lambda}f,f\rangle\right)^{\frac{1}{2}}
 \leq ({2et})^{-\frac{1}{2}}\lVert f\rVert_{L^2(\ambient,\meas)}.\label{e.l2bd}
\end{equation}
Fix $f\in L^\infty(\ambient,\meas)\cap L^2(\ambient,\meas)$. For $\medm$-a.e. $y\in\skeleton$, we have
\begin{align}
 \lvert\wgrad P_tf(y)\rvert
 &\overset{\eqref{e.kernel2}}{\leq} \lVert f\rVert_{L^\infty(\ambient,\meas)}
 \int_{\ambient}\lvert\wgrad p_{t,x}(y)\rvert\dif\meas(x)\\
 &\overset{\Grad}{\lesssim} \frac{\lVert f\rVert_{L^\infty(\ambient,\meas)}}{\scale^{-1}(t)\volume(\scale^{-1}(t))}
 \int_{\ambient}e^{-c\Phi(\metric(x,y),t)}\dif\meas(x)\overset{\eqref{e.intexp}}{\lesssim} \frac{C}{\scale^{-1}(t)}\lVert f\rVert_{L^\infty(\ambient,\meas)}.
\end{align}
Hence
\begin{equation}\label{e.gradinf}
 \lVert\wgrad P_tf\rVert_{L^\infty(\skeleton,\medm)}
 \leq \frac{C}{\scale^{-1}(t)}\lVert f\rVert_{L^\infty(\ambient,\meas)},\ \text{ for all $f\in L^\infty(\ambient,\meas)\cap L^2(\ambient,\meas)$}.
\end{equation}
Let $\theta=1-\frac{2}{p}$. 
Riesz--Thorin interpolation between \eqref{e.l2bd} and \eqref{e.gradinf} gives
\begin{align}
 \lVert\wgrad P_tf\rVert_{L^{p}(\skeleton,\medm)}
 &\leq C_p\bigl(t^{-\frac{1}{2}}\bigr)^{1-\theta}
 \bigl(\scale^{-1}(t)^{-1}\bigr)^\theta\lVert f\rVert_{L^p(\ambient,\meas)}=C_pt^{-\frac{1}{p}}\scale^{-1}(t)^{-1+\frac{2}{p}}\lVert f\rVert_{L^p(\ambient,\meas)}.
\end{align}
This proves \eqref{e.migrad}.  
\end{proof}

\begin{proposition}\label{p.cellgrad}
Assume $\UVG$, $\HKE$ and $\Grad$. There exist $C,c\in(0,\infty)$ such that, for every $t\in(0,\infty)$, every $w\in\vertex$, every Borel subset $E\subset\ambient$, and every $g\in L^2(\ambient,\meas)$,
\begin{equation}\label{e.rawgd}
 \left\lVert\wgrad P_t(\one_{E}g)\right\rVert_{L^2(\skeleton_w,\medm)}
 \leq C\Lambda(t){\exp(-c\Phi(\dist(B(w,4),E),t))}\norm{g}_{L^{2}(E,\meas)},
\end{equation}
where $\Lambda:(0,\infty)\to(0,\infty)$ is the function defined in \eqref{e.deflamb}.
\end{proposition}

\begin{proof}
	Let $\eta_{w}:=\eta_{w,1}$ be the cutoff function in Lemma \ref{l.cutoff} with $\epsilon=1$. Let $f:=\one_{E}g$. By Lemma \ref{l.wgrad}, \begin{equation}\label{e.cg1}
		\wgrad(\eta_{w}^{2}(P_{t}f-P_{t}f(w)))=\eta_{w}^{2}\wgrad P_{t}f+2\eta_{w}(P_{t}f-P_{t}f(w))\wgrad\eta_{w}.
	\end{equation}
	Multiplying both sides of \eqref{e.cg1} by $\wgrad P_{t}f$ and rearranging the terms, we obtain \begin{align}\label{e.cg1+}
		\eta_{w}^{2}\abs{\wgrad P_{t}f}^{2}=\wgrad(\eta_{w}^{2}(P_{t}f-P_{t}f(w)))\cdot\wgrad P_{t}f-2\eta_{w}(P_{t}f-P_{t}f(w))\wgrad\eta_{w}\wgrad P_{t}f.
	\end{align}
	Therefore, using Young's inequality $2ab\leq \frac{1}{2}a^{2}+2b^2$ in the first inequality below, we obtain\begin{align}
		&\phantom{\ \leq}\int_{\skeleton}\eta_{w}^{2}\abs{\wgrad P_{t}f}^{2}\dif\medm \overset{\eqref{e.cg1+}}{=}\form(\eta_{w}^{2}(P_{t}f-P_{t}f(w)),P_{t}f)-2\int_{\skeleton}\eta_{w}(P_{t}f-P_{t}f(w))\wgrad\eta_{w}\wgrad P_{t}f\dif\medm\\
		&=\langle-\gen P_{t}f,\eta_{w}^{2}(P_{t}f-P_{t}f(w))\rangle_{L^{2}(\ambient,\meas)}-2\int_{\skeleton}\eta_{w}\wgrad P_{t}f\cdot (P_{t}f-P_{t}f(w))\wgrad\eta_{w}\dif\medm\\
		&\leq \langle-\gen P_{t}f,\eta_{w}^{2}(P_{t}f-P_{t}f(w))\rangle_{L^{2}(\ambient,\meas)}\\
		&\qquad +\frac{1}{2}\int_{\skeleton}\eta_{w}^{2}\abs{\wgrad P_{t}f}^{2}\dif\medm+2\int_{\skeleton}(P_{t}f-P_{t}f(w))^{2}\abs{\wgrad\eta_{w}}^{2}\dif\medm.\label{e.cg2}
	\end{align}
	Since $\int_{\skeleton}\eta_{w}^{2}\abs{\wgrad P_{t}f}^{2}\dif\medm<\infty$, rearranging the terms in \eqref{e.cg2} gives \begin{equation}\label{e.cg2.1}
		\int_{\skeleton}\eta_{w}^{2}\abs{\wgrad P_{t}f}^{2}\dif\medm \leq 2\langle-\gen P_{t}f,\eta_{w}^{2}(P_{t}f-P_{t}f(w))\rangle_{L^{2}(\ambient,\meas)}+4\int_{\skeleton}(P_{t}f-P_{t}f(w))^{2}\abs{\wgrad\eta_{w}}^{2}\dif\medm.
	\end{equation}
	We estimate the terms on the right-hand side separately. We first note that $\supp_{\meas}[\eta_{w}]\subset B(w,4)$ and $\supp_{\medm}[\abs{\wgrad\eta_{w}}]\subset B(w,4)$. Let $R:=\dist(B(w,4),E)$.
	
		We first estimate $-\gen P_{t}f(x)$ for $x\in B(w,4)$. Note that $E\subset B(x,R)^{c}$ for every $x\in B(w,4)$.  By \eqref{e.kernel1}, and the fact that $f=\one_{E}f$, we have\begin{align}
		&\phantom{\ \leq}\abs{-\gen P_tf(x)}\\
		&\overset{\eqref{e.kernel1}}{\leq}\int_{E}\abs{\frac{\dif}{\dif t}p_t(z,x)}\abs{f(z)} \dif\meas(z)
			\overset{\eqref{e.ddt<}}{\lesssim}\int_{E}\frac{\exp\left(-c_{2}\Phi\left({\metric(z, x)},{t}\right)\right)}{t\volume(\scale^{-1}(t))} \abs{f(z)}\dif\meas(z)\\
			&\lesssim \frac{\norm{\exp\left(-c_{2}\Phi\left({\metric(\cdot, x)},{t}\right)\right)}_{L^{2}(B(x,R)^{c},\meas)}}{{t\volume(\scale^{-1}(t))}}
			 \norm{f}_{L^{2}(E,\meas)}\ \text{(Cauchy--Schwarz inequality)}\\
			&\overset{\eqref{e.intexp}}{\lesssim} \frac{\exp\left(-c_{3}\Phi\left(R,{t}\right)\right)}{{t\volume(\scale^{-1}(t))^{\frac{1}{2}}}}\norm{f}_{L^{2}(E,\meas)}\label{e.cg3}
		\end{align} 

		{We then estimate $\abs{P_{t}f(x)-P_{t}f(w)}$ for $x\in B(w,4)$.} By Proposition \ref{p.arc}, for every $x\in B(w,4)$, we have\begin{align}
			&\phantom{\ \leq}\abs{P_{t}f(x)-P_{t}f(w)}\\
			&\leq \int_{(x,w)}\abs{\wgrad P_{t}f(y)}\dif\medm(y)\overset{\eqref{e.kernel2}}{\leq} \int_{(x,w)}\int_{E}\abs{\wgrad p_{t,z}(y)}\abs{f(z)}\dif\meas(z)\dif\medm(y)\\
			&\leq \int_{(x,w)}\left(\int_{B(y,R)^{c}}\abs{\wgrad p_{t,z}(y)}^{2}\dif\meas(z)\right)^{\frac{1}{2}}
			\dif\medm(y)\norm{f}_{L^{2}(E,\meas)}\\
			&\overset{\Grad,\eqref{e.intexp}}{\lesssim} \int_{(x,w)}\frac{\exp(-c_{4}\Phi(R,t))}{\scale^{-1}(t)\volume(\scale^{-1}(t))^{\frac{1}{2}}}\dif\medm(y)\norm{f}_{L^{2}(E,\meas)}\\
			&=\frac{\exp(-c_{4}\Phi(R,t))}{\scale^{-1}(t)\volume(\scale^{-1}(t))^{\frac{1}{2}}}d(x,w)\norm{f}_{L^{2}(E,\meas)};\label{e.cg4}
		\end{align}
	here we use Proposition \ref{p.arc} in the first inequality, the fact that $(x,w)\subset B(w,4)$ (since $B(w,4)$ is a tree) and the Cauchy--Schwarz inequality in the third inequality; in the final equality, we use the fact that $\medm$ is a length measure on each tree.

	Combining these two estimates, we obtain \begin{align}
		&\phantom{\ \leq}\abs{\langle-\gen P_{t}f,\eta_{w}^{2}(P_{t}f-P_{t}f(w))\rangle_{L^{2}(\ambient,\meas)}}\\
		&\overset{\eqref{e.cg3},\eqref{e.cg4}}{\lesssim} \int_{B(w,4)}\frac{\exp\left(-c_{5}\Phi\left(R,{t}\right)\right)}{{t\scale^{-1}(t)\volume(\scale^{-1}(t))}}d(x,w)\dif\meas(x)\norm{f}_{L^{2}(E,\meas)}^{2}\\
		&\overset{\UVG}{\lesssim }\frac{\exp\left(-c_{5}\Phi\left(R,{t}\right)\right)}{{t\scale^{-1}(t)\volume(\scale^{-1}(t))}}\norm{f}_{L^{2}(E,\meas)}^{2},\label{e.cg5}
	\end{align}
	and \begin{align}
		&\phantom{\ \leq}\int_{\skeleton}(P_{t}f-P_{t}f(w))^{2}\abs{\wgrad\eta_{w}}^{2}\dif\medm\\
		&\overset{\eqref{e.cutoff1},\eqref{e.cg4}}{\lesssim} \int_{B(w,2)\cap\supp_{\medm}[\abs{\wgrad\eta_{w}}]}\frac{\exp(-c_{6}\Phi(R,t))}{\scale^{-1}(t)^{2}\volume(\scale^{-1}(t))}d(x,w)^{2}\norm{f}_{L^{2}(E,\meas)}^{2}\dif\medm(x)\\
		&\overset{\eqref{e.cutoff2}}{\lesssim}\frac{\exp(-c_{6}\Phi(R,t))}{\scale^{-1}(t)^{2}\volume(\scale^{-1}(t))}\norm{f}_{L^{2}(E,\meas)}^{2}.\label{e.cg6}
	\end{align}
	Combining Lemma \ref{l.cutoff}, \eqref{e.cg2.1}, \eqref{e.cg5}, and \eqref{e.cg6}, we obtain  \begin{align}
		\left\lVert\wgrad P_t(f)\right\rVert_{L^2(\skeleton_w,\medm)}^{2}&\leq \int_{\skeleton}\eta_{w}^{2}\abs{\wgrad P_{t}f}^{2}\dif\medm \ \text{ (since $\restr{\eta_{w}}{K_{w}}=1$)}\\
		&\lesssim \frac{\exp(-c_{7}\Phi(R,t))}{(\scale^{-1}(t)\wedge t)\scale^{-1}(t)\volume(\scale^{-1}(t))}\norm{f}_{L^{2}(E,\meas)}^{2}
	\end{align}
	This proves the claim.
	\end{proof}

\begin{corollary}\label{c.sebagd}
Assume $\UVG$, $\HKE$ and $\Grad$.

\begin{enumerate}[label=\textup{({\arabic*})},align=right,leftmargin=*,topsep=5pt,parsep=0pt,itemsep=2pt]
	\item\label{it.corgd1}  There exist constants $C,c\in(0,\infty)$ such that, for every $(v,R)\in\vertex\times[1,\infty)$ and every Borel set $E\subset\ambient$ such that $\metric_{\graph}(B_{\graph}(v,R),E)\geq 20$, the following estimate holds for every $t\in(0,\infty)$ and every $f\in L^2(\ambient,\meas)$,
\begin{equation}\label{e.sepgrad1}
 \lVert\wgrad P_t(\one_{E}f)
 \rVert_{L^2(\skeleton_{B_{\graph}(v,R)},\medm)}
 \leq C \volume(R)^{\frac{1}{2}}\Lambda(t){\exp(-c\Phi(\metric_{\graph}(B_{\graph}(v,R),E),t))}\norm{f}_{L^{2}(E,\meas)}.
\end{equation}
\item\label{it.corgd2} There exists $c\in(0,\infty)$ such that, for every $q\in[2,\infty)$, there exists $C_q\in(0,\infty)$ for which the following estimate holds for every $v\in\vertex$, every Borel set
$E\subset\ambient$, every $t\in(0,\infty)$, and every $f\in L^2(\ambient,\meas)$,
\begin{equation}\label{e.sepgrad2}
 \lVert\wgrad P_t(\one_{E}f)\rVert_{L^q(\skeleton_v,\medm)}\leq C_q
 \frac{\Lambda(t)^{\frac{2}{q}}\exp(-c q^{-1}\Phi(\dist(B(v,4),E),t))}{\scale^{-1}(t)^{1-\frac{2}{q}}\volume(\scale^{-1}(t))^{\frac{1}{2}-\frac{1}{q}}}\lVert f\rVert_{L^{2}(E,\meas)}.
\end{equation}
\end{enumerate}
 \end{corollary}

\begin{proof}
\begin{enumerate}[label=\textup{({\arabic*})},align=right,leftmargin=*,topsep=5pt,parsep=0pt,itemsep=2pt]
	\item[\ref{it.corgd1}] By Lemma \ref{l.c-dist}, if $\metric_{\graph}(B_{\graph}(v,R),E)\geq 20$,\begin{equation}
	\min_{w\in B_{\graph}(v,R)}\dist(B(w,4),E)\overset{\eqref{e.d-ball}}{\geq}\frac{1}{2}\inf_{w\in B_{\graph}(v,R)}\metric_{\graph}(w,E)-5\geq \frac{1}{4}\metric_{\graph}(B_{\graph}(v,R),E). 
\end{equation}
Replacing $f$ by $\one_Ef$, Proposition~\ref{p.cellgrad} and the partition $\skeleton_{B_{\graph}(v,R)}=\bigsqcup_{w\in B_{\graph}(v,R)}\skeleton_w$ give
\begin{align}
 \lVert\wgrad P_t(f)\rVert_{L^2(\skeleton_{B_{\graph}(v,R)},\medm)}^{2} &\overset{\eqref{e.rawgd}}{\lesssim} \# B_{\graph}(v,R)\Lambda(t)^{2}{\exp(-c\Phi(\metric_{\graph}(B_{\graph}(v,R),E),t))}\norm{f}_{L^{2}(E,\meas)}^{2}\\
 &\overset{\eqref{e.count}}{\lesssim}\volume(R) \Lambda(t)^{2}{\exp(-c\Phi(\metric_{\graph}(B_{\graph}(v,R),E),t))}\norm{f}_{L^{2}(E,\meas)}^{2}.
\end{align}
\item[\ref{it.corgd2}] By \eqref{e.meab1} and $\HKE$, we have
\begin{equation}
 |P_{t/2} (\one_{E}f)(x)|\leq p_{t}(x,x)^{\frac{1}{2}}\lVert (\one_{E}f)\rVert_{L^{2}(\ambient,\meas)}
 \overset{\HKE}{\lesssim}{\volume(\scale^{-1}(t))^{-\frac{1}{2}}} \lVert f\rVert_{L^{2}(E,\meas)},
\end{equation}
and hence 
\begin{equation}
	\lVert P_{t/2}(\one_{E}f)\rVert_{L^{\infty}(\ambient,\meas)}\lesssim {\volume(\scale^{-1}(t))^{-\frac{1}{2}}} \lVert f\rVert_{L^{2}(E,\meas)}.\label{e.q2norm1}
\end{equation}
Consequently,
\begin{align}
 &\phantom{\ \leq}\lVert\wgrad P_{t}(\one_{E}f)\rVert_{L^{\infty}(\skeleton,\medm)}
 =\lVert\wgrad P_{t/2}(P_{t/2} (\one_{E}f))\rVert_{L^{\infty}(\skeleton,\medm)}\\
 &\overset{\eqref{e.migrad}}{\lesssim} \scale^{-1}(t)^{-1}\lVert P_{t/2} (\one_{E}f)\rVert_{L^{\infty}(\ambient,\meas)}\overset{\eqref{e.q2norm1}}{\lesssim}\scale^{-1}(t)^{-1}\volume(\scale^{-1}(t))^{-\frac{1}{2}}\lVert f\rVert_{L^{2}(E,\meas)},\label{e.fd2r}
\end{align}
and
\begin{align}
 \lVert\wgrad P_t(\one_{E}f)\rVert_{L^q(\skeleton_v,\medm)}
& \leq 
 \lVert\wgrad P_t(\one_{E}f)\rVert_{L^2(\skeleton_v,\medm)}^{\frac{2}{q}}
 \lVert\wgrad P_t(\one_{E}f)\rVert_{L^{\infty}(\skeleton,\medm)}^{1-\frac{2}{q}}\\
& \overset{\eqref{e.rawgd},\eqref{e.fd2r}}{\lesssim} \frac{\Lambda(t)^{\frac{2}{q}}\exp(-c q^{-1}\Phi(\dist(B(v,4),E),t))}{\scale^{-1}(t)^{1-\frac{2}{q}}\volume(\scale^{-1}(t))^{\frac{1}{2}-\frac{1}{q}}}\lVert f\rVert_{L^{2}(E,\meas)}.
\end{align}
This completes the proof.
\qedhere
\end{enumerate}
\end{proof}

\begin{lemma}\label{l.lcons}
Assume $\UVG$, $\HKE$ and $\Grad$. Let $D_n\subset\vertex$, $n\in\bN$, be finite subsets of $\vertex$ such that $D_n\subset D_{n+1}$ and $\bigcup_{n\in\bN}D_n=\vertex$. Let $B\subset\vertex$ be non-empty and finite and suppose $\metric_{\graph}(B,\vertex\setminus D_n)\uparrow\infty$. Then
\begin{equation}\label{e.lcons}
 \lim_{n\to\infty}\left\lVert\wgrad P_t\one_{K_{D_n}}
 \right\rVert_{L^2(\skeleton_B,\medm)}=0,
 \ \text{ for all }t\in(0,\infty).
\end{equation}
\end{lemma}

\begin{proof}
Fix $t\in(0,\infty)$. By $\Grad$, Remark \ref{r.Grad} and the estimate in \eqref{e.intexp}, we have $\int_{\ambient}\left\lvert\wgrad p_{t,x}(y)\right\rvert\dif\meas(x)<\infty$ for all $y\in\skeleton$. Let $[a,b]$ be a compact geodesic. Since
$\medm([a,b])=\metric(a,b)<\infty$, Fubini's theorem gives, for $z\in[a,b]$,
\begin{equation}\label{e.ckergd}
 \begin{aligned}
0\overset{\eqref{e.conserv}}{=}P_t\one_{\ambient}(z)-P_t\one_{\ambient}(a)
 &=\int_{\ambient}\bigl(p_t(x,z)-p_t(x,a)\bigr)\dif\meas(x)\\
 &=\int_{[a,z]}\sigma_{a,z}(y)
 \left(\int_{\ambient}\wgrad p_{t,x}(y)\dif\meas(x)\right)
 \dif\medm(y),
 \end{aligned}
\end{equation}
and hence
\begin{equation}\label{e.lcons1}
 \int_{\ambient}\wgrad p_{t,x}(y)\dif\meas(x)=0
 \quad\text{for $\medm$-a.e. $y\in\skeleton$}.
\end{equation}
Since $D_n$ is finite, $\one_{K_{D_n}}\in L^2(\ambient,\meas)$, and the same argument gives
\begin{equation}\label{e.lcons2}
 \wgrad P_t\one_{K_{D_n}}(y)
 =\int_{K_{D_n}}\wgrad p_{t,x}(y)\dif\meas(x),
 \ \text{for $\medm$-a.e. $y\in\skeleton$}.
\end{equation}
Combining \eqref{e.lcons1} and
\eqref{e.lcons2}, we obtain 
\begin{equation}\label{e.lcons3}
 \wgrad P_t\one_{K_{D_n}}(y)
 =-\int_{\ambient\setminus K_{D_n}}\wgrad p_{t,x}(y)\dif\meas(x),\ \text{for $\medm$-a.e. $y\in\skeleton$}.
\end{equation}
The integral on the right is absolutely convergent by
$\Grad$ and \eqref{e.intexp}.
Set $r_n:=\metric_{\graph}(B,\vertex\setminus D_n)$, and partite $\vertex\setminus D_n$ into 
\begin{equation}\label{e.lcons4}
 A_{n,k}:=\set{w\in\vertex\setminus D_n:
 k\leq\metric_{\graph}(B,w)<k+1},
 \  k\in\bN\cup\set{0}.
\end{equation}
For each $n\in\bN$, the sets $\{A_{n,k}\}_{k\in\bN\cup\set{0}}$ are pairwise disjoint, $\bigcup_{k\in\bN\cup\set{0}}A_{n,k}=\vertex\setminus D_n$, and
$A_{n,k}=\emptyset$ whenever $k<r_n$. 
By Proposition~\ref{p.count},
\begin{equation}\label{e.lcons5}
 \left\lVert \one_{K_{A_{n,k}}}\right\rVert_{L^2(\ambient,\meas)}^2
 =\meas(K_{A_{n,k}})
 \overset{\eqref{e.count}}{\lesssim}\volume(1+k+\diam(B,\metric_{\graph})).
\end{equation}
Fix $v\in B$. For $k\geq20$,
\begin{equation}\label{e.lcons6}
 \lVert\wgrad P_t\one_{K_{A_{n,k}}}\rVert_{L^2(\skeleton_v,\medm)}^{2}
 \overset{\eqref{e.sepgrad1},\eqref{e.lcons5}}{\lesssim}\Lambda(t)^{2}{\exp({-c\Phi(k,t)})\volume(1+k+\diam(B,\metric_{\graph}))}.
\end{equation}
Summing \eqref{e.lcons6} over $v\in B$, we have
\begin{equation}\label{e.lcons7}
 \left\lVert\wgrad P_t\one_{K_{A_{n,k}}}\right\rVert_{L^2(\skeleton_B,\medm)}
 \lesssim (\# B)^{\frac{1}{2}}\Lambda(t)\volume(1+k+\diam(B,\metric_{\graph}))^{\frac{1}{2}}\exp(-c_{1}\Phi(k,t)),
 \ \text{ for all }k\geq20.
\end{equation}
By Lemma \ref{l.phi}, we have
\begin{equation}\label{e.lcons8}
 \sum_{k=20}^{\infty}
\volume(1+k+\diam(B,\metric_{\graph}))^{\frac{1}{2}}\exp(-c_{1}\Phi(k,t))<\infty.
\end{equation}
By \eqref{e.lcons3}, $\wgrad P_t\one_{K_{D_n}}=-\sum_{k\geq r_{n}}\wgrad P_t\one_{K_{A_{n,k}}}$, and hence \begin{align}
	&\phantom{\ \leq}\left\lVert\wgrad P_t\one_{K_{D_n}}
 \right\rVert_{L^2(\skeleton_B,\medm)}\leq \sum_{k\geq r_{n}}\left\lVert\wgrad P_t\one_{K_{A_{n,k}}} \right\rVert_{L^2(\skeleton_B,\medm)}\\
 &\overset{\eqref{e.lcons7}}{\lesssim}(\# B)^{\frac{1}{2}}\Lambda(t)\sum_{k\geq r_{n}}\volume(1+k+\diam(B,\metric_{\graph}))^{\frac{1}{2}}\exp(-c_{1}\Phi(k,t))\\
 &\to 0 \text{ (as $n\to\infty$, by \eqref{e.lcons8})}.
\end{align}
This proves \eqref{e.lcons}.
\end{proof}
\begin{lemma}\label{l.trunid}
Assume $\UVG$, $\HKE$ and $\Grad$. Let $B\subset\vertex$ be finite, $u\in\domain$, $c\in\bR$, and
let $\{A_j\}_{j\in\bN}$ be a disjoint partition of $\vertex$ into finite sets. Let $u_j:=(u-c)\one_{K_{A_j}}$. If, for some $t\in(0,\infty)$ and $q\in[2,\infty)$, $\sum_{j\in\bN}
 \lVert\wgrad P_tu_j\rVert_{L^q(\skeleton_B,\medm)}<\infty$, then 
\begin{equation}\label{e.trunid}
 \wgrad P_tu=\sum_{j\in\bN}\wgrad P_tu_j
 \quad\text{in }L^q(\skeleton_B,\medm).
\end{equation}
\end{lemma}

\begin{proof}
Since $\sum_{j=1}^{N}u\one_{K_{A_{j}}}\to u$ in $L^2(\ambient,\meas)$, spectral calculus gives $\wgrad P_t(\sum_{j=1}^{N}u\one_{K_{A_{j}}})\to\wgrad P_tu$ in $L^2(\skeleton,\medm)$.
Local finiteness of $\graph$ gives$d_{\graph}(B,\vertex\setminus \bigcup_{j=1}^{N}A_{j})\to\infty$ as $N\to\infty$, and Lemma~\ref{l.lcons} gives $\wgrad P_t(c\sum_{j=1}^{N}\one_{K_{A_{j}}})\to0$ in $L^2(\skeleton_B,\medm)$. Thus the partial sums in \eqref{e.trunid} converge to $\wgrad P_tu$ in $L^2(\skeleton_B,\medm)$. The assumption $\sum_{j\in\bN} \lVert\wgrad P_tu_j\rVert_{L^q(\skeleton_B,\medm)}<\infty$ implies convergence in $L^q(\skeleton_B,\medm)$.  A subsequence converges $\medm$-a.e. to both limits; hence the limits agree.
\end{proof}

\subsection{Graph Poincar\'{e} inequality}

It is well-known that the volume doubling condition, which is implied by $\UVG$, and the heat kernel estimates $\HKE$ give a \emph{Poincar\'{e} inequality}. To be precise, we let \begin{equation}\label{e.localformdomain}
\domain_{\loc}:=\left\{u\in L^2_{\loc}(\ambient,\meas):\begin{array}{l}\text{for every relatively compact open }U\subset\ambient,\\\text{there exists }v\in\domain\text{ such that }u=v\ \meas\text{-a.e. on }U\end{array}\right\}.
\end{equation}
\begin{proposition}
	Assume $\UVG$ and $\HKE$. There exist constants $C_{\mathrm{PI}}\in(0,\infty)$ and $A_{\mathrm{PI}}\in[1,\infty)$ such that, for all $(x,r)\in\ambient\times(0,\infty)$ and all $f\in\domain_{\loc}$, \begin{equation}\label{e.PI}
		\int_{B(x,r)}\abs{f-f_{B(x,r)}}^{2}\dif\meas\leq C_{\mathrm{PI}} \scale(r)\int_{B(x,A_{\mathrm{PI}}r)}\abs{\wgrad f}^{2}\dif\medm,
	\end{equation}
	where $f_{B(x,r)}:=\frac{1}{\meas({B(x,r)})}\int_{B(x,r)}f\dif\meas$.
\end{proposition}
\begin{proof}
		This follows from \cite[Proof of Theorem~3.2]{Lie15} (see also \cite[Remark~2.9-(b)]{KM20}).
\end{proof}

Define the \emph{uncentered graph maximal operator} by
\begin{equation}\label{e.ghmax}
 M_{\graph}f(z)=\sup\Sett{\frac{1}{\#B}
 \sum_{q\in B}\lvert f(q)\rvert}{\text{$B$ is a non-empty ball in $(\vertex,d_{\graph})$ that contains $z$}},
\end{equation}
where the supremum is over finite graph balls.

\begin{proposition}
\label{p.ghpPII}
Assume $\UVG$ and $\HKE$.

\begin{enumerate}[label=\textup{({\arabic*})},align=right,leftmargin=*,topsep=5pt,parsep=0pt,itemsep=2pt]
	\item\label{it.pi1} There are constants $C,A\in(1,\infty)$ such that, for every graph ball
$B=B_{\graph}(v,R)$ with $(v,R)\in \vertex\times[1,\infty)$ and every
$u\in\domain_{\loc}$,
\begin{equation}\label{e.ghpPI}
 \int_{K_B}\lvert u-u_{K_B}\rvert^2\dif\meas
 \leq C\scale(R)
 \sum_{w\in B_{\graph}(v,AR)}
       \left\lVert\wgrad u\right\rVert_{L^2(\skeleton_w,\medm)}^2.
\end{equation}
\item\label{it.pi2} There exists $C\in(1,\infty)$ such that the following statement holds. For every $v\in\vertex$ and every pair of numbers $1\leq r\leq R<\infty$, write $B_{r}=B_{\graph}(v,r)$ and $B_{R}=B_{\graph}(v,R)$. Then 
\begin{equation}\label{e.nestPI}
 \left\lVert u-u_{K_{B_{r}}}\right\rVert_{L^2(K_{B_{R}},\meas)}
 \leq C\scale(R)^{\frac{1}{2}}\volume(R)^{\frac{1}{2}}
\min_{q\in B_{r}} (M_{\graph}E_u)(q)^{\frac{1}{2}},\ \text{for all }u\in\domain_{\loc},
\end{equation}
where $E_u(q):=\left\lVert\wgrad u\right\rVert_{L^2(\skeleton_{q},\medm)}^2$, $q\in\vertex$.
\end{enumerate}
\end{proposition}

\begin{proof}
\begin{enumerate}[label=\textup{({\arabic*})},align=right,leftmargin=*,topsep=5pt,parsep=0pt,itemsep=2pt]
	\item[\ref{it.pi1}]

Lemma~\ref{l.sanwich} and \eqref{e.quaiso} give $K_B\subset B(v,(3R+1))\subset B(v,4R)$. By the Poincar\'{e} inequality \eqref{e.PI} for $(\ambient,\metric,\meas,\form,\domain)$,
 \begin{equation}\label{e.gpi1}
 \int_{K_B}\lvert u-u_{K_B}\rvert^2\dif\meas\leq\int_{B(v,4R)}\lvert u-u_{B(v,4R)}\rvert^2\dif\meas\overset{\eqref{e.PI}}{\leq} C_{\mathrm{PI}}{\scale(4R)} \int_{B(v,4A_{\mathrm{PI}}R)\cap\skeleton}
       \lvert\wgrad u\rvert^2\dif\medm.
 \end{equation}

For every $x\in B(v,4A_\mathrm{PI}R)\cap\skeleton$, there exists a unique $w\in\vertex$ such that $x\in \skeleton_{w}$. Since $R\in[1,\infty)$, \begin{equation}
	\metric_{\graph}(v,w)\overset{\eqref{e.quaiso}}{\leq}2\metric(v,w)\leq  2(\metric(v,x)+\metric(x,w))\overset{\eqref{e.sanwich}}{\leq} 2(4A_{\mathrm{PI}}R+1)\leq(8A_{\mathrm{PI}}+3)R.
\end{equation}
This implies that $B(v,4A_\mathrm{PI} R)\cap\skeleton\subset \bigcup_{w\in B_{\graph}(v,(8A_{\mathrm{PI}}+3)R)}\skeleton_{w}$. Combining this inclusion with \eqref{e.gpi1} proves the assertion.
\item[\ref{it.pi2}]
Let $N\in\bN$ such that $R\in[2^{N-1}r,2^{N}r]$. With a slight abuse of notation, set $B_k:=B_{\graph}(v,2^kr)$.  By \ref{it.pi1} and $\UVG$, for $k\in\{0,\ldots,N\}$ and every $q\in B_{r}$,
 \begin{equation}
 \left\lVert u-u_{K_{B_k}}\right\rVert_{L^2(K_{B_k})}^{2}
 \overset{\eqref{e.ghpPI}}{\lesssim} {\scale(2^k r)}
\sum_{w\in B_{\graph}(v,A2^k r)}E_u(w)\overset{\eqref{e.count}}{\lesssim} \scale(2^k r){\volume(2^k r)}
 (M_{\graph}E_u)(q)\label{e.PIescale}
 \end{equation}
Therefore, for every $q\in B_{r}$,
 \begin{align}
 \lvert u_{K_{B_{k+1}}}-u_{K_{B_k}}\rvert
 &\leq 
 \fint_{K_{B_k}}\lvert u-u_{K_{B_{k+1}}}\rvert\dif\meas\leq \meas(K_{B_k})^{-\frac{1}{2}}
 \left\lVert u-u_{K_{B_{k+1}}}\right\rVert_{L^2(K_{B_{k+1}})}\\
 &\overset{\eqref{e.PIescale}}{\lesssim} \scale(2^kr)^{\frac{1}{2}}(M_{\graph}E_u)(q)^{\frac{1}{2}},\label{e.PImax2}
 \end{align}
Therefore,
\begin{align}
	 &\phantom{\ \leq}\left\lVert u-u_{K_{B_{r}}}\right\rVert_{L^2(K_{B_{R}},\meas)}\leq \left\lVert u-u_{K_{B_{0}}}\right\rVert_{L^2(K_{B_{N}},\meas)}\\
	 &\leq \left\lVert u-u_{K_{B_{N}}}\right\rVert_{L^2(K_{B_{N}},\meas)}+\left\lVert u_{K_{B_{0}}}-u_{K_{B_{N}}}\right\rVert_{L^2(K_{B_{N}},\meas)}\\
	 &\overset{\eqref{e.PIescale}}{\lesssim} \scale(2^N r)^{\frac{1}{2}}\volume(2^Nr)^{\frac{1}{2}}(M_{\graph}E_u)(q)^{\frac{1}{2}}+\meas(K_{B_{N}})^{\frac{1}{2}}\sum_{k=0}^{N-1}\lvert u_{K_{B_{k+1}}}-u_{K_{B_k}}\rvert\\
	 &\overset{\eqref{e.PImax2}}{\lesssim} \scale(2^N r)^{\frac{1}{2}}\volume(2^N r)^{\frac{1}{2}} (M_{\graph}E_u)(q)^{\frac{1}{2}}\left(1+\sum_{k=0}^{N-1}\frac{\scale(2^k r)^{\frac{1}{2}}}{\scale(2^N r)^{\frac{1}{2}}}\right)\\
	 &\overset{\eqref{e.scale}}{\lesssim} \scale(2^N  r)^{\frac{1}{2}}\volume(2^N r)^{\frac{1}{2}} (M_{\graph}E_u)(q)^{\frac{1}{2}}\left(1+\sum_{k=0}^{N-1}2^{\beta_{1}(k-N)/2}\right)\\
	 &\lesssim \scale(R)^{\frac{1}{2}}\volume(R)^{\frac{1}{2}}(M_{\graph}E_u)(q)^{\frac{1}{2}}\ \text{(since $\textstyle \sum_{k=0}^{\infty}2^{-\beta_{1}k/2}<\infty$)}.
\end{align}
Taking the minimum over all $q\in B_{r}$, we obtain \eqref{e.nestPI}.
\qedhere
\end{enumerate}
\end{proof}
\begin{lemma}\label{l.utsmth}
Assume $\UVG$, $\HKE$ and $\Grad$.  For any $p\in(2,\infty)$ and $t\in(0,\infty)$, there exists $C_{p}(t)\in(0,\infty)$, such that
\begin{equation}\label{e.utsmth}
 \lVert\wgrad P_{t}u\rVert_{L^p(\skeleton,\medm)}
 \leq C_{p}(t)
 \left(\sum_{v\in\vertex}
 \lVert\wgrad u\rVert_{L^2(\skeleton_v,\medm)}^p\right)^{\frac{1}{p}},\ \text{ for all }u\in\domain.
\end{equation}
\end{lemma}

\begin{proof}
Define $E_u(q):=\lVert\wgrad u\rVert_{L^2(\skeleton_q,\medm)}^2$ for $q\in\vertex$. Fix $v\in\vertex$. Define $A_0:=B_{\graph}(v,2^{10})$ and $A_j:=B_{\graph}(v,2^{j+10})\setminus B_{\graph}(v,2^{j+9})$ for $j\in\bN$. For $j\in\bN$, \eqref{e.scale}, $\UVG$, and Proposition \ref{p.ghpPII}-\ref{it.pi2} give
\begin{equation}
 \lVert u-u_{K_{v}}\rVert_{L^{2}(K_{A_{j}},\meas)}\leq \lVert u-u_{K_{v}}\rVert_{L^{2}(K_{B_{\graph}(v,2^{j+10})},\meas)}\overset{\eqref{e.nestPI}}{\lesssim} 
 \scale(2^j)^{\frac{1}{2}}\volume(2^j)^{\frac{1}{2}}(M_{\graph}E_u)(v)^{\frac{1}{2}}.\label{e.smoot1}
\end{equation}
For $j\in\bN$, we have $\dist(B(v,4),K_{A_{j}})\overset{\eqref{e.d-ball}}{\geq} \frac{1}{2}\metric_{\graph}(v,K_{A_{j}})-5\gtrsim 2^{j}$, which, combined with Corollary \ref{c.sebagd}-\ref{it.corgd2}, gives
\begin{align}
 &\phantom{\ \leq}\lVert\wgrad P_{t}((u-u_{K_{v}})\one_{K_{A_{j}}})\rVert_{L^p(\skeleton_v,\medm)}\\
& \overset{\eqref{e.sepgrad2}}{\lesssim} \frac{\Lambda(t)^{\frac{2}{p}}\exp(-c p^{-1}\Phi(\dist(B(v,4),K_{A_{j}}),t))}{\scale^{-1}(t)^{1-\frac{2}{p}}\volume(\scale^{-1}(t))^{\frac{1}{2}-\frac{1}{p}}}\lVert (u-u_{K_{v}})\rVert_{L^{2}(K_{A_{j}},\meas)}\\
&\overset{\eqref{e.smoot1}}{\lesssim} \frac{\Lambda(t)^{\frac{2}{p}}}{\scale^{-1}(t)^{1-\frac{2}{p}}\volume(\scale^{-1}(t))^{\frac{1}{2}-\frac{1}{p}}}\exp(-c_{1}p^{-1}\Phi(2^{j},t))
 \scale(2^j)^{\frac{1}{2}}{\volume(2^j)^{\frac{1}{2}}}(M_{\graph}E_u)(v)^{\frac{1}{2}}.\label{e.smoot2}
\end{align}
For $j=0$, \eqref{e.nestPI} and Corollary~\ref{c.sebagd}-\ref{it.corgd2} give the same estimate, since the exponential term can be absorbed in the constant $C_p(t)$. By \eqref{e.smoot1}, \eqref{e.smoot2}, \eqref{e.scale}, $\UVG$ and the exponential decay of $\exp(-c_{1}p^{-1}\Phi(2^{j},t))$, there exists a constant $C_{p}(t)\in(0,\infty)$ such that \begin{equation}\label{e.smoot2+}
	\sum_{j=0}^{\infty}\lVert\wgrad P_{t}((u-u_{K_{v}})\one_{K_{A_{j}}})\rVert_{L^p(\skeleton_v,\medm)}\leq C_{p}(t)(M_{\graph}E_u)(v)^{\frac{1}{2}}<\infty.
\end{equation}
By Lemma~\ref{l.trunid},
\begin{equation}\label{e.utcell}
 \lVert\wgrad P_{t}u\rVert_{L^p(\skeleton_v,\medm)}
 \leq \sum_{j=0}^{\infty}\lVert\wgrad P_{t}((u-u_{K_{v}})\one_{K_{A_{j}}})\rVert_{L^p(\skeleton_v,\medm)}\overset{\eqref{e.smoot2+}}{\leq} C_{p}(t)(M_{\graph}E_u)(v)^{\frac{1}{2}}.
\end{equation}
Since $\frac{p}{2}>1$, and $(\vertex,d_{\graph},\#)$ is doubling by \eqref{e.count} and \eqref{e.VD}, the boundedness of the maximal operator $M_{\graph}$ on $\ell^{\frac{p}{2}}(\vertex)$ \cite[Theorem~2.2]{Hei01} gives
\begin{align}
 \lVert\wgrad P_{t}u\rVert_{L^p(\skeleton,\medm)}^p
=\sum_{v\in\vertex}\lVert\wgrad P_{t}u\rVert_{L^p(\skeleton_v,\medm)}^p&\overset{\eqref{e.utcell}}{\leq} C_{p}(t)\sum_{v\in\vertex}(M_{\graph}E_{u}(v))^{\frac{p}{2}}
= C_{p}(t)\norm{M_{\graph}E_{u}}_{\ell^{\frac{p}{2}}(\vertex)}^{\frac{p}{2}}\\
&\leq C_{p}(t)\norm{E_{u}}_{\ell^{\frac{p}{2}}(\vertex)}^{\frac{p}{2}}=C_{p}(t)\sum_{v\in\vertex}
 \lVert\wgrad u\rVert_{L^2(\skeleton_v,\medm)}^p.
\end{align}
This proves \eqref{e.utsmth}.
\end{proof}

\begin{lemma}
Assume $\UVG$, $\HKE$ and $\Grad$. There exists a constant $C\in(0,\infty)$ such that, for every $(v,R)\in\vertex\times[1,\infty)$ and every $u\in\domain$, we have \begin{equation}\label{e.harhb}
		\sup_{q\in B_{\graph}(v,R)}\norm{\wgrad P_{\scale(R)} u}_{L^{2}(\skeleton_{q},\medm)}\leq CR^{-1}\scale(R)^{\frac{1}{2}}\min_{q\in B_{\graph}(v,R)}(M_{\graph}E_{u})(q)^{\frac{1}{2}},
	\end{equation}
	where $E_{u}(q):=\left\lVert \wgrad u\right\rVert_{L^2(\skeleton_q,\medm)}^2$ for $q\in\vertex$.
\end{lemma}
\begin{proof}
	Define $C_{0}:=B_{\graph}(v,2^{5}R)$ and $C_{j}:=B_{\graph}(v,2^{j+5}R)\setminus B_{\graph}(v,2^{j+4}R)$ for $j\in\bN$.
	Let $u_{j}:=(u-u_{K_{B_{\graph}(v,R)}})\one_{K_{C_{j}}}$ for $j\in\{0\}\cup\bN$. By Proposition \ref{p.ghpPII}-\ref{it.pi2}, \begin{align}
		\norm{u_{j}}_{L^{2}(K_{C_{j}},\meas)}&\leq \norm{u-u_{K_{B_{\graph}(v,R)}}}_{L^{2}(K_{B_{\graph}(v,2^{j+7}R)},\meas)}\\
		&\lesssim\scale(2^{j}R)^{\frac{1}{2}}{\volume(2^{j}R)^{\frac{1}{2}}}\min_{w\in B_{\graph}(v,R)}(M_{\graph}E_{u})(w)^{\frac{1}{2}}.\label{e.harhb0}
	\end{align}
	For $j\in\bN$ and every $q\in B_{\graph}(v,R)$, \begin{align}
		&\phantom{\ \leq}\norm{\wgrad P_{\scale(R)}u_{j}}_{L^{2}(\skeleton_{q},\medm)}\\&\overset{\eqref{e.rawgd}}{\lesssim} \Lambda(\scale(R))\exp(-c\Phi(\dist(B(q,4),K_{C_{j}}),\scale(R)))\norm{u_{j}}_{L^{2}(K_{C_{j}},\meas)}\\
		&\overset{\eqref{e.harhb0},\eqref{e.phi1}}{\lesssim} \Lambda(\scale(R))\exp(-c_{1}2^{\frac{j\beta_{1}}{\beta_{2}-1}})\scale(2^{j}R)^{\frac{1}{2}}{\volume(2^{j}R)^{\frac{1}{2}}}\min_{q\in B_{\graph}(v,R)}(M_{\graph}E_{u})(q)^{\frac{1}{2}}\\
		&\overset{\eqref{e.growth1},\eqref{e.scale},\eqref{e.VD1}}{\lesssim} R^{-1}\scale(R)^{\frac{1}{2}}\exp(-c_{1}2^{\frac{j\beta_{1}}{\beta_{2}-1}})2^{\frac{j(\beta_{2}+d_{2})}{2}}\min_{w\in B_{\graph}(v,R)}(M_{\graph}E_{u})(w)^{\frac{1}{2}}.\label{e.harhb1}
	\end{align}
	For $j=0$, the same bound follows from \eqref{e.rawgd} and \eqref{e.harhb0}; the term $\exp(-c_1)$ is absorbed in the constant. Here \eqref{e.growth1} gives $R\wedge\scale(R)\asymp R$ for $R\in[1,\infty)$. Since $B_{\graph}(v,R)$ is finite and the bound is uniform in $q$ in this ball, the series converges absolutely in $L^2(\skeleton_{B_{\graph}(v,R)},\medm)$. Therefore,
\begin{align}
		\sup_{q\in B_{\graph}(v,R)}\norm{\wgrad P_{\scale(R)} u}_{L^{2}(\skeleton_{q},\medm)}
		&\overset{\eqref{e.trunid}}{\lesssim}\sup_{q\in B_{\graph}(v,R)}\sum_{j=0}^{\infty}\norm{\wgrad P_{\scale(R)}u_{j}}_{L^{2}(\skeleton_{q},\medm)}\\
		&\overset{\eqref{e.harhb1}}{\lesssim} R^{-1}\scale(R)^{\frac{1}{2}}\min_{w\in B_{\graph}(v,R)}(M_{\graph}E_{u})(w)^{\frac{1}{2}}.
	\end{align}
	This completes the proof of \eqref{e.harhb}.
\end{proof}

\section{Boundedness of the Riesz transforms}\label{s.riesz}
To prove the boundedness of the Riesz transform $\riesz$, we define
\begin{equation}
	\riesz_{\high}:=\riesz\circ (I-P_{1}),\ \text{ and }\ \riesz_{\low}:=\riesz\circ P_{1}.
\end{equation}
Thus, $\riesz=\riesz_{\high}+\riesz_{\low}$. 
\subsection{Estimates for the low-frequency part \texorpdfstring{$\riesz_{\low}$}{R\_{♭}}}
The following proposition is an adaptation of \cite[Theorem~2.1 and Lemmas~2.2--2.3]{ACDH04}, which originates in \cite{Mar04}.
\begin{proposition}\label{p.ACDH}
Assume $\UVG$. Let $T:L^2(\ambient,\meas)\to[0,\infty)^{\vertex}\cap\ell^2(\vertex)$ be a sublinear bounded map. Let $A_{r}:L^2(\ambient,\meas)\to L^2(\ambient,\meas)$, $r\in(0,\infty)$, be a family of bounded linear operators. Suppose that there exists $C_{0}\in(1,\infty)$ such that the following two conditions hold. \begin{enumerate}[label=\textup{({H\alph*})},align=right,leftmargin=*,topsep=5pt,parsep=0pt,itemsep=2pt]
	\item\label{it.ha} For every $(v,R)\in\vertex\times[1,\infty)$ and every $f\in L^2(\ambient,\meas)$, one has 
\begin{equation}\label{e.ACDH-1}
 \frac{1}{\#B_{\graph}(v,R)}\sum_{w\in B_{\graph}(v,R)}
 \lvert T(I-A_{R})f(w)\rvert^2
 \leq C_0 \inf_{z\in B_{\graph}(v,R)}(M_{\graph}F_{f})(z),
\end{equation}
where $F_{f}(v):=\left\lVert f\right\rVert_{L^2(K_v,\meas)}^2$ for $v\in\vertex$.
\item\label{it.hb} For every $(v,R)\in\vertex\times[1,\infty)$ and every $f\in L^2(\ambient,\meas)$, one has \begin{equation}\label{e.ACDH-2}
 \sup_{w\in B_{\graph}(v,R)} (T(A_{R}f))(w)
 \leq C_0\inf_{z\in B_{\graph}(v,R)}
M_{\graph}((Tf)^2)(z)^{\frac{1}{2}}.
\end{equation}
\end{enumerate}
Then, for every $p\in(2,\infty)$, there exist constants $C_{p},C_{p}^{\prime}\in(0,\infty)$, depending only on $p$, $C_{0}$, $\volume(1)$, and the constants in $\UVG$, such that
\begin{equation}\label{e.ACDHcl}
 \left\lVert Tf\right\rVert_{\ell^p(\vertex)}
 \leq C_p  \left(\sum_{v\in\vertex}
 \left\lVert f\right\rVert_{L^2(K_v,\meas)}^p\right)^{\frac{1}{p}}
 \leq C_{p}^{\prime}\left\lVert f\right\rVert_{L^p(\meas)},\ \text{ for every $f\in L^2(\ambient,\meas)$}.
\end{equation}
\end{proposition}
\begin{proof}
Define
\begin{equation}\label{e.spfunc}
 M_{T,A}^{\#}f(z)
 :=\sup_{\substack{(v,r)\in\vertex\times[1,\infty)\\ z\in B_{\graph}(v,r)}}
 \left(
 \frac{1}{\#B_{\graph}(v,r)}
 \sum_{w\in B_{\graph}(v,r)}
 \lvert T(I-A_{r})f(w)\rvert^2
 \right)^{\frac{1}{2}},\ z\in\vertex.
\end{equation}
For each ball $B_{\graph}(v,r)$ containing $z$, condition \ref{it.ha}, namely \eqref{e.ACDH-1}, gives
\[
 \frac{1}{\#B_{\graph}(v,r)}
 \sum_{w\in B_{\graph}(v,r)}
 \lvert T(I-A_{r})f(w)\rvert^2
 \leq C_0(M_{\graph}F_f)(z).
\]
Taking the supremum over all such balls yields
\begin{equation}\label{e.spctol}
 M_{T,A}^{\#}f(z)
 \leq C_0^{\frac{1}{2}}(M_{\graph}F_f)(z)^{\frac{1}{2}},
 \ \text{for every } z\in\vertex.
\end{equation}

We next prove a \emph{local good-$\lambda$ estimate}. Fix $\vartheta=1/100$. Let $\Omega\subsetneq\vertex$ be a nonempty subset. Let $\{B_{\graph}(x_i,r_i)\}_{i\in I}$ be a $\vartheta$-Whitney covering of $\Omega$, provided by Proposition \ref{p.Whitney}.

Fix $B_{j}=B_{\graph}(x_{j},r_{j})$ and write
$\delta_{j}=\metric_{\graph}(x_{j},\vertex\setminus\Omega)$. Set
$\wt B_{j}:=B_{\graph}(x_{j},2\delta_{j})$. Since the graph is locally finite, there exists $z_{j}\in\vertex\setminus\Omega$ such that
$\metric_{\graph}(x_{j},z_{j})=\delta_{j}$. In particular, $z_{j}\in\wt B_{j}$; moreover,
$3B_{j}\subset\wt B_{j}$ because $3r_{j}=3\vartheta\delta_{j} /(1+\vartheta)<2\delta_{j}$.

For $j\in I$, $\lambda,K\in(0,\infty)$, and $\gamma\in(0,1)$, define \begin{equation}\label{e.defsG}
	\sG(\Omega,j,\lambda,K,\gamma):=\Sett{z\in3B_j}{
 M_{\graph}((Tf)^2)(z)>K^2\lambda \text{ and }
 M_{T,A}^{\#}f(z)\leq\gamma\lambda^{\frac{1}{2}}}.
\end{equation}

\begin{enumerate}[label=\textit{Claim {\arabic*}},align=right,leftmargin=*,topsep=5pt,parsep=0pt,itemsep=2pt]
	\item\label{it.ACDHc1} \textit{There exists $K_{0}\in(1,\infty)$ such that, for every $j\in I$ and $\lambda\in(0,\infty)$, if $M_{\graph}((Tf)^2)(z_{j})\leq\lambda$, if $z\in 3B_{j}$, and if a metric ball $B\subset\vertex$ containing $z$ satisfies  $\frac{1}{\#B}
 \sum_{w\in B}\lvert Tf(w)\rvert^2> K_{0}^{2}\lambda$, then $B\subset \wt{B}_{j}$.}

Suppose $j\in I$ and $\lambda\in(0,\infty)$ such that $M_{\graph}((Tf)^2)(z_{j})\leq\lambda$. Let $z\in 3B_{j}$ and $B=B_{\graph}(y,\rho)\ni z$. Let $K_{0}\in(1,\infty)$ be determined below. Suppose $\frac{1}{\# B_{\graph}(y,\rho)}
 \sum_{w\in B_{\graph}(y,\rho)}\lvert Tf(w)\rvert^2> K_{0}^{2}\lambda$. Suppose, to the contrary, that $B_{\graph}(y,\rho)\setminus\wt B_{j}\neq\emptyset$, and choose $q\in B_{\graph}(y,\rho)\setminus\wt B_{j}$. Since
$\metric_{\graph}(z,x_j)<3r_j=3\vartheta\delta_j/(1+\vartheta)$ and
$\metric_{\graph}(q,x_j)\geq2\delta_j$, we have
\[
 2\rho> \metric_{\graph}(z,y)+\metric_{\graph}(y,q)\geq \metric_{\graph}(z,q)
 \geq \metric_{\graph}(q,x_j)-\metric_{\graph}(z,x_j)
 >\frac{2-\vartheta}{1+\vartheta}\delta_j.
\]
Thus $\delta_j<2(1+\vartheta)(2-\vartheta)^{-1}\rho$. Moreover,
\[
 \metric_{\graph}(z_j,y)
 \leq \metric_{\graph}(z_j,x_j)+\metric_{\graph}(x_j,z)+d_\graph(z,y)
 <\delta_j+\frac{3\vartheta}{1+\vartheta}\delta_j+\rho<\frac{4+7\vartheta}{2-\vartheta}\rho<3\rho.
\]
Hence $z_j\in B_{\graph}(y,3\rho)$. By the doubling property of $\graph$, the definition of $M_{\graph}$, the fact that $z_j\in B_{\graph}(y,3\rho)$, and the assumption in the claim, 
\begin{align}
 K_{0}^{2}\lambda<\frac{1}{\#B_{\graph}(y,\rho)}
 \sum_{w\in B_{\graph}(y,\rho)}\lvert Tf(w)\rvert^2
 &\leq C\frac{1}{\#B_{\graph}(y,3\rho)}
 \sum_{w\in B_{\graph}(y,3\rho)}\lvert Tf(w)\rvert^2\\
 &\leq CM_{\graph}((Tf)^{2})(z_{j})\leq C\lambda.\label{e.lgballav}
\end{align}
Choosing $K_0^2$ larger than the constant $C$ in \eqref{e.lgballav} gives a contradiction. Therefore, $B_{\graph}(y,\rho)\setminus\wt B_{j}=\emptyset$.
 
 	\item\label{it.ACDHc2} \textit{There exist $K_{1}\in(1,\infty)$ and $C\in(0,\infty)$ such that, for every $j\in I$, $\lambda\in(0,\infty)$, $K\in[K_{1},\infty)$ and $\gamma\in(0,1)$, if $M_{\graph}((Tf)^2)(z_{j})\leq\lambda$, then $\#\sG(\Omega,j,\lambda,K,\gamma) \leq C\gamma^2\#\wt B_{j}$.}
We may assume that $\sG(\Omega,j,\lambda,K,\gamma)$ is nonempty. In this case, using the fact that $3B_j\subset\wt B_j$, \eqref{e.spfunc}, and the definition in \eqref{e.defsG},
\begin{equation}\label{e.rghsml}
 \frac{1}{\#\wt B_j}
 \sum_{w\in\wt B_j}
 \lvert T((I-A_{2\delta_j})f)(w)\rvert^2
 \leq\gamma^2\lambda.
\end{equation}
 By the sublinearity of $T$, \begin{equation}
	Tf(w)\leq T(A_{2\delta_j}f)(w)
+T((I-A_{2\delta_{j}})f)(w),\ \text{ for every $w\in\vertex$}.\label{e.sublin}
\end{equation} 
Condition \ref{it.hb}, namely \eqref{e.ACDH-2}, applied with $(x_j,2\delta_j)$, together with the assumption in the claim and the fact that $z_j\in\wt B_j$, gives
\begin{equation}\label{e.smthctl}
 \sup_{w\in\wt B_j}T(A_{2\delta_j}f)(w)
 \overset{\eqref{e.ACDH-2}}{\leq} C_0\inf_{z\in\wt B_j}M_{\graph}((Tf)^2)(z)^{\frac{1}{2}}\leq C_0M_{\graph}((Tf)^2)(z_{j})^{\frac{1}{2}}
 \leq C_0\lambda^{\frac{1}{2}}.
\end{equation}
Let $K_{1}:=K_{0}\vee (2C_{0})$. Suppose that $z\in3B_j$ satisfies $M_{\graph}((Tf)^2)(z)>K^2\lambda$. Then, by definition, there exists $B_{\graph}(y,\rho)\ni z$ such that \[\frac{1}{\# B_{\graph}(y,\rho)}
 \sum_{w\in B_{\graph}(y,\rho)}\lvert Tf(w)\rvert^2> K^{2}\lambda\geq K_{0}^{2}\lambda.\] By \ref{it.ACDHc1}, $B_{\graph}(y,\rho)\subset \wt{B}_{j}$. Hence, by \eqref{e.sublin} and \eqref{e.smthctl}, for every $K\in[K_{1},\infty)$,
\begin{align}
 K^2\lambda
 &<\frac{1}{\#B_{\graph}(y,\rho)}
 \sum_{w\in B_{\graph}(y,\rho)}\lvert Tf(w)\rvert^2\\
 &\overset{\eqref{e.sublin},\eqref{e.smthctl}}{\leq}2C_0^2\lambda
 +\frac{2}{\#B_{\graph}(y,\rho)}
 \sum_{w\in B_{\graph}(y,\rho)}
 \lvert T((I-A_{2\delta_j})f)(w)\rvert^2\\
 &\leq \frac{1}{2} K^2\lambda+\frac{2}{\#B_{\graph}(y,\rho)}
 \sum_{w\in B_{\graph}(y,\rho)}
 \lvert T((I-A_{2\delta_j})f)(w)\rvert^2.\label{e.sublin+}
 \end{align}
Consequently, since $B_{\graph}(y,\rho)\subset \wt{B}_{j}$,
\begin{align}
	 &\phantom{\ \leq}	M_{\graph}\left(\lvert T((I-A_{2\delta_j})f)\rvert^2\one_{\wt{B}_{j}}\right)(z)\geq \frac{1}{\#B_{\graph}(y,\rho)}
 \sum_{w\in B_{\graph}(y,\rho)}
 \lvert T((I-A_{2\delta_j})f)(w)\rvert^2\\
 & \overset{\eqref{e.sublin+}}{\geq}\frac{1}{4}K^2\lambda, \ \text{ for every }z\in \sG(\Omega,j,\lambda,K,\gamma).\label{e.level-set-lower-bound}
\end{align}
The weak $(1,1)$ estimate for $M_{\graph}$ \cite[Theorem~2.2]{Hei01} and \eqref{e.rghsml} now give
\begin{align}
	&\phantom{\ \leq}\#\sG(\Omega,j,\lambda,K,\gamma)\\&\overset{\eqref{e.level-set-lower-bound}}{\leq}\#\Sett{z\in\vertex}{M_{\graph}\left(\lvert T((I-A_{2\delta_j})f)\rvert^2\one_{\wt{B}_{j}}\right)(z)\geq \frac{1}{4}K^2\lambda}\\
	&\leq \frac{4C}{K^2\lambda}\sum_{w\in\wt B_j}
 \lvert T((I-A_{2\delta_j})f)(w)\rvert^2\overset{\eqref{e.rghsml}}{\leq}  \frac{C}{C_{0}^2\lambda}\gamma^{2} \lambda\#\wt B_j=C_{1}\gamma^{2}\#\wt B_j. 
\end{align}
This proves \ref{it.ACDHc2}.
\end{enumerate}

For $\lambda\in(0,\infty)$, set
$\Omega_\lambda:=\Sett{z\in\vertex}{M_{\graph}((Tf)^2)(z)>\lambda}$. Since
$T:L^2(\ambient,\meas)\to\ell^2(\vertex)$ is bounded, the weak $(1,1)$ estimate for $M_{\graph}$ gives $\#\Omega_\lambda
 \leq\frac{C}{\lambda}\lVert Tf\rVert_{\ell^2(\vertex)}^2<\infty$. If $\Omega_\lambda$ is nonempty, then it is a proper subset of the infinite graph. We apply Proposition \ref{p.Whitney}, obtaining a Whitney covering of $\Omega_\lambda$, that is, pairwise disjoint balls
$\{B_i^{(\lambda)}:=B_{\graph}(x^{(\lambda)}_i,r^{(\lambda)}_i)\}_{i\in I_{\lambda}}$. Put
$\delta_i^{(\lambda)}:=\metric_{\graph}(x_i^{(\lambda)},\vertex\setminus\Omega_\lambda)$ and
$\wt B^{(\lambda)}_i:=B_{\graph}(x^{(\lambda)}_i,2\delta^{(\lambda)}_i)$. Recall that we defined $z_{j}^{(\lambda)}\in\vertex\setminus\Omega_{\lambda}$ such that $\metric_{\graph}(x^{(\lambda)}_{j},z^{(\lambda)}_{j})=\delta^{(\lambda)}_{j}$. In particular, $M_{\graph}((Tf)^2)(z^{(\lambda)}_{j})\leq \lambda$.

Since $\Sett{z\in\vertex}{M_{\graph}((Tf)^2)(z)>K^2\lambda}=\Omega_{K^{2}\lambda}\subset \Omega_{\lambda}\subset \bigcup_{j\in I_{\lambda}}3B_j^{(\lambda)}$, we have 
\begin{align}
	&\phantom{\ \leq}\Sett{z\in\vertex}{M_{\graph}((Tf)^2)(z)>K^2\lambda,\ 
 M_{T,A}^{\#}f(z)\leq\gamma\lambda^{\frac{1}{2}}}\\
 &\subset \bigcup_{j\in I_{\lambda}}\Sett{z\in 3B_j^{(\lambda)}}{M_{\graph}((Tf)^2)(z)>K^2\lambda,\ 
 M_{T,A}^{\#}f(z)\leq\gamma\lambda^{\frac{1}{2}}}\\
 &=\bigcup_{j\in I_{\lambda}}\sG(\Omega_{\lambda},j,\lambda,K,\gamma).\label{e.witn1}
\end{align}

Therefore, by \ref{it.ACDHc2}, for $K\in[K_{1},\infty)$, \begin{align}
	&\phantom{\ \leq}\#\Sett{z\in\vertex}{M_{\graph}((Tf)^2)(z)>K^2\lambda,\ 
 M_{T,A}^{\#}f(z)\leq\gamma\lambda^{\frac{1}{2}}}\\
 &\overset{\eqref{e.witn1}}{\leq} \sum_{j\in I_{\lambda}}\#\sG(\Omega_{\lambda},j,\lambda,K,\gamma)\leq \sum_{j\in I_{\lambda}}C\gamma^2\#\wt B^{(\lambda)}_{j}\overset{\eqref{e.count}}{\leq} C_{1}\gamma^2\sum_{j\in I_{\lambda}}\#B^{(\lambda)}_{j} \text{ (doubling property)}\\
 &\leq  C_{1}\gamma^2\#\Omega_{\lambda}\ \text{ (since $\{B_i^{(\lambda)}\}_{i\in I_{\lambda}}$ are disjoint and contained in $\Omega_{\lambda}$)}.\label{e.witn2}
\end{align}
Consequently,
\begin{equation}\label{e.dglamb}
 \begin{aligned}
 &\phantom{\ \leq}\#\set{M_{\graph}((Tf)^2)>K^2\lambda}\\
 &\leq \#\set{M_{\graph}((Tf)^2)>K^2\lambda,\ 
 M_{T,A}^{\#}f(z)\leq\gamma\lambda^{\frac{1}{2}}}+\#\set{M_{T,A}^{\#}f>\gamma\lambda^{\frac{1}{2}}}\\
  &\overset{\eqref{e.witn2}}{\leq} C_{1}\gamma^2
 \#\set{M_{\graph}((Tf)^2)>\lambda}+\#\set{M_{T,A}^{\#}f>\gamma\lambda^{\frac{1}{2}}}.
 \end{aligned}
\end{equation}

Set $q:=p/2>1$, $U:=M_{\graph}((Tf)^2)$, and $S:=M_{T,A}^{\#}f$.
Since $Tf\in\ell^2(\vertex)$ and $F_f\in\ell^1(\vertex)$, the inclusion
$\ell^1(\vertex)\subset\ell^q(\vertex)$, the boundedness of the maximal operator, and
\eqref{e.spctol} imply that $M_{\graph}((Tf)^2)\in\ell^q(\vertex)$ and $S\in\ell^{2q}(\vertex)$. Fix $K\geq K_1$. Multiplying \eqref{e.dglamb} by $qK^{2q}\lambda^{q-1}$ and integrating over $(0,\infty)$ gives
\[
 \lVert M_{\graph}((Tf)^2)\rVert_{\ell^q}^{q}
 \leq C_1\gamma^2K^{2q}\lVert M_{\graph}((Tf)^2)\rVert_{\ell^q}^{q}
      +K^{2q}\gamma^{-2q}\lVert M_{T,A}^{\#}f\rVert_{\ell^{2q}}^{2q}.
\]
Choose $\gamma\in(0,1)$ so that $C_1\gamma^2K^{2q}\leq1/2$.
Absorbing the first term and taking the $p$-th root yields
\begin{equation}\label{e.dglamb2}
 \left\lVert M_{\graph}((Tf)^2)^{\frac{1}{2}}\right\rVert_{\ell^p(\vertex)}
 \leq C_p\lVert M_{T,A}^{\#}f\rVert_{\ell^p(\vertex)}.
\end{equation}
Since $\lvert Tf\rvert^2\leq M_{\graph}((Tf)^2)$, equations \eqref{e.spctol} and \eqref{e.dglamb2}, together with the boundedness of $M_{\graph}$ on $\ell^{\frac{p}{2}}(\vertex)$, yield
 \begin{align}
 &\phantom{\ \leq}\lVert Tf\rVert_{\ell^p(\vertex)}\leq \lVert M_{\graph}((Tf)^2)^{\frac{1}{2}}\rVert_{\ell^p(\vertex)}
 \overset{\eqref{e.dglamb2}}{\leq} C_p\lVert M_{T,A}^{\#}f\rVert_{\ell^p(\vertex)}\\
& \overset{\eqref{e.spctol}}{\leq} C_p\left\lVert(M_{\graph}F_f)^{\frac{1}{2}}\right\rVert_{\ell^p(\vertex)}
 \leq C_p\left(\sum_{v\in\vertex}F_f(v)^{\frac{p}{2}}\right)^{\frac{1}{p}}
 =C_p\left(\sum_{v\in\vertex}
 \lVert f\rVert_{L^2(K_v,\meas)}^p\right)^{\frac{1}{p}}.\label{e.Tcell}
 \end{align}
Finally, \eqref{e.sanwich} and H\"older's inequality yield
\begin{equation}\label{e.cellhr}
 \lVert f\rVert_{L^2(K_v,\meas)}^p
 \leq\meas(K_v)^{\frac{p}{2}-1}\int_{K_v}\lvert f\rvert^p\dif\meas
 \overset{\eqref{e.sanwich},\UVG}{\leq} C
 \int_{K_v}\lvert f\rvert^p\dif\meas.
\end{equation}
Summing \eqref{e.cellhr} over all $v\in\vertex$ and combining with
\eqref{e.Tcell} proves \eqref{e.ACDHcl}.
\end{proof}

As an application, we obtain the following result. 

\begin{proposition}
	Assume $\UVG$, $\HKE$, $\Grad$ and \ref{e.EventGau}. For every $p\in(2,\infty)$ and $s\in(0,\infty)$, there exists $C_{p,s}\in(0,\infty)$ such that, for every $f\in L^{2}(\ambient,\meas)$, \begin{equation}
		\sup_{\kappa\in(0,\infty)}\left(\sum_{v\in\vertex}\norm{\wgrad(-\gen+\kappa)^{-\frac{1}{2}}P_{s}f}_{L^{2}(\skeleton_{v},\medm)}^{p}\right)^{\frac{1}{p}}\leq C_{p,s}\norm{f}_{L^{p}(\ambient,\meas)}.\label{e.ACDHapp}
	\end{equation}
\end{proposition}

\begin{proof}
Fix $p\in(2,\infty)$, $s\in(0,\infty)$, and $\kappa\in(0,\infty)$, and define $Tf:\vertex\to\bR$ by
\begin{equation}
		(Tf)(v):=\norm{\wgrad(-\gen+\kappa)^{-\frac{1}{2}}P_{s}f}_{L^{2}(\skeleton_{v},\medm)}, \ \text{ for every }v\in\vertex.
	\end{equation}
It is clear that $T$ is non-negative and is sublinear. By the spectral calculus, we have
\begin{equation}\label{e.bddT}
 \lVert Tf\rVert_{\ell^2(\vertex)}^2
 =\int_{[0,\infty)}\frac{\lambda}{\lambda+\kappa}e^{-2s\lambda}
        \dif\langle\proj_\lambda f,f\rangle
 \leq\lVert f\rVert_{L^2(\ambient,\meas)}^2.
\end{equation}
Set $A_r:=P_{\scale(r)}$. We will verify that $T$ and $\{A_{r}\}_{r\in(0,\infty)}$ satisfy conditions \ref{it.ha} and \ref{it.hb} in Proposition \ref{p.ACDH}, with constants independent of $\kappa$.

Fix $v\in\vertex$ and $R\in[1,\infty)$, and let $F_f(w):=\lVert f\rVert_{L^2(K_w,\meas)}^2$. Choose $L_{s}:=1+s/\scale(1)$, so that $s<L_{s}\scale(R)$.
By \eqref{e.scale}, there exists an integer $J_{s}\in\bN$ such that
\[
 2(L_{s}+2)\scale(R)\leq\scale(2^jR)
 \quad\text{for all }R\in[1,\infty)\text{ and }j\in [J_{s},\infty)\cap\bN.
\]
Let
\[
 C_0:=B_{\graph}(v,2^{J_{s}+4}R),\ \text{ and }\
 C_j:=B_{\graph}(v,2^{j+J_{s}+4}R)\setminus B_{\graph}(v,2^{j+J_{s}+3}R),\ j\in\bN.
\]
Define $f_{j}:=f\one_{K_{C_j}}$ for $j\in\{0\}\cup\bN$. Then $f=\sum_{j=0}^{\infty}f_j$ in $L^2(\ambient,\meas)$. By \eqref{e.bddT}, the contractivity of the semigroup, and the doubling property,
\begin{align}
		&\phantom{\ \leq}\frac{1}{\#B_{\graph}(v,R)}\sum_{w\in B_{\graph}(v,R)}
 \lvert T(I-A_{R})(f_{0})(w)\rvert^2\\
 &\leq  \frac{1}{\#B_{\graph}(v,R)}\norm{T(I-P_{\scale(R)})(f\one_{K_{C_{0}}})}_{\ell^{2}(\vertex)}^{2}\\
 &\overset{\eqref{e.bddT}}{\lesssim} \frac{1}{\#B_{\graph}(v,R)}\norm{f\one_{K_{C_{0}}}}_{L^{2}(\ambient,\meas)}^{2}=\frac{1}{\#B_{\graph}(v,R)}\sum_{w\in B_{\graph}(v,2^{J_{s}+4}R)}F_{f}(w)\lesssim \inf_{z\in B_{\graph}(v,R)}M_{\graph}F_{f}(z).
	\end{align}
	For $j\in\bN$, we have $2^{j+J_{s}}R\gtrsim\metric_{\graph}(B_{\graph}(v,R),K_{C_{j}})\gtrsim 2^{j+J_{s}}R-C \gtrsim 2^{j+J_{s}}R$. Considering the function $Q_{s,\scale(R),\kappa}$ defined in Lemma~\ref{l.KerQ}, and applying Lemma~\ref{l.funint} with $\phi:= Q_{s,\scale(R),\kappa}$, we have \begin{equation}
		\wgrad(-\gen+\kappa)^{-\frac{1}{2}}P_{s}(I-P_{\scale(R)})=\riesz n_{\phi}(-\gen).
	\end{equation}
	By Lemma~\ref{l.funint}-\ref{it.scalar3} and the duality, we have for every measurable subset $F\subset\skeleton$, 
	\begin{equation}
		\norm{\wgrad(-\gen+\kappa)^{-\frac{1}{2}}P_{s}(I-P_{\scale(R)})f}_{L^{2}(F,\medm)}\leq \int_{0}^{\infty}\abs{Q_{s,\scale(R),\kappa}(t)}\norm{\wgrad P_{t}f}_{L^{2}(F,\medm)}\dif t.\label{e.bddT1}
	\end{equation} 
	As a consequence,
	\begin{align}
		&\phantom{\ \leq}\frac{1}{\#B_{\graph}(v,R)}\sum_{w\in B_{\graph}(v,R)}
 \lvert T(I-A_{R})(f\one_{K_{C_{j}}})(w)\rvert^2\\
 &\leq  \frac{1}{\#B_{\graph}(v,R)}\norm{\wgrad(-\gen+\kappa)^{-\frac{1}{2}}P_{s}(I-P_{\scale(R)})(f\one_{K_{C_{j}}})}_{L^{2}(\skeleton_{B_{\graph}(v,R)},\medm)}^{2}\\
 &\overset{\eqref{e.bddT1},\eqref{e.sepgrad1}}{\lesssim} \frac{\volume(R)}{\#B_{\graph}(v,R)}
\left(\int_{0}^{\infty}\frac{\abs{Q_{s,\scale(R),\kappa}(t)}\Lambda(t)}{\exp(c\Phi(\metric_{\graph}(B_{\graph}(v,R),K_{C_{j}}),t))}\dif t\right)^{2}\norm{f}_{L^{2}(K_{C_{j}},\meas)}^{2}\\
&\overset{\eqref{e.Qweight}}{\lesssim} \left(\frac{\scale(R)}{\scale(2^{j}R)}\right)^{2}\frac{\scale(2^{j}R)}{(2^{j}R)^{2}V(2^{j}R)}\norm{f}_{L^{2}(K_{C_{j}},\meas)}^{2}\\
&\lesssim 2^{-2\beta_{1}j}\left(\sup_{r\in[1,\infty)}r^{-2}{\scale(r)}\right)\inf_{z\in B_{\graph}(v,R)}M_{\graph}F_{f}(z)\overset{\eqref{e.EventGau}}{\lesssim}2^{-2\beta_{1}j}\inf_{z\in B_{\graph}(v,R)}M_{\graph}F_{f}(z).\label{e.bddT2}
	\end{align}
	The sublinearity of $T$ and Minkowski’s inequality give \begin{align}
		&\phantom{\ \leq} \frac{1}{\#B_{\graph}(v,R)}\sum_{w\in B_{\graph}(v,R)}
 \lvert T(I-A_{R})f(w)\rvert^2\\
 &\leq \left(\sum_{j=0}^{\infty}\left(\frac{1}{\#B_{\graph}(v,R)}\sum_{w\in B_{\graph}(v,R)}
 \lvert T(I-A_{R})(f\one_{K_{C_{j}}})(w)\rvert^2\right)^{\frac{1}{2}}\right)^{2}\\
 &\overset{\eqref{e.bddT2}}{\lesssim} \left(\sum_{j=0}^{\infty}2^{-\beta_{1}j}\right)^{2} \inf_{z\in B_{\graph}(v,R)}M_{\graph}F_{f}(z)\lesssim  \inf_{z\in B_{\graph}(v,R)}M_{\graph}F_{f}(z).
	\end{align}
	This proves \ref{it.ha}. To verify \ref{it.hb}, observe that \begin{equation}\label{e.bDDT3}
		(Tf)^{2}(q)=\norm{\wgrad (-\gen+\kappa)^{-\frac{1}{2}}P_{s}f}_{L^{2}(\skeleton_{q},\medm)}^{2}=E_{(-\gen+\kappa)^{-\frac{1}{2}}P_{s}f}(q).
	\end{equation}
Consequently, \begin{align}
		&\phantom{\ \leq}\sup_{w\in B_{\graph}(v,R)} (T(A_{R}f))(w)=\sup_{w\in B_{\graph}(v,R)}\norm{\wgrad(-\gen+\kappa)^{-\frac{1}{2}}P_{s}P_{\scale(R)}f}_{L^{2}(\skeleton_{w},\medm)}\\
		&=\sup_{w\in B_{\graph}(v,R)}\norm{\wgrad P_{\scale(R)}(-\gen+\kappa)^{-\frac{1}{2}}P_{s}f}_{L^{2}(\skeleton_{w},\medm)}\\
		&\overset{\eqref{e.harhb}}{\lesssim}R^{-1}\scale{(R)}^{\frac{1}{2}}\min_{q\in B_{\graph}(v,R)}(M_{\graph}E_{(-\gen+\kappa)^{-\frac{1}{2}}P_{s}f})(q)^{\frac{1}{2}}\\
		&\overset{\eqref{e.bDDT3}}{=}R^{-1}\scale{(R)}^{\frac{1}{2}}\min_{q\in B_{\graph}(v,R)}(M_{\graph}((Tf)^{2}))(q)^{\frac{1}{2}}\overset{\eqref{e.EventGau}}{\leq}C\min_{q\in B_{\graph}(v,R)}(M_{\graph}((Tf)^{2}))(q)^{\frac{1}{2}}.
	\end{align}
	This proves \ref{it.hb}. Proposition~\ref{p.ACDH} applies with constants
independent of $\kappa$. Taking the supremum over $\kappa\in(0,\infty)$ proves
\eqref{e.ACDHapp}.
\end{proof}

\begin{proposition}\label{p.bddlow}
	Assume $\UVG$, $\HKE$, $\Grad$ and \ref{e.EventGau}. For every $p\in(2,\infty)$, there exists $C_{p}\in(0,\infty)$ such that \begin{equation}\label{e.lowest}
		\norm{\riesz_{\low}f}_{L^{p}(\skeleton,\medm)}\leq C_{p} \norm{f}_{L^{p}(\ambient,\meas)},\ \text{for all }f\in L^{p}(\ambient,\meas)\cap L^{2}(\ambient,\meas).
	\end{equation}
\end{proposition}
\begin{proof}
Fix $f\in L^p(\ambient,\meas)\cap L^2(\ambient,\meas)$. We have \begin{align}
	&\phantom{\ \leq}\sup_{\kappa\in(0,\infty)}\norm{\wgrad(-\gen+\kappa)^{-\frac{1}{2}}P_{1}f}_{L^{p}(\skeleton,\medm)}=\sup_{\kappa\in(0,\infty)}\norm{\wgrad P_{\frac{1}{2}}(-\gen+\kappa)^{-\frac{1}{2}}P_{\frac{1}{2}}f}_{L^{p}(\skeleton,\medm)}\\
	&\overset{\eqref{e.utsmth}}{\lesssim}\sup_{\kappa\in(0,\infty)}\left(\sum_{v\in\vertex}\norm{\wgrad (-\gen+\kappa)^{-\frac{1}{2}}P_{\frac{1}{2}}f}_{L^{2}(\skeleton_v,\medm)}^{p}\right)^{\frac{1}{p}}\overset{\eqref{e.ACDHapp}}{\lesssim}\norm{f}_{L^{p}(\ambient,\meas)}.\label{e.lowest1}
\end{align}

By Lemma~\ref{l.defriesz}-\ref{it.dres3}, there exists a sequence $\kappa_n\downarrow0$ such that $\wgrad(-\gen+\kappa_{n})^{-\frac{1}{2}}P_{1}f$ converges to $\riesz_{\low}f$ $\medm$-a.e.. By Fatou's lemma, \begin{equation}
	\norm{\riesz_{\low}f}_{L^{p}(\skeleton,\medm)}\overset{\text{(Fatou)}}{\leq}\liminf_{n\to\infty}\norm{\wgrad(-\gen+\kappa_{n})^{-\frac{1}{2}}P_{1}f}_{L^{p}(\skeleton,\medm)}\overset{\eqref{e.lowest1}}{\lesssim}\norm{f}_{L^{p}(\ambient,\meas)}.
\end{equation}
This completes the proof of \eqref{e.lowest}.
\end{proof}

\subsection{Estimates for the high-frequency part \texorpdfstring{$\riesz_{\high}$}{R\_{♯}}}

\begin{proposition}\label{p.bddhi}
Assume $\UVG$, $\HKE$, and $\Grad$. Let $p\in(2,\infty)$. Suppose that $\DN{p}$ holds. Then there exists $C_p\in(0,\infty)$ such that
\begin{equation}\label{e.hiest}
 \norm{\riesz_{\high}f}_{L^p(\skeleton,\medm)}
 \leq C_p\norm{f}_{L^p(\ambient,\meas)},
 \ \text{ for all } f\in L^p(\ambient,\meas)\cap L^2(\ambient,\meas).
\end{equation}
\end{proposition}

\begin{proof}
Set $\phi:=Q_{0,1,0}$. By Lemma~\ref{l.KerQ},
\begin{equation}
		\int_{0}^{\infty}\phi(t)e^{-t\lambda}\dif t=\lambda^{-\frac{1}{2}}(1-e^{-\lambda}),\ \lambda\in(0,\infty).
	\end{equation}
By Proposition~\ref{p.equiH1} and the condition $\DN{p}$, we know that $\int_0^1t^{-\frac12-\frac1p} \scale^{-1}(t)^{-1+\frac2p}\dif t<\infty$. Since the function $t\mapsto t^{-1/p}\scale^{-1}(t)^{-1+2/p}$ is decreasing, we therefore have
\begin{equation}\label{e.Jpftn}
	 J_p:=\int_0^\infty|\phi(t)|t^{-1/p} \scale^{-1}(t)^{-1+2/p}\dif t<\infty.
\end{equation}
Fix $f\in L^2(\ambient,\meas)\cap L^p(\ambient,\meas)$. For $\varepsilon\in(0,1)$, define
\[
 \phi_\varepsilon(t):=\one_{(\varepsilon,\infty)}(t)\phi(t),
 \ \text{ and }\
 n_\varepsilon(\lambda):=\lambda^{\frac{1}{2}}
 \int_\varepsilon^\infty\phi(t)e^{-t\lambda}\dif t,\ \lambda\in[0,\infty).
\]
Now $\int_0^\infty|\phi_\varepsilon(t)|t^{-1/2}\dif t<\infty$, so by Lemma~\ref{l.funint}-\ref{it.scalar3} and the H\"{o}lder inequality, we have, for all $g\in L^2(\skeleton,\medm)\cap L^{p'}(\skeleton,\medm)$
\begin{align}
 \bigl|\langle\riesz n_\varepsilon(-\gen)f,g\rangle_{L^2(\skeleton,\medm)}\bigr|
 &\overset{\eqref{e.sccgradf}}{\leq}\int_\varepsilon^\infty|\phi(t)|
 \norm{\wgrad P_tf}_{L^p(\skeleton,\medm)}\dif t\,
 \norm{g}_{L^{p'}(\skeleton,\medm)}\\
 &\overset{\eqref{e.migrad},\eqref{e.Jpftn}}{\leq} C_pJ_p\norm{f}_{L^p(\ambient,\meas)}
 \norm{g}_{L^{p'}(\skeleton,\medm)},
\end{align}
where the last inequality follows from Proposition~\ref{p.migrad}. By duality,
\begin{equation}\label{e.hiestvep}
	 \norm{\riesz n_\varepsilon(-\gen)f}_{L^p(\skeleton,\medm)}
 \leq C_pJ_p\norm{f}_{L^p(\ambient,\meas)}.
\end{equation}
Since
\[
 n_\varepsilon(\lambda)
 =1-e^{-\lambda}-\frac{\lambda^{\frac{1}{2}}}{\sqrt\pi}
 \int_0^\varepsilon t^{-1/2}e^{-t\lambda}\dif t,\ \text{for every $\lambda\in[0,\infty)$},
\]
we have $|n_\varepsilon(\lambda)|\leq2$ and $\lim_{\varepsilon\downarrow0} n_\varepsilon(\lambda)=1-e^{-\lambda}$ for all $\lambda\in[0,\infty)$. By the spectral theorem and Lemma~\ref{l.defriesz}, we have
\[
 \lim_{\varepsilon\downarrow0}\riesz n_\varepsilon(-\gen)f=
 \riesz(I-P_1)f=\riesz_{\high}f
 \quad\text{in }L^2(\skeleton,\medm).
\]
Let $\varepsilon_{j}\downarrow0$ as $j\uparrow\infty$ such that $\lim_{j\uparrow \infty}\riesz n_{\varepsilon_{j}}(-\gen)f=\riesz_{\high}f$ $\medm$-a.e. on $\skeleton$. Then \begin{equation}
	\norm{\riesz_{\high}f}_{L^p(\skeleton,\medm)}\leq\liminf_{j\uparrow\infty}\norm{\riesz n_{\varepsilon_{j}}(-\gen)f}_{L^p(\skeleton,\medm)}\overset{ \eqref{e.hiestvep}}{ \leq} C_pJ_p\norm{f}_{L^p(\ambient,\meas)},
\end{equation} 
which gives \eqref{e.hiest}.
\end{proof}

\subsection{The one-dimensional case}

 \begin{proposition}\label{p.onedim}
 	Assume $\UVG$, $\HKE$, $\Grad$, \ref{e.EventGau} and \ref{e.linear}. Then for each $p\in[2,\infty)$, $\R{p}$ holds, namely, there exists $C_{p}\in(0,\infty)$ such that \begin{equation}
	\norm{\riesz f}_{L^{p}(\skeleton,\medm)}\leq C_{p}\norm{f}_{L^{p}(\ambient,\meas)},\ \text{ for all }f\in L^{2}(\ambient,\meas)\cap L^{p}(\ambient,\meas).
\end{equation}
 \end{proposition}
 \begin{proof}
By \ref{e.EventGau}, \ref{e.linear}, and the continuity of $\scale$, we may assume that $\scale(r)=r^{2}$ for all $r\in(0,\infty)$. By \cite[Theorem 2.13-(b)]{KM20}, the energy measure $\Gamma\langle f,f\rangle\ll\meas$ for every $f\in\domain$, and the intrinsic metric $\metric_{\mathrm{int}}$ of $(\form,\domain)$ is bi-Lipschitz equivalent to the underlying metric $\metric$.  Therefore, by the definition in \cite[p.~4]{CJKS20}, $(\ambient,\metric_{\mathrm{int}},\meas,\form)$ is a \emph{Dirichlet metric measure space endowed with a ``carr\'e du champ''}. We verify that the ambient measure $\meas$ is comparable to $\medm$, and that $(\ambient,\metric_{\mathrm{int}},\meas,\form)$ satisfies the conditions ${(UE)}$, ${P}_{p}$ and ${G}_{p}$ in \cite{CJKS20}.\begin{enumerate}[label=\textup{(\roman*)},align=right,leftmargin=*,topsep=5pt,parsep=0pt,itemsep=2pt]
 	\item We first prove the equivalence between $\meas$ and $\medm$. By Lemma~\ref{l.growthPsi} and \ref{e.linear}, the metric measure space $(\ambient,\metric,\meas)$ has \emph{locally linear} volume growth: there exists $C\in(1,\infty)$ such that
   \begin{equation}\label{e.linear1}
C^{-1}r\overset{\eqref{e.linear}}{\leq} \volume(r)\leq Cr,\ \text{ for all }r\in(0,r_{*}).
 \end{equation} Fix $x\in\ambient$ and $r\in(0,\min(r_{*},\iota))$. Since $(\ambient,\metric)$ is unbounded, there exists $y\in\ambient$ with $\metric(x,y)>r$. By the geodesic property, there exists a unit-speed geodesic $\gamma:[0,\metric(x,y)]\to\ambient$ with $\gamma(0)=x$ and $\gamma(d(x,y))=y$. It follows that $\medm(B(x,r))\geq\sH^{1}(\gamma(0,r))=r$. Let $\delta\in(0,r/10)$. Since $\ol{B(x,r)}$ is compact, there exists a finite maximal $\delta$-separated set $\{x_{j}\}_{j=1}^{N_{\delta}}$. Since $\{B(x_{j},\delta/3)\}_{j=1}^{N_{\delta}}$ are pairwise disjoint, and $\bigcup_{j=1}^{N_{\delta}}B(x_{j},\delta/3)\subset B(x,2r)$, we have \begin{equation}
 		\frac{N_{\delta}\delta}{3C}\leq \sum_{j=1}^{N_{\delta}}\meas(B(x_{j},\delta/3))\leq \meas(B(x,2r))\overset{\eqref{e.linear1}}{\leq} Cr.
 	\end{equation}
 	Hence, $N_{\delta}\leq 3C^{2}\delta^{-1}r$. Since $\{x_{j}\}_{j=1}^{N_{\delta}}$ is maximal, and $\{B(x_{j},\delta)\}_{j=1}^{N_\delta}$ covers $B(x,r)$, by the definition of the one-dimensional Hausdorff measure, we have \begin{equation}
 		\medm(B(x,r))=\sH^{1}(\skeleton\cap B(x,r))\leq\limsup_{\delta\downarrow0}\sum_{j=1}^{N_{\delta}}\diam(B(x_{j},\delta))\leq 6C^{2}r.
 	\end{equation}
 	Therefore, \begin{equation}
 		r\leq\medm(B(x,r))\leq 6C^{2}r,\ \text{ for every }(x,r)\in\ambient\times(0,r_{*}\wedge\iota).
 	\end{equation}
 	
 	Let $A\subset \ambient$ be a Borel set. Let $K\subset A$ be a compact set, and let $O\supset A$ be an open set. For each $x\in K$, choose $r_{x}\in (0,(10)^{-1}(\iota\wedge r_{*}\wedge\dist(x,\ambient\setminus O)))$. Here, if $O=\ambient$, then $\dist(x,\ambient\setminus O):=\infty$. The family $\{B(x,r_{x})\}_{x\in K}$ is a covering of $K$. By the $5B$-covering lemma, there exists an at most countable subset $K_{1}\subset K$ such that $K\subset\bigcup_{x\in K_{1}}B(x,5r_{x})\subset O$ and $\{B(x,r_{x})\}_{x\in K_{1}}$ are pairwise disjoint. Therefore, \begin{equation}
 		\medm(K)\leq \sum_{x\in K_1}\medm(B(x,5r_{x}))\leq 30C^{2}\sum_{x\in K_1} r_{x}\overset{\eqref{e.linear1}}{\leq} 30C^{3}\sum_{x\in K_1} \meas(B(x,r_{x}))\leq 30C^{3}\meas(O).
 	\end{equation}
 	Moreover, \begin{equation}
 		\meas(K)\leq \sum_{x\in K_1}\meas(B(x,5r_{x}))\overset{\eqref{e.linear1}}{\leq} 5C\sum_{x\in K_1} r_{x}\leq 5C\sum_{x\in K_1}\medm(B(x,r_{x}))\leq 5C \medm(O).
 	\end{equation}
 	In particular, this proves that $\medm$ is a Radon measure on $(\ambient,\metric)$. Using the inner and outer regularity of $\medm$ and $\meas$, we may take the supremum over compact $K\subset A$ and then the infimum over open $O\supset A$, and obtain
\begin{equation}\label{e.medm=meas}
 		\frac{1}{5C}\meas(A)\leq\medm(A)\leq 30C^{3}\meas(A).
 	\end{equation}
	\item  Since the heat kernel estimates $\HKE$ are invariant under bi-Lipschitz changes of metric, we know that $(\ambient,\metric_{\mathrm{int}},\meas,\form)$ satisfies ${(UE)}$ defined in \cite[p.~5]{CJKS20}.
	\item Under $\UVG$ and $\HKE$, the Poincar\'{e} inequality \eqref{e.PI} holds. By \cite[Theorem 2.4]{Stu96}, the constant $A_{\mathrm{PI}}$ appearing in \eqref{e.PI} can be chosen to be $A_{\mathrm{PI}}=1$. Therefore, by H\"{o}lder's inequality, the bi-Lipschitz equivalence between $\metric$ and $\metric_{\mathrm{int}}$, and $p\geq2$, we have
\begin{align}
		&\phantom{\ \leq}\fint_{B(x,r)}\abs{f-f_{B(x,r)}}\dif\meas\\
		&\leq \left(\fint_{B(x,r)}\abs{f-f_{B(x,r)}}^{2}\dif\meas\right)^{\frac{1}{2}}\overset{\eqref{e.PI}}{\leq} C\left(\frac{r^{2}}{m(B(x,r))}\int_{B(x,r)}\Gamma\langle f,f\rangle\dif \meas\right)^{\frac{1}{2}}\\
		&\overset{\eqref{e.medm=meas}}{\leq}C\left(r^{2}\fint_{B(x,r)}\frac{\dif\Gamma\langle f,f\rangle}{\dif \meas}\dif \meas\right)^{\frac{1}{2}}\leq C r\left(\fint_{B(x,r)}\abs{\frac{\dif\Gamma\langle f,f\rangle}{\dif \meas}}^{\frac{p}{2}}\dif \meas\right)^{\frac{1}{p}}
	\end{align}
	This proves the $L^{p}$-Poincar\'{e} inequality in \cite[p.~6]{CJKS20}.
	\item  The heat-semigroup gradient estimate $(G_{p})$ from \cite[p.~11]{CJKS20} is established in Proposition \ref{p.migrad}, upon noting that \eqref{e.medm=meas} and that $\left(\frac{\Gamma\langle f,f\rangle}{\dif \medm}\right)^{\frac{1}{2}}=\abs{\wgrad f}$ for every $f\in\domain$.
\end{enumerate}
It follows from \cite[Theorem 1.9]{CJKS20} that the Riesz transform $\riesz$ is bounded from $L^{p}(\ambient,\meas)$ to $L^{p}(\skeleton,\medm)$ for all $p\in[2,\infty)$.
 \end{proof}

 \begin{proof}[Proof of Theorem~\ref{t.main}]
 	Assume $\UVG$, $\HKE$, and $\Grad$.
 	\begin{enumerate}[label=\textup{({\roman*})},align=right,leftmargin=*,topsep=5pt,parsep=0pt,itemsep=2pt]
 		\item Assume \ref{e.EventGau} and $\DN{q}$ for some $q\in(2,\infty)$. Then $\DN{p}$ holds for all $p\in[q,\infty)$. Propositions~\ref{p.bddlow} and~\ref{p.bddhi} therefore show that $\R{p}$ holds for every $p\in[q,\infty)$. For $p\in(2,q)$, we can use interpolation of $\R{q}$ and $\R{2}$, since $(\ambient,\meas)$ and $(\skeleton,\medm)$ are $\sigma$-finite. For any $q\in(1,2)$, $\RR{q}$ follows from Lemma \ref{l.reverse-duality}.
 		\item If \ref{e.EventGau} and \ref{e.linear} hold and $p\in[2,\infty)$, then $\R{p}$ follows directly from Proposition \ref{p.onedim}. Again, for any $q\in(1,2]$, $\RR{q}$ follows from Lemma \ref{l.reverse-duality}.
 		\qedhere
 	\end{enumerate}
 \end{proof}

\subsection{Reverse Riesz inequality}\label{s.sharpness}
\begin{lemma}\label{l.reverse-semigroup}
Let $p\in(1,\infty)$ and $t\in(0,\infty)$. For every $f\in L^2(\ambient,\meas)\cap L^p(\ambient,\meas)$, one has $(I-P_t)f\in\Dom((-\gen)^{-\frac12})$ and
\begin{equation}\label{e.sharpmultiplier}
\norm{(-\gen)^{-\frac12}(I-P_t)f}_{L^p(\ambient,\meas)}\leq\frac{4\sqrt{t}}{\sqrt{\pi}}\norm{f}_{L^p(\ambient,\meas)}.
\end{equation}
Consequently, every $u\in\domain$ with $(-\gen)^{\frac12}u\in L^p(\ambient,\meas)$ satisfies $u-P_tu\in L^p(\ambient,\meas)$ and
\begin{equation}\label{e.reverse-semigroup}
\norm{u-P_tu}_{L^p(\ambient,\meas)}\leq\frac{4\sqrt{t}}{\sqrt{\pi}}\norm{(-\gen)^{\frac12}u}_{L^p(\ambient,\meas)}.
\end{equation}
\end{lemma}
\begin{proof}
Since $0\leq 1-e^{-s}\leq (s\wedge1)$ for all $s\in[0,\infty)$, one has
\begin{equation}\label{e.sharpmult1}
\int_{(0,\infty)}\lambda^{-1}(1-e^{-t\lambda})^2\dif\langle\proj_\lambda f,f\rangle\leq t\norm{f}_{L^2(\ambient,\meas)}^2.
\end{equation}
It follows that $(I-P_t)f\in\Dom((-\gen)^{-\frac12})$. Set $g:=(-\gen)^{-\frac12}(I-P_t)f\in L^2(\ambient,\meas)$.

For $t\in(0,\infty)$, let $Q_{0,t,0}$ be defined in \eqref{e.Qkern},
\begin{equation}\label{e.sharpmult2}
Q_{0,t,0}(s)=\frac{1}{\sqrt{\pi}}\left(s^{-\frac12}\one_{(0,t]}(s)+(s^{-\frac12}-(s-t)^{-\frac12})\one_{(t,\infty)}(s)\right), \ s\in(0,\infty).
\end{equation}
Direct integration gives $\int_0^t s^{-\frac12}\dif s=2\sqrt{t}$ and  $\int_t^\infty\bigl((s-t)^{-\frac12}-s^{-\frac12}\bigr)\dif s=2\sqrt{t}$. Therefore,
\begin{equation}\label{e.sharpmult3}
\int_0^\infty\abs{Q_{0,t,0}(s)}\dif s=\frac{4\sqrt{t}}{\sqrt{\pi}},\ \text{ and }\ \int_0^\infty{Q_{0,t,0}(s)}\dif s=0.
\end{equation}
For $\lambda\in(0,\infty)$, substitution in the two convergent Laplace integrals yields
\begin{equation}\label{e.sharpmult4}
\begin{aligned}
\int_0^\infty Q_{0,t,0}(s)e^{-s\lambda}\dif s=\frac{1-e^{-t\lambda}}{\sqrt{\pi}}\int_0^\infty s^{-\frac12}e^{-s\lambda}\dif s=\frac{1-e^{-t\lambda}}{\sqrt{\lambda}}.
\end{aligned}
\end{equation}

For $h\in L^2(\ambient,\meas)$, by \cite[Lemma IX.1.9, p.~257]{Con90}, the total variation of the spectral measure $A\mapsto \int_{A}\dif\langle\proj_\lambda f,h\rangle$ satisfies
\begin{equation}\label{e.sharpmult5}
\int_{[0,\infty)}\dif\abs{\langle\proj_\lambda f,h\rangle}\leq\norm{f}_{L^2(\ambient,\meas)}\norm{h}_{L^2(\ambient,\meas)}.
\end{equation}
Consequently,
\begin{equation}\label{e.sharpmult7}
\int_0^\infty\int_{[0,\infty)}\abs{Q_{0,t,0}(s)}e^{-s\lambda}\dif\abs{\langle\proj_\lambda f,h\rangle}\dif s\overset{\eqref{e.sharpmult5}}{\leq}\frac{4\sqrt{t}}{\sqrt{\pi}}\norm{f}_{L^2(\ambient,\meas)}\norm{h}_{L^2(\ambient,\meas)}<\infty.
\end{equation}
For $h\in L^2(\ambient,\meas)\cap L^{p'}(\ambient,\meas)$, where $p'=p/(p-1)$, the spectral theorem, \eqref{e.sharpmult4}, and Fubini's theorem (which is justified by \eqref{e.sharpmult7}), give
\begin{align}
	\abs{\langle (-\gen)^{-\frac12}(I-P_t)f,h\rangle_{L^2(\ambient,\meas)}}&\overset{\eqref{e.sharpmult4}}{=}\abs{\int_0^\infty Q_{0,t,0}(s)\langle P_sf,h\rangle_{L^2(\ambient,\meas)}\dif s}\\
	&\leq \int_0^\infty \abs{Q_{0,t,0}(s)}\norm{P_sf}_{L^p(\ambient,\meas)}\norm{h}_{L^{p'}(\ambient,\meas)}\dif s\\
	&\overset{\eqref{e.sharpmult3}}{\leq}\frac{4\sqrt{t}}{\sqrt{\pi}}\norm{f}_{L^p(\ambient,\meas)}\norm{h}_{L^{p'}(\ambient,\meas)}.
\end{align}
By duality, we obtain \eqref{e.sharpmultiplier}.

Finally, apply \eqref{e.sharpmultiplier} to $f:=(-\gen)^{\frac12}u$. Since $\Ker(-\gen)=\{0\}$, the spectral calculus gives
\begin{equation}\label{e.sharpmult11}
(-\gen)^{-\frac12}(I-P_t)(-\gen)^{\frac12}u=(I-P_t)u\ \text{ in }L^2(\ambient,\meas).
\end{equation}
This proves \eqref{e.reverse-semigroup}.
\end{proof}
\begin{proposition}\label{p.reverse-rigidity}
Assume $\UVG$ and $\HKE$, and let $p\in(2,\infty)$. There exist $c_p\in(0,\infty)$ and $r_0\in(0,\iota)$ such that, for any $(x,r)\in\ambient\times(0,r_{0}]$, the function $\eta_{x,r}$ defined in Lemma~\ref{l.cutoff} satisfies
\begin{equation}\label{e.reverse-test-functions}
\norm{(-\gen)^{\frac12}\eta_{x,r}}_{L^p(\ambient,\meas)}\geq c_p\left(\frac{\volume(r)}{r}\right)^{\frac1p-\frac12}\norm{\wgrad\eta_{x,r}}_{L^p(\skeleton,\medm)}>0.
\end{equation}
\end{proposition}
\begin{proof}
For $x\in\ambient$ and $r\in(0,\iota/6)$, by Lemma~\ref{l.cutoff} the function $\eta_{x,r}\in\domain\cap C_c(\ambient)$ satisfies
\begin{equation}\label{e.sharpcut1}
0\leq\eta_{x,r}\leq1,\ \restr{\eta_{x,r}}{B(x,r)}=1,\ \text{ and }\ \restr{\eta_{x,r}}{\ambient\setminus B(x,4r)}=0.
\end{equation}
Since $\eta_{x,r}$ is not constant, \eqref{e.cutoff1} and \eqref{e.cutoff2} imply
\begin{equation}\label{e.reverse-cutoff}
0<\int_{\skeleton}\abs{\wgrad\eta_{x,r}}^p\dif\medm\overset{\eqref{e.cutoff1}}{\leq} r^{-p}\medm\bigl(\supp_{\medm}[\abs{\wgrad\eta_{x,r}}]\bigr)\overset{\eqref{e.cutoff2}}{\leq} Cr^{1-p}.
\end{equation}
For $A\in(4,\infty)$, the heat-kernel upper bound, $\UVG$, and \eqref{e.volreverse} give
\begin{equation}\label{e.sharpcut2}
0\leq P_{\scale(Ar)}\eta_{x,r}(z)\overset{\HKE, \eqref{e.sharpcut1}}\leq\frac{C\meas(B(x,4r))}{\volume(Ar)}\overset{\UVG}{\leq}\frac{C\volume(4r)}{\volume(Ar)}\overset{\eqref{e.volreverse}}{\leq}\frac{C}{A},\ \text{ $\meas$-a.e. $z\in\ambient$}.
\end{equation}
Fix $A>8C+4$ so that the last term is at most $1/2$. Choose $r_{0}\in(0,1\wedge\frac{\iota}{3A})$. For $r\in(0,r_0]$, \eqref{e.sharpcut1} and \eqref{e.sharpcut2} yield
\begin{equation}\label{e.sharpcut4}
(I-P_{\scale(Ar)})\eta_{x,r}\geq\frac12\ \meas\text{-a.e.  on }B(x,r).
\end{equation}
If $(-\gen)^{\frac12}\eta_{x,r}\notin L^p(\ambient,\meas)$, then \eqref{e.reverse-test-functions} holds with infinite left-hand side. Otherwise, \eqref{e.reverse-semigroup} implies that
\begin{align}\label{e.sharpcut6}
&\phantom{\ \leq}\norm{(-\gen)^{\frac12}\eta_{x,r}}_{L^p(\ambient,\meas)}\overset{\eqref{e.reverse-semigroup}}{\geq}\frac{\sqrt{\pi}}{4\sqrt{\scale(Ar)}}\norm{(I-P_{\scale(Ar)})\eta_{x,r}}_{L^p(\ambient,\meas)}\\
&\overset{\eqref{e.sharpcut4}}{\geq} \frac{\sqrt{\pi}}{8\sqrt{\scale(Ar)}}\meas(B(x,r))^{\frac1p}\overset{\eqref{e.growth2},\UVG}{\gtrsim} c_pr^{-\frac12}\volume(r)^{\frac1p-\frac12}\\
&\overset{\eqref{e.reverse-cutoff}}{\gtrsim}\left(\frac{\volume(r)}{r}\right)^{\frac1p-\frac12}\norm{\wgrad\eta_{x,r}}_{L^p(\skeleton,\medm)}>0.
\end{align}
This proves \eqref{e.reverse-test-functions}.
\end{proof}
\begin{proof}[Proof of Theorem~\ref{t.smallp}]
By Lemma \ref{l.reverse-duality}, if $\R{q}$ holds for some $q\in(1,2)$, then $\RR{p}$ holds for $p=q/(q-1)\in(2,\infty)$. So it suffices to prove the second case. Suppose that $\RR{p}$ holds for some $p\in(2,\infty)$. Then \eqref{e.reverse-test-functions} gives
\begin{equation}\label{e.sharpcut7}
\sup_{r\in(0,r_{0}]}\left(\frac{r}{\volume(r)}\right)^{\frac12-\frac1p}< \infty.
\end{equation}
Since $\frac12-\frac1p>0$, this implies $\volume(r)\geq cr$. On the other hand, by \eqref{e.volreverse}, we have $\volume(r)\leq C\volume(1)r$ for all $r\in(0,r_{0}]$. Thus $\volume(r)\asymp r$, and $\scale(r)\asymp r^2$ by \eqref{e.growth2}. The fact that $\meas\asymp\medm$ follows from the same proof as in Proposition \ref{p.onedim}.
\end{proof}
 \section{A sufficient condition for the gradient heat kernel estimates}\label{s.RHG} 
 
 \subsection{Reverse H\"{o}lder inequality}
 \begin{definition}
Let $(\ambient,\metric,\meas,\form,\domain)$ be a uniform local tree Dirichlet space in Framework \ref{f.treeMMD}.
\begin{enumerate}[label=\textup{({\arabic*})},align=right,leftmargin=*,topsep=5pt,parsep=0pt,itemsep=2pt]
  \item  Let $\Omega$ be a non-empty open subset of $\ambient$. We say that $h\in\domain_{\loc}$ is \emph{$\form$-harmonic} on $\Omega$, if \begin{equation}
	\form(h,\phi)=0,\ \text{ for all }\phi\in\domain\cap C_{c}(\Omega).
\end{equation}
\item We say that a \hypertarget{RH}{\emph{reverse H\"{o}lder inequality} for harmonic functions} $\RH$ holds, if there exist constants $A_{\mathrm{RH}},C_{\mathrm{RH}}\in(1,\infty)$ such that for any $(x,r)\in \ambient\times (0,\infty)$ and any $h\in \domain_{\loc}$ that is $\sE$-harmonic on $B(x,A_{\mathrm{RH}}r)$\begin{equation}\label{e.RH}
	\norm{\wgrad h}_{L^{\infty}(B(x,r),\medm)}\leq \frac{C _{\mathrm{RH}}}{r}\fint_{B(x,A_{\mathrm{RH}}r)}\abs{h}\dif\meas.
\end{equation}
\end{enumerate}
 \end{definition}
 
 The following gradient estimate for solutions of the Poisson equation is adapted from \cite[Theorem~3.2]{CJKS20}.
 \begin{proposition}\label{p.Poisson}
 Assume $\UVG$, $\UHK$ and $\RH$ hold for $(\ambient,\metric,\meas,\form,\domain)$. Then there exists $C\in(1,\infty)$ such that, for any $(x,r)\in\ambient\times(0,\infty)$, any $f\in L^{\infty}(B(x,A_{\mathrm{RH}}r),\meas)$, and any $u\in\domain$ satisfying \begin{equation}\label{e.Pos-1}
 		\form(u,\phi)=\int_{B(x,A_{\mathrm{RH}}r)}f\phi\dif\meas,\ \text{ for all }\phi\in \domain\cap C_{c}(B(x,A_{\mathrm{RH}}r)),
 	\end{equation}
 	we have 
 	\begin{equation}
 		\norm{\wgrad u}_{L^{\infty}(B(x,r),\medm)}\leq \frac{C }{r}\fint_{B(x,A_{\mathrm{RH}}r)}\abs{u}\dif\meas+C\frac{\scale{(r)}}{r}\norm{f}_{L^{\infty}(B(x,A_{\mathrm{RH}}r),\meas)}.\label{e.Pos-cm}
 	\end{equation}
 \end{proposition}
 
 \begin{proof}
Since $\UVG$ and $\UHK$ imply the positivity of the first Dirichlet eigenvalue on $B(x,A_{\mathrm{RH}}r)$, by Riesz's representation theorem, there exists a unique $v\in\domain(B(x,A_{\mathrm{RH}}r))$ such that \begin{equation}\label{e.Pos-2}
 		\form(v,\phi)=\int_{B(x,A_{\mathrm{RH}}r)}f\phi\dif\meas,\ \text{ for all }\phi\in \domain\cap C_{c}(B(x,A_{\mathrm{RH}}r)).
 	\end{equation}
 	Let $h:=u-v\in\domain_{\loc}$. Then $h$ is $\form$-harmonic on $B(x,A_{\mathrm{RH}}r)$ by the bilinearity of $\form$, \eqref{e.Pos-1} and \eqref{e.Pos-2}. By \cite[Theorem 2.2]{GH14} and \cite[Lemmas~3.1 and 3.2, Definition~3.10]{GT12}, we know that there exists a constant $C\in(1,\infty)$, depending only on $A_{\mathrm{RH}}$ and constants appearing in $\UVG$ and $\UHK$, such that \begin{equation}\label{e.Pos-3}
 		\norm{v}_{L^{\infty}(B(x,A_{\mathrm{RH}}r),\meas)}\leq C\scale(r)\norm{f}_{L^{\infty}(B(x,A_{\mathrm{RH}}r),\meas)}.
 	\end{equation} By $\RH$, the definition of $h$, and the triangle inequality, we have \begin{align}
 		\norm{\wgrad h}_{L^{\infty}(B(x,r),\medm)}&\overset{\RH}{\lesssim} \frac{1}{r}\fint_{B(x,A_{\mathrm{RH}}r)}\abs{h}\dif\meas\leq \frac{1}{r}\fint_{B(x,A_{\mathrm{RH}}r)}\abs{u}\dif\meas+\frac{1}{r}\norm{v}_{L^{\infty}(B(x,A_{\mathrm{RH}}r),\meas)}\\
 		&\overset{\eqref{e.Pos-3}}{\lesssim}\frac{1}{r}\fint_{B(x,A_{\mathrm{RH}}r)}\abs{u}\dif\meas+\frac{\scale(r)}{r}\norm{f}_{L^{\infty}(B(x,A_{\mathrm{RH}}r),\meas)}.\label{e.Pos-4}
 	\end{align}
 
 	We next estimate $\norm{\wgrad v}_{L^{\infty}(B(x,r),\medm)}$. Fix $y\in B(x,r)\cap\skeleton$, $A=4A_{\mathrm{RH}}(A_{\mathrm{RH}}-1)^{-1}$. Let $r_{0}=r$, $r_{j}:=2^{-j}A^{-1}r$ for $j\in\bN$, $B_{j}:=B(y,r_{j})$ for $j\in\{0\}\cup\bN$. Define $v_{0}:=v$. Again, by $\UVG$ and $\UHK$, we know that for each $j\in\bN$, the first Dirichlet eigenvalue on $A_{\mathrm{RH}}B_{j}$ is positive. Therefore, by the Riesz representation theorem, there exists a unique $v_{j}\in\domain(A_{\mathrm{RH}}B_{j})$ such that \begin{equation}
 		\form(v_{j},\phi)=\int_{A_{\mathrm{RH}}B_{j}}f\phi\dif\meas,\ \text{ for all }\phi\in \domain\cap C_{c}(A_{\mathrm{RH}}B_{j}).
 	\end{equation}
 	By \cite[Lemmas~3.1 and 3.2, Definition~3.10, Theorem~7.4]{GT12}, we know that there exists a constant $C\in(1,\infty)$, depending only on $A_{\mathrm{RH}}$ and constants appearing in $\UHK$ and $\UVG$, such that \begin{equation}\label{e.Pos<}
 		\norm{v_{j}}_{L^{\infty}(A_{\mathrm{RH}}B_{j},\meas)}\leq C\scale(r_{j})\norm{f}_{L^{\infty}(B(x,A_{\mathrm{RH}}r),\meas)},\ \text{ for all }j\in\{0\}\cup\bN.
 	\end{equation}
 	
 	For each $j\in\{0\}\cup\bN$, we define $w_{j}:=v_{j}-v_{j+1}\in\domain(A_{\mathrm{RH}}B_{j})$. Since $\restr{v_{j}}{\ambient\setminus (A_{\mathrm{RH}}B_{j})}=0$ $\meas$-a.e., we have by the definition of $w_{j}$ and the triangle inequality that\begin{align}
 		\norm{w_{j}}_{L^{\infty}(A_{\mathrm{RH}}B_{j+1},\meas)}&\leq\norm{v_{j}}_{L^{\infty}(A_{\mathrm{RH}}B_{j},\meas)}+\norm{v_{j+1}}_{L^{\infty}(A_{\mathrm{RH}}B_{j+1},\meas)}\\
& \overset{\eqref{e.Pos<},\eqref{e.scale}}{\lesssim}{\scale(r_{j})}\norm{f}_{L^{\infty}(B(x,A_{\mathrm{RH}}r),\meas)}.\label{e.Pos<1}
 	\end{align}
 	Let $z\in (B(y,\iota\wedge (2^{-10}A^{-1}r))\cap \skeleton)\setminus\{y\}$. Let $N=N_{y,z}\in\bN$ be an integer such that \begin{equation}\label{e.Pos<0}
 		r_{N_{y,z}}\in[\metric(y,z),2\metric(y,z)).
 	\end{equation}
 	Then, for each $j\in\{0,\ldots,N_{y,z}-1\}$, since $[y,z]\subset B_{j}$ and $w_{j}$ is $\form$-harmonic on $A_{\mathrm{RH}}B_{j+1}$ for each $j\in\{0\}\cup\bN$, we have \begin{align}
 	&\phantom{\ \leq}	\abs{w_{j}(z)-w_{j}(y)}\leq\int_{[z,y]}\abs{\wgrad w_{j}}\dif\medm\leq d(y,z)\norm{\wgrad w_{j}}_{L^{\infty}(B_{j+1},\medm)}\\ 
 		&\overset{\RH}{\lesssim}\frac{d(y,z)}{r_{j}}\norm{w_{j}}_{L^{\infty}(A_{\mathrm{RH}}B_{j+1},\meas)}\overset{\eqref{e.Pos<1}}{\lesssim} \metric(y,z)\frac{\scale(r_{j})}{r_{j}}\norm{f}_{L^{\infty}(B(x,A_{\mathrm{RH}}r),\meas)}.\label{e.Pos<2}
 	\end{align}
 	For $N=N_{y,z}\in\bN$, we have \begin{align}
 		&\phantom{\ \leq}\abs{v_{N}(y)-v_{N}(z)}\leq 2\norm{v_{N}}_{L^{\infty}(B_{N},\meas)}\overset{\eqref{e.Pos<}}{\lesssim} \scale(r_{N})\norm{f}_{L^{\infty}(B(x,A_{\mathrm{RH}}r),\meas)}\\
 		&=r_{N}\cdot \frac{\scale(r_{N})}{r_{N}}\norm{f}_{L^{\infty}(B(x,A_{\mathrm{RH}}r),\meas)}\overset{\eqref{e.Pos<0}}{\lesssim}\metric(y,z)\frac{\scale(r_{N})}{r_{N}}\norm{f}_{L^{\infty}(B(x,A_{\mathrm{RH}}r),\meas)}.\label{e.Pos<3}
 	\end{align}
 	Therefore, using the fact that $v=v_{N_{y,z}}+\sum_{j=0}^{N_{y,z}-1}w_{j}$, we have that, for $N=N_{y,z}$ \begin{align}
		\abs{v(y)-v(z)}&\leq  \abs{v_{N}(y)-v_{N}(z)}+\sum_{j=0}^{N-1}\abs{w_{j}(z)-w_{j}(y)}\\
		&\overset{\eqref{e.Pos<2},\eqref{e.Pos<3}}\lesssim \metric(y,z)\left(\sum_{j=0}^{N}\frac{\scale(r_{j})}{r_{j}}\right)\norm{f}_{L^{\infty}(B(x,A_{\mathrm{RH}}r),\meas)}\\
		&\leq \metric(y,z)\left(\sum_{j=0}^{\infty}\frac{\scale(r_{j})}{r_{j}}\right)\norm{f}_{L^{\infty}(B(x,A_{\mathrm{RH}}r),\meas)}.	\label{e.Pos3+}\end{align}
		For $j\in\{0\}\cup\bN$, we know that $r_{j}\leq r$ and therefore \begin{equation}\label{e.Pos<4}
			\frac{\scale(r_{j})}{r_{j}}\overset{\eqref{e.scale}}{\lesssim}\frac{\scale(r)}{r_{j}}\left(\frac{r_{j}}{r}\right)^{\beta_{1}}=\frac{\scale(r)}{r}\left(\frac{r_{j}}{r}\right)^{\beta_{1}-1}=A^{-(\beta_{1}-1)}2^{-j(\beta_{1}-1)}\frac{\scale(r)}{r}
		\end{equation}
		which implies that \begin{equation}
			\sum_{{j\in\{0\}\cup\bN}}\frac{\scale(r_{j})}{r_{j}}\overset{\eqref{e.Pos<4}}{\lesssim} \frac{\scale(r)}{r}\sum_{j\in\{0\}\cup\bN}A^{-(\beta_{1}-1)}2^{-j(\beta_{1}-1)}\lesssim\frac{\scale(r)}{r}
		\end{equation}
		which, combined with \eqref{e.Pos3+}, gives that, for all $y\in B(x,r)\cap\skeleton$ and every $z\in (B(y,\iota\wedge (2^{-10}A^{-1}r))\setminus\{y\})\cap \skeleton$, we have\begin{equation}
			\frac{\abs{v(y)-v(z)}}{\metric(y,z)}\lesssim \frac{\scale(r)}{r}\norm{f}_{L^{\infty}(B(x,A_{\mathrm{RH}}r),\meas)}
		\end{equation}
		Letting $z\to y$, we see that \begin{equation}
			\norm{\wgrad v}_{L^{\infty}(B(x,r),\medm)}\lesssim\frac{\scale(r)}{r}\norm{f}_{L^{\infty}(B(x,A_{\mathrm{RH}}r),\meas)},
		\end{equation}
		which, combined with \eqref{e.Pos-4} and the triangle inequality \[\norm{\wgrad u}_{L^{\infty}(B(x,r),\medm)}\leq \norm{\wgrad h}_{L^{\infty}(B(x,r),\medm)}+\norm{\wgrad v}_{L^{\infty}(B(x,r),\medm)},\] gives \eqref{e.Pos-cm}.
 \end{proof}
\begin{theorem}\label{t.RH>Gd}
Assume $\UVG$, $\UHK$ and $\RH$. Then $\Grad$ holds.
\end{theorem}
\begin{proof}
	
Fix $(t,x)\in(0,\infty)\times\ambient$. By Lemma \ref{l.meab}-\ref{it.meab1}, we know that $p_{t,x}=p_{t}(x,\cdot)\in\Dom(-\gen)$, and therefore, for any $(z,r)\in\ambient\times(0,\infty)$, and any $\phi\in \domain\cap C_{c}(B(z,A_{\mathrm{RH}}r))$, \begin{equation}
	\form(p_{t,x},\phi)=\langle-\gen p_{t,x},\phi\rangle_{L^{2}(\ambient,\meas)}\overset{\eqref{e.dif-gen}}{=}-\int_{B(z,A_{\mathrm{RH}}r)}\frac{\dif}{\dif t}p_{t}(x,y)\cdot\phi(y)\dif\meas(y).
\end{equation}
By Proposition \ref{p.Poisson}, \begin{equation}
	\norm{\wgrad p_{t,x}}_{L^{\infty}(B(z,r),\medm)}\lesssim\frac{1}{r}\fint_{B(z,A_{\mathrm{RH}}r)}p_{t,x}\dif m+\frac{\scale(r)}{r}\norm{\frac{\dif}{\dif t}p_{t}(x,\cdot)}_{L^{\infty}(B(z,A_{\mathrm{RH}}r),\meas)}.\label{e.sufgrd}
\end{equation}
 
Take $r=\scale^{-1}(t)$ in \eqref{e.sufgrd}. We now estimate the two terms on the right hand side of \eqref{e.sufgrd}. If $\metric(x,z)\leq 2A_{\mathrm{RH}}\scale^{-1}(t)$, then by \eqref{e.phi1} and \eqref{e.scale}, $\Phi(\metric(x,z),t)\lesssim1$, and therefore, \begin{equation}\label{e.sufgrad1}
	\exp(-c\Phi(\metric(x,y),t))\leq 1\lesssim \exp(-c\Phi(\metric(x,z),t)),\ \text{ for any $y\in B(z,A_{\mathrm{RH}}\scale^{-1}(t))$}.
\end{equation}
If $\metric(x,z)\geq 2A_{\mathrm{RH}}\scale^{-1}(t)$, then for any $y\in B(z,A_{\mathrm{RH}}\scale^{-1}(t))$, $\metric(x,y)\geq \metric(x,z)-\metric(y,z)\geq\frac{1}{2}\metric(x,z)$, so by \eqref{e.phi2}, \begin{equation}
	\exp(-c\Phi(\metric(x,y),t))\leq \exp(-c_{1}\Phi(\metric(x,z),t)),\ \text{ for any $y\in B(z,A_{\mathrm{RH}}\scale^{-1}(t))$}.\label{e.sufgrad2}
\end{equation}
Combining the above two cases, we have
\begin{align}
	&\phantom{\ \leq}\int_{B(z,A_{\mathrm{RH}}\scale^{-1}(t))}p_{t}(x,y)\dif m(y)\\
	&\overset{\UHK}{\lesssim}\frac{1}{\meas(B(x,\scale^{-1}(t)))}\int_{B(z,A_{\mathrm{RH}}\scale^{-1}(t))}\exp(-c\Phi(\metric(x,y),t))\dif m(y)\\
	&\overset{\eqref{e.sufgrad1},\eqref{e.sufgrad2}}{\lesssim}\frac{1}{\meas(B(x,\scale^{-1}(t)))}\int_{B(z,A_{\mathrm{RH}}\scale^{-1}(t))}\exp(-c_{2}\Phi(\metric(x,z),t))\dif m(y)\\
	&\overset{\UVG}{\lesssim}\exp(-c_{2}\Phi(\metric(x,z),t)).\label{e.sufgrad3}
\end{align}

For the second term of \eqref{e.sufgrd}, we use \eqref{e.ddt<} and obtain \begin{align}
	\norm{\frac{\dif}{\dif t}p_{t}(x,\cdot)}_{L^{\infty}(B(z,A_{\mathrm{RH}}\scale^{-1}(t)),\meas)}&\lesssim \frac{1}{t\volume(\scale^{-1}(t))}\sup_{y\in B(z,A_{\mathrm{RH}}\scale^{-1}(t))} \exp\left(-c_{3}\Phi\left({\metric(x, y)},{t}\right)\right)\\
	&\overset{\eqref{e.sufgrad1},\eqref{e.sufgrad2}}{\lesssim}\frac{1}{t\volume(\scale^{-1}(t))}\exp(-c_{4}\Phi(\metric(x,z),t)).\label{e.sufgrad4}
\end{align}
Combining $\UVG$, \eqref{e.sufgrad3}, \eqref{e.sufgrad4}, and \eqref{e.sufgrd} with $r=\scale^{-1}(t)$, we see that\begin{align}
	&\phantom{\ \leq}\norm{\wgrad p_{t,x}}_{L^{\infty}(B(z,\scale^{-1}(t)),\medm)}\\
	&\overset{\eqref{e.sufgrd}}{\lesssim}\frac{1}{\scale^{-1}(t)}\fint_{B(z,A_{\mathrm{RH}}\scale^{-1}(t))}p_{t,x}\dif m+\frac{t}{\scale^{-1}(t)}\norm{\frac{\dif}{\dif t}p_{t}(x,\cdot)}_{L^{\infty}(B(z,A_{\mathrm{RH}}\scale^{-1}(t)),\meas)}\\
	&\overset{\eqref{e.sufgrad3}, \eqref{e.sufgrad4}}{\lesssim} \frac{\exp(-c_{2}\Phi(\metric(x,z),t))}{\scale^{-1}(t)\volume(\scale^{-1}(t))}+\frac{t}{\scale^{-1}(t)}\frac{\exp(-c_{4}\Phi(\metric(x,z),t))}{t\volume(\scale^{-1}(t))}\\
	&\lesssim \frac{1}{\scale^{-1}(t)\meas(B(x,\scale^{-1}(t)))}\exp(-c_{5}\Phi(\metric(x,z),t))\label{e.sufgrad5}
\end{align} 

Let $Z$ be a countable dense subset of $(\ambient,\metric)$. Apply \eqref{e.sufgrad5} for every $z\in Z$, we obtain a $\medm$-measurable subset $N\subset\skeleton$ with $\medm(N)=0$, such that if $y\in (\skeleton \cap  B(z,\scale^{-1}(t)))\setminus N$ for some $z\in Z$,  then \begin{equation}
	\abs{\wgrad p_{t,x}(y)}\lesssim \frac{1}{\scale^{-1}(t)\meas(B(x,\scale^{-1}(t)))}\exp(-c_{5}\Phi(\metric(x,z),t)).
\end{equation}
For any $y\in\skeleton\setminus N$, since $Z$ is dense, we can choose a sequence $\{z_{j}\}_{j\in\bN}\subset Z$ such that $y\in B(z_{j},\scale^{-1}(t))\setminus N$  and $z_{j}\to y$. Since $\lim_{j\to\infty}\metric(x,z_{j})=\metric (x,y)$, by the lower semicontinuity of $\Phi$ in Lemma \ref{l.phi}, we obtain that \begin{equation}
	\abs{\wgrad p_{t,x}(y)}\lesssim \frac{1}{\scale^{-1}(t)\meas(B(x,\scale^{-1}(t)))}\exp(-c_{6}\Phi(\metric(x,y),t)),\ \text{ for all }y\in \skeleton\setminus N,
\end{equation}
which is exactly $\Grad$.
\end{proof}
\subsection{Geometric group actions on uniform local trees}
In this subsection, we use group actions to give a sufficient \emph{geometric} condition to ensure the heat kernel estimate $\HKE$, heat kernel gradient estimate $\Grad$ and reverse H\"{o}lder inequality $\RH$. This condition is relatively easy to verify. For Riemannian coverings of compact manifolds with nilpotent deck group, Dungey \cite{Dun04} proved heat kernel gradient estimates $\Grad$ and the boundedness of Riesz transform $\R{p}$. The method we use here in adapted from \cite[Appendix~A.5]{CJKS20}. For the reader's convenience, we recall the definition of a group action on a metric space. We refer to \cite{dlH00} for more details.
\begin{definition}
Let $G$ be a group, and let $(\ambient,\metric)$ be a metric space. 
\begin{enumerate}
[label=\textup{(\alph*)},align=right,leftmargin=*,topsep=5pt,parsep=0pt,itemsep=2pt]
\item We say that \emph{$G$ acts on $\ambient$}, if there exists a map $\phi:G\times \ambient\to \ambient$, which is abbreviated as $\phi(g,x)=g(x)$ such that $(gh)(x)=g(h(x))$ for every $g,h\in G$ and every $x\in \ambient$, and $e(x)=x$ for all $x\in\ambient$, where $e\in G$ is the unit of $G$.
\item We say that \emph{$G$ acts on $(\ambient,\metric)$ by isometries}, if $G$ acts on $\ambient$ and for each $g\in G$, $\phi(g,\cdot):\ambient\to \ambient$ is an isometry on $(\ambient,\metric)$.
\item An isometric action of $G$ on $(\ambient,\metric)$ is called \emph{properly discontinuous}, if \[\text{ for every compact $K\subset \ambient$, $\#\{g\in G: g(K)\cap K\neq\emptyset\}<\infty$. }\]
\item An action of $G$ on $(\ambient,\metric)$ is called \emph{co-compact} if there exists a compact set $K\subset \ambient$ such that $\ambient=\bigcup_{g\in G}g(K)$.
\item We say that $G$ acts on $(\ambient,\metric)$ \emph{geometrically}, if $G$ acts on $(\ambient,\metric)$ by isometries, and is properly discontinuous and co-compact.
\end{enumerate}
\end{definition}
The following lemma shows that a finitely generated group acting properly discontinuously by isometries on a proper metric space, with bounded displacement of its generators, is virtually abelian and hence has polynomial growth. We refer to \cite{dlH00} for the definition of \emph{growth} of groups.
\begin{lemma}\label{l.G-center}
Let $G$ be a group that is finitely generated by $S$, and is acting properly discontinuously by isometries on a proper metric space $(\ambient,\metric)$. If \begin{equation}\label{e.G-center-displacement}
	\sup_{(g,x)\in S\times\ambient}\metric(x,g(x))<\infty,
\end{equation} then $G$ is virtually abelian. In particular, $G$ is of polynomial growth.
\end{lemma}
\begin{proof}
Let $D:=\sup_{(g,x)\in S\times\ambient}\metric(x,g(x))<\infty$. Proper discontinuity implies that, for every $x\in\ambient$ and every $R\in(0,\infty)$,\begin{equation}\label{e.G-center-finite-displacement}
	\#\Sett{g\in G}{\metric(x,gx)\leq R}\leq \#\Sett{g\in G}{g(\overline{B(x,R+1)})\cap \overline{B(x,R+1)}\neq\emptyset}<\infty.
\end{equation}
For every $s\in S$, let $C_G(s)=\Sett{g\in G}{gs=sg}$ be the centralizer. Then \begin{equation}
	[G:C_G(s)]=\#\{hsh^{-1}:h\in G\}\leq \#\Sett{g\in G}{\metric(x,gx)\leq D}\overset{\eqref{e.G-center-finite-displacement}}{<}\infty.\label{e.G-cetfdis}
\end{equation}
Let $Z(G):=\Sett{g\in G}{gh=hg\text{ for every }h\in G}$ be the centre of $G$. Then $Z(G)=\bigcap_{s\in S}C_G(s)$. Therefore
\begin{equation}\label{e.G-center-finite-index}
[G:Z(G)]\leq\prod_{s\in S}[G:C_G(s)]\overset{\eqref{e.G-cetfdis}}{<}\infty.
\end{equation}
Since $Z(G)$ is abelian (thus nilpotent), this proves that $G$ is virtually abelian and virtually nilpotent. By Wolf's theorem \cite{Wol68} (see also \cite{Gro81}), $G$ is of polynomial growth.
\end{proof}

\begin{theorem}\label{t.Gp}
	Assume that $(\ambient,\metric,\meas,\form,\domain)$ satisfies $\UVG$. Suppose there exists a group $G$, such that \begin{enumerate}[label=\textup{(\roman*)},align=right,leftmargin=*,topsep=5pt,parsep=0pt,itemsep=2pt]
  \item\label{it.G-geo} $G$ acts on $(\ambient,\metric)$ {geometrically}, which implies that $G$ is finitely generated by a symmetric set $S=S^{-1}$;
  \item\label{it.G-dis} The generating set $S$ acts on $(\ambient,\metric)$ with \emph{bounded displacement}, that is, \begin{equation}\label{e.G-dis}
  	\sup_{(g,x)\in S\times \ambient}\metric(x,g(x))<\infty;
  \end{equation}
   \item\label{it.G-pre} $G$ preserves $(\form,\domain)$, that is, for every $f\in\domain$ and every $g\in G$, we have $f\circ g\in\domain$ and $\form(f\circ g,f\circ g)=\form(f,f)$.
\end{enumerate}
Then the following hold:
\begin{enumerate}[label=\textup{(\alph*)},align=right,leftmargin=*,topsep=5pt,parsep=0pt,itemsep=2pt]
\item\label{it.G-Vol} There exist constants $C\in(1,\infty)$ and $A\in(1,\infty)$ such that \begin{equation}
	C^{-1}\#B^{G}(e,A^{-1}r)\leq \volume(r)\leq C\#B^{G}(e,Ar),\ \text{ for all }r\in[1,\infty).
\end{equation}
Here $B^G(e,r)$ denotes the ball of radius $r$ centred at the identity $e$ in the Cayley graph of $G$ with generating set $S$ and its graph metric. In particular, there exists a constant  $C_{1}\in(1,\infty)$ and a unique $D\in\bN$ such that \begin{equation}\label{e.G-pol}
	C_{1}^{-1}r^{D}\leq V(r)\leq C_{1} r^{D},\ \text{ for all }r\in[1,\infty).
\end{equation}
  \item\label{it.G-HKE} $(\ambient,\metric,\meas,\form,\domain)$ satisfies $\HKE$ with \begin{equation}\label{e.G-scale}
	\scale(r)=r\cdot\left(r\wedge\int_{0}^{r}\frac{V(s)}{s}\dif s\right)\asymp
	\begin{dcases}
		rV(r),\ &r\in(0,\iota]\\
		r^{2},\ &r\in(\iota,\infty).
	\end{dcases}
\end{equation}
In particular, $(\ambient,\metric,\meas,\form,\domain)$ is eventually Gaussian \ref{e.EventGau}.
\item\label{it.G-RHG} $(\ambient,\metric,\meas,\form,\domain)$ satisfies $\RH$ and $\Grad$.
\end{enumerate}
\end{theorem}
\begin{proof}
	We first record an implication of the assumption \ref{it.G-geo}. By the \v{S}varc--Milnor lemma \cite[Theorem~IV.B.25]{dlH00}, $G$ is generated by a finite symmetric set $S$ and the Cayley graph of $G=\langle S\rangle$, equipped with graph metric $d_{S}$, is quasi-isometric to $(\ambient,\metric)$. To be precise, for any $x\in\ambient$, there is a constant $C\in (1,\infty)$ such that the map $g\mapsto g(x)$ is \emph{coarsely bi-Lipschitz} between $(G,d_{S})$ and $(\ambient,\metric)$, that is, \begin{equation}
		C^{-1}d_{S}(g_{1},g_{2})-C\leq \metric (g_{1}(x),g_{2}(x))\leq Cd_{S}(g_{1},g_{2})+C,\ \text{ for all }g_{1},g_{2}\in G,
	\end{equation}
	and the image is \emph{relatively dense} $\ambient=\bigcup_{g\in G}B(g(x),C)$.
	By $\UVG$, there exists $C_{1},C_{2}\in(0,\infty)$, depending on $V(1)$, the constants in $\UVG$, and $\# S$, such that \begin{equation}
		C_{2}^{-1}\# B^{G}(e,C_{1})\leq \meas(B(g(x),C_{1}))\leq C_{2} \# B^{G}(e,C_{1}).
	\end{equation}
	Therefore, $(\ambient,\metric,\meas)$ and $(G,d_{S},\#)$ are \emph{roughly isometric} in the sense of \cite[Definition~2.20]{BBK06}.

	\begin{enumerate}[label=\textup{(\alph*)},align=right,leftmargin=*,topsep=5pt,parsep=0pt,itemsep=2pt]
\item[\ref{it.G-Vol}] The first assertion is a direct consequence of the rough isometry between $(\ambient,\metric,\meas)$ and $(G,d_{S},\#)$ and \cite[Proposition~2.2]{CSC95}. The second assertion \eqref{e.G-pol} is implied by \cite[Theorem~1.1]{Bre14}, \ref{it.G-dis} and Lemma \ref{l.G-center}.
\item[\ref{it.G-HKE}] By \cite[Section~3]{BCY25}, we know that the \emph{local Poincar\'{e} inequality} $\mathrm{PI}(\Psi)_{\loc}$ and the \emph{local cutoff Sobolev inequality} $\mathrm{CS}(\Psi)_{\loc}$ defined in \cite[p.~494]{BBK06} holds for the metric measure Dirichlet space $(\ambient,\metric,\meas,\form,\domain)$. By \ref{it.G-Vol} and \cite[Theorem 2.2]{Kle10}, the Cayley graph $(G,d_{S})$ satisfies the Poincar\'{e} inequality with scale function $r\mapsto r^{2}$.  By \cite[Remarks 1.6-(2)]{BB04}, the cutoff Sobolev inequality $\mathrm{CS}(2)$ defined in \cite[Definition 1.4]{BB04} holds for $(G,d_{S})$. The  arguments in \cite[Section~5, proof of Proposition~5.5]{BBK06} also apply to the present scale function $\scale$, since the comparisons require only \eqref{e.scale} and $\UVG$. Therefore, by \ref{it.G-Vol}, the rough isometry between $(\ambient,\metric,\meas)$ and $(G,d_{S},\#)$, \cite[Theorem 2.21]{BBK06} and \cite[Theorem 1.2]{GHL15}, we know that $(\ambient,\metric,\meas,\form,\domain)$ satisfies $\HKE$ with $\Psi$ given in \eqref{e.G-scale}.
\item[\ref{it.G-RHG}] We begin with the proof of $\RH$. The proof is motivated by \cite[Theorem~A.2]{CJKS20}. By $\HKE$, $\UVG$ and \cite[Theorem~1.2]{GHL15}, we know that the \emph{elliptic Harnack inequality} \ref{e.EHI} holds for $(\ambient,\metric,\meas,\form,\domain)$, that is \begin{equation}\label{e.EHI}\tag{$\mathrm{EHI}$}
	{\begin{minipage}{370pt}
		there exist constants $A,C\in(1,\infty)$ such that, for all $(x_{0},r)\in\ambient\times(0,\infty)$, all $u\in\domain_{\loc}$ that is non-negative and $\form$-harmonic on $B(x_{0},Ar)$, we have $\esup_{B(x_{0},r)}u\leq C\einf_{B(x_{0},r)}u$.		
	\end{minipage}}
\end{equation}

An immediate consequence of \ref{e.EHI} is that all $\form$-harmonic functions have a continuous $\meas$-version, and there exist $A,C\in(1,\infty)$ and $\theta\in(0,1]$ such that, for any $(x_{0},r)\in\ambient\times(0,\infty)$ and any $u\in\domain_{\loc}$ that is $\form$-harmonic in $B(x_{0},Ar)$, we have \begin{equation}\label{e.EHI1}
	\abs{u(x)-u(y)}\leq C\left(\frac{\metric(x,y)}{r}\right)^{\theta}\fint_{B(x_{0},Ar)}\abs{u}\dif\meas,\ \text{ for every $x,y\in B(x_{0},r)$.}
\end{equation}
Here and in the following, we represent every $\form$-harmonic function by its continuous $\meas$-version. See \cite[Lemma~A.1]{BBKT08} and \cite[Theorem~6.3,~Lemma~9.2]{GHL15}.
Define \begin{equation}
	\delta(g):=\sup_{x\in\ambient}\metric(x,g(x)),\ \text{ for each $g\in G$}.
\end{equation}
By \eqref{e.EHI1}, for $A_{1}=2A\in(1,\infty)$ and $\theta\in(0,1]$, the following  \hypertarget{PH}{property}, denoted by $\hyperlink{PH}{P(A_{1},\theta)}$, holds: 

\begin{equation}\label{e.PEHI}
	{\begin{minipage}{370pt}
		there exists a constant $C_{1}\in(1,\infty)$ such that, for all $(x_{0},r)\in\ambient\times(0,\infty)$, all $g\in G$ with $\delta(g)\in[0,r)$, and all $u\in\domain_{\loc}$ that is $\form$-harmonic on $B(x_{0},A_{1}r)$, we have \vspace{-15pt}\begin{equation}
		\abs{u(g(z))-u(z)}\leq C_{1}\left(\frac{\delta(g)}{r}\right)^{\theta}\fint_{B(x_{0},A_{1}r)}\abs{u}\dif\meas,\ \text{for all $z\in B(x_{0},r)$}.
		\end{equation}
	\end{minipage}}
\end{equation}
In fact, for any $g\in G$, if $\delta(g)=0$ then $g(z)=z$ for all $z\in\ambient$; if $\delta(g)\in(0,r)$ and $z\in B(x_{0},r)$, we know that $\metric(g(z),z)\leq\delta(g)<r$, so $\{z,g(z)\}\subset B(x_{0},2r)$. So we may apply \eqref{e.EHI1} with $2r$ in place of $r$. Define a function \begin{equation}\label{e.EHI1.1}
h(\theta)=2\theta\one_{(0,2^{-1})}(\theta)+\frac{3}{4}\one_{\{2^{-1}\}}(\theta)+\one_{(2^{-1},1]}(\theta),\ \text{ for }\theta\in(0,1].
\end{equation} We will prove that, for every $A_{1}\geq 2A$ such that $\hyperlink{PH}{P(A_{1},\theta)}$ holds, the property can be improved to $\hyperlink{PH}{P(A_{2},h(\theta))}$ for some $A_{2}\in(1,\infty)$. For this, we first fix $(x_{0},r)\in\ambient\times(0,\infty)$, fix $g\in G$ with $\delta(g)\in(0,r)$, and fix $u\in\domain_{\loc}$ that is $\form$-harmonic on $B(x_{0},A_{2}r)$ for $A_{2}=A_{1}(A_{1}\vee 3)+2$. If $\delta(g)\in[\frac{1}{4}r,r)$, then since $A_{2}>A_{1}\geq2A$, \begin{align}
	\abs{u(g(z))-u(z)}&\leq 2\sup_{B(x_{0},2r)}\abs{u}\leq 2C\fint_{B(x_{0},2Ar)}\abs{u}\dif\meas\\
	&\leq 2C\cdot 4^{2\theta}\left(\frac{\delta(g)}{r}\right)^{2\theta}\fint_{B(x_{0},2Ar)}\abs{u}\dif\meas.\label{e.EHI2+}
\end{align}
Thus the claim holds automatically. If $\delta(g)\in(0,\frac{1}{4}r)$, then there exists a unique $k\in\{0\}\cup\bN$ such that $\delta(g)\in(2^{-k-3}r,2^{-k-2}r]$. For each $j\in\{0,\ldots,k+1\}$, we define $g_{j}:=g^{2^{j}}\in G$ and define $w_{j}:=(u\circ g_{j})-u$. Since $G$ acts on $(\ambient,\metric)$ by isometries, we know that $\delta(g_{j})\leq 2^{j}\delta(g)<r$. Since $G$ acts on $(\ambient,\metric)$ by isometries and $G$ preserves $(\form,\domain)$ by \ref{it.G-pre}, we know that $(u\circ g_{j})\in\domain_{\loc}$ is $\form$-harmonic on $B(g_{j}^{-1}(x_{0}),A_{2}r)$. Since $A_{2}\geq A_{1}^{2}+2$ and $\delta(g_{j}^{-1})=\delta(g_{j})\leq\frac{1}{2}r$, we know that \begin{equation}\label{e.EHI3}
	B(x_{0},A_{1}^{2}r)\subset B(x_{0},A_{2}r)\cap B(g_{j}^{-1}(x_{0}),A_{2}r).
\end{equation}
Since $u$ is $\form$-harmonic on $B(x_{0},A_{1}^{2}r)$, applying $\hyperlink{PH}{P(A_{1},\theta)}$ to $u$, we see that \begin{align}
	\sup_{z\in B(x_{0},A_{1}r)}\abs{w_{j}(z)}&=\sup_{z\in B(x_{0},A_{1}r)}\abs{u(g_{j}(z))-u(z)}\leq C_{1}\left(\frac{\delta(g_{j})}{A_{1}r}\right)^{\theta}\fint_{B(x_{0},A_{1}^{2}r)}\abs{u}\dif\meas\\
&\leq C_{1}A_{1}^{-\theta}2^{\theta j}\left(\frac{\delta(g)}{r}\right)^{\theta}\fint_{B(x_{0},A_{1}^{2}r)}\abs{u}\dif\meas.\label{e.EHI4}
\end{align}
By \eqref{e.EHI3}, $w_{j}$ is $\form$-harmonic on $B(x_{0},A_{1}r)$. So we may also apply $\hyperlink{PH}{P(A_{1},\theta)}$ to $w_{j}$ on $B(x_{0},A_{1}r)$ and see that, for all $z\in B(x_{0},r)$, and all $j\in\{0,\ldots, k\}$, \begin{align}
	&\phantom{\ \leq}\abs{u(g_{j+1}(z))-2u(g_{j}(z))+u(z)}\\
	&=\abs{w_{j}(g_{j}(z))-w_{j}(z)}\leq C_{1}\left(\frac{\delta(g_{j})}{r}\right)^{\theta}\fint_{B(x_{0},A_{1}r)}\abs{w_{j}}\dif\meas\\
	&\overset{\eqref{e.EHI4}}{\leq}C_{1}^{2}A_{1}^{-\theta}2^{2\theta j}\left(\frac{\delta(g)}{r}\right)^{2\theta}\fint_{B(x_{0},A_{1}^{2}r)}\abs{u}\dif\meas\label{e.EHI4.1}
\end{align}
The definition of $k$, we have $1/4<2^{k+1}\delta(g)/r\leq1/2$ and $\delta(g_{k+1})\leq r/2$. For $z\in B(x_0,r)$, we have $\{z,g_{k+1}(z)\}\subset B(x_0,2r)$. By the same argument in \eqref{e.EHI2+}, $\UVG$ and the fact that $A_1^2\geq2A$, we have
\begin{align}
\abs{u(g_{k+1}(z))-u(z)}
&\leq2\sup_{B(x_0,2r)}\abs{u}
\leq C\fint_{B(x_0,A_1^2r)}\abs{u}\dif\meas\\
&\leq C4^{2\theta}2^{2\theta(k+1)}
\left(\frac{\delta(g)}r\right)^{2\theta}
\fint_{B(x_0,A_1^2r)}\abs{u}\dif\meas.\label{e.EHI4.2}
\end{align}
Using the following identity, \begin{align}
&\phantom{\ \leq}u(z)-u(g(z))\\
	&=2^{-k-1}(u(z)-u(g_{k+1}(z)))+\sum_{j=0}^{k}2^{-j-1}(u(z)-2u(g_{j}(z))+u(g_{j+1}(z))),\label{e.EHI5}
\end{align}
we have \begin{align}
	&\phantom{\ \leq}\abs{u(z)-u(g(z))}\\
	&\overset{\eqref{e.EHI5}}{\leq}2^{-k-1}\abs{u(z)-u(g_{k+1}(z))}+\sum_{j=0}^{k}2^{-j-1}\abs{u(z)-2u(g_{j}(z))+u(g_{j+1}(z))}\\
	&\overset{\eqref{e.EHI4.1},\eqref{e.EHI4.2}}{\leq}C(C_{\volume},\theta,A_{1})\sum_{j=0}^{k+1}2^{(2\theta-1) j}\left(\frac{\delta(g)}{r}\right)^{2\theta}\fint_{B(x_{0},A_{1}^{2}r)}\abs{u}\dif\meas.\label{e.EHI6}
\end{align}
\begin{itemize}
	\item If $\theta\in(0,\frac{1}{2})$, then $\sum_{j=0}^{k+1}2^{(2\theta-1) j}\leq \sum_{j=0}^{\infty}2^{(2\theta-1) j}=C(\theta)<\infty$. So \begin{equation}
		\abs{u(z)-u(g(z))}\leq C_{1}(C_{\volume},\theta,A_{1})\left(\frac{\delta(g)}{r}\right)^{2\theta}\fint_{B(x_{0},A_{1}^{2}r)}\abs{u}\dif\meas.
	\end{equation}
	This gives $\hyperlink{PH}{P(A_{2},2\theta)}$
	\item If $\theta=\frac{1}{2}$, then using the inequality $\sup_{s\in[1,\infty)}s^{-\frac{1}{4}}\log s<\infty$, we have \begin{align}
		&\phantom{\ \leq}\sum_{j=0}^{k+1}2^{(2\theta-1) j}\left(\frac{\delta(g)}{r}\right)^{2\theta}=(k+2)\left(\frac{\delta(g)}{r}\right)\\
		&\leq \frac{1}{\log 2}\left(\log\frac{r}{\delta(g)}\right)\left(\frac{\delta(g)}{r}\right)^{\frac{1}{4}}\left(\frac{\delta(g)}{r}\right)^{\frac{3}{4}}\leq C\left(\frac{\delta(g)}{r}\right)^{\frac{3}{4}},
	\end{align}
	which, combined with \eqref{e.EHI6}, gives $\hyperlink{PH}{P(A_{2},\frac{3}{4})}$.
	\item If $\theta\in(\frac{1}{2},1]$, then $2\theta>1$, so \begin{align}
		&\phantom{\ \leq}\sum_{j=0}^{k+1}2^{(2\theta-1) j}\left(\frac{\delta(g)}{r}\right)^{2\theta}\leq C(\theta)2^{(2\theta-1) k}\left(\frac{\delta(g)}{r}\right)^{2\theta}\\
		&\leq C(\theta)\left(\frac{r}{4\delta(g)}\right)^{2\theta-1}\left(\frac{\delta(g)}{r}\right)^{2\theta}=C_{1}(\theta)\frac{\delta(g)}{r},
	\end{align}
	which, combined with \eqref{e.EHI6}, gives $\hyperlink{PH}{P(A_{2},1)}$.
\end{itemize}
Combining the above three cases and the definition of $h$ in \eqref{e.EHI1.1}, we obtain $\hyperlink{PH}{P(A_{2},h(\theta))}$. Note that for any $\theta\in(0,1]$, there exists $N(\theta)\in\bN$ such that $h^{\circ N(\theta)}(\theta)=1$. Iterating the self-improvement $\hyperlink{PH}{P(A_{1},\theta)} \Longrightarrow \hyperlink{PH}{P(A_{2},h(\theta))}$ total of $N(\theta)$ times, we obtain $\hyperlink{PH}{P(A_{*},1)}$, for some $A_{*}\in(1,\infty)$ depending only on $A_{1}$ and $\theta$.

By the co-compactness of the action $G$ on $(\ambient,\metric)$, there exists a non-empty compact set $K\subset\ambient$ such that $\ambient=\bigcup_{g\in G}g(K)$. Fix $q_{0}\in K$. Fix $L_{1}:=2\diam(K,\metric)$ so that $K\subset B(q_{0},L_1)$. By \ref{e.EHI} and \cite[Lemma~A.1]{BBKT08}, there exists $L_{2}\in(A_{*}L_{1},\infty)$, depending only on $A_{*}$, $L_{1}$ and the constants appearing in \ref{e.EHI}, such that, for any $u\in\domain_{\loc}$ that is $\form$-harmonic on $B(q_{0},AL_{2})$,\begin{equation}\label{e.EHI6.1}
	\osc_{K}u\leq \osc_{B(q_{0},L_1)}u\leq \frac{1}{8}\osc_{B(q_{0},L_2)}u,
\end{equation}
where $A\in(1,\infty)$ is the constant from \ref{e.EHI}. Let \begin{equation}\label{e.EHI7}
	J_{0}:=\Sett{g\in G}{g(K)\cap\overline{B(q_{0},L_2)}\neq\emptyset},\ \text{ and }\ \Omega:=\bigcup_{g\in J_{0}}g(K).
\end{equation}
By the proper discontinuity of the action, we know that $\# J_{0}<\infty$, which, combined with \eqref{e.G-dis}, gives \begin{equation}\label{e.EHI7.1}
	M_{0}:=\max_{g\in J_{0}}\delta(g)<\infty.
\end{equation}

 Since $G$ acts on $(\ambient,\metric)$ by isometries, we know that $\Omega\subset B(q_{0},L_{1}+L_{2})$.  Since $\ambient=\bigcup_{g\in G}g(K)$, we have \begin{equation}\label{e.EHI8}
K\subset B(q_{0},L_2)\subset	\ol{B(q_{0},L_2)}=\ol{B(q_{0},L_2)}\cap\ambient=\bigcup_{g\in G}(g(K)\cap \ol{B(q_{0},L_2)})\subset \Omega.
\end{equation}
For any $v:\Omega\to \bR$, we write \begin{equation}\label{e.EHI8.1}
	D[v]:=\max_{g\in J_{0}}\sup_{x\in K}\abs{v(g(x))-v(x)}.
\end{equation}
Then for any $x_{1},x_{2}\in \Omega$, by the definition of $\Omega$ in \eqref{e.EHI7}, there exists $g_{1},g_{2}\in J_{0}$ and $y_{1},y_{2}\in K$ such that $x_{j}=g_{j}(y_{j})$, $j\in\{1,2\}$. In this case, \begin{align}
	\abs{v(x_{1})-v(x_{2})}&\leq\abs{v(g_{1}(y_{1}))-v(y_{1})}+\abs{v(y_{1})-v(y_{2})}+\abs{v(g_{2}(y_{2}))-v(y_{2})}\\
	&\leq 2D[v]+\osc_{K}v.
\end{align}
Taking the supremum over $x_{1},x_{2}\in \Omega$, we obtain $\osc_{\Omega}v\leq 2D[v]+\osc_{K}v$. If $u\in\domain_{\loc}$ is an $\form$-harmonic function on $B(q_{0},(A+1)L_{2})\supset \Omega$, then, using \eqref{e.EHI6.1}, we obtain \begin{equation}
	\osc_{\Omega}u\leq 2D[u]+\osc_{K}u\overset{\eqref{e.EHI6.1}}{\leq} 2D[u]+\frac{1}{8}\osc_{B(q_{0},L_2)}u\overset{\eqref{e.EHI8}}{\leq} 2D[u]+\frac{1}{8}\osc_{\Omega}u.
\end{equation}
Therefore, \begin{equation}\label{e.EHI9}
	\osc_{\Omega}u\leq \frac{16}{7}D[u], \text{ for any $u\in\domain_{\loc}$ that is $\form$-harmonic on $B(q_{0},(A+1)L_{2})$}.
\end{equation}

For any $(x,r)\in\ambient\times(0,\infty)$, we define \begin{equation}
	J({x,r}):=\Sett{g\in G}{g(K)\cap B(x,r)\neq\emptyset}.
\end{equation}
Again, the proper discontinuity of $G$ and the properness of $(\ambient,\metric)$ imply that $\# J({x,r})<\infty$ for all $(x,r)\in\ambient\times(0,\infty)$. As in \eqref{e.EHI8}, we have \begin{equation}\label{e.EHI10}
	B(x,r)\subset \bigcup_{g\in J({x,r})}g(K),\ \text{ for all $(x,r)\in\ambient\times(0,\infty)$}.
\end{equation}

Let $u\in\domain_{\loc}$ be $\form$-harmonic on $B(x,A_{3}r)$ for some $A_{3}=A_{0}+3A_{*}+1$, where $A_{0}$ is the constant from the \emph{locally Lipschitz regular condition} in \cite[Lemma 4.9]{BCY25}.
\begin{itemize}
	\item Suppose $r\in(0,(6A_{0})^{-1}]$. By \cite[Lemma 4.9]{BCY25}, we know that there exists $C \in(1,\infty)$, such that\begin{equation}\label{e.G-LGH}
	\norm{\wgrad u}_{L^{\infty}(B(x,r),\medm)}\leq \frac{C}{r}\fint_{B(x,A_{0}r)}\abs{u}\dif\meas.
\end{equation}

	\item Suppose $r\in [(\frac{A+2}{A_{3}-1}L_{2})\vee L_{1}\vee A_{*}\vee M_{0} ,\infty)$. Let $g\in J({x,r})$. We will use $\hyperlink{PH}{P(A_{*},1)}$ and \eqref{e.EHI9} to estimate $\osc_{g(\Omega)}u$. Let $v:=u\circ g$. By \ref{it.G-geo} and \ref{it.G-pre}, we know that $v$ is $\form$-harmonic on $B(g^{-1}(x),A_{3}r)$. Let $z\in K\cap g^{-1}(B(x,r))$. For any $y\in B(q_{0},(A+1)L_{2})$, we have  \begin{align}
		\metric(y,g^{-1}(x))&\leq \metric(y,q_{0})+\metric(q_{0},z)+\metric(z,g^{-1}(x))\\
		&\leq (A+1)L_{2}+\diam(K,\metric)+r\\
		&< (A+2)L_{2}+r\leq A_{3}r. 
	\end{align}
	So $B(q_{0},(A+1)L_{2})\subset B(g^{-1}(x),A_{3}r)$. In particular, $v$ is $\form$-harmonic on $B(q_{0},(A+1)L_{2})$. By \eqref{e.EHI9}, \begin{align}
		&\phantom{\ \leq}\osc_{g(\Omega)}u=\osc_{\Omega}v\overset{\eqref{e.EHI9}}{\leq} \frac{16}{7}D[v]\overset{\eqref{e.EHI8.1}}{=}\frac{16}{7}\max_{h\in J_{0}}\sup_{z\in K}\abs{v(h(z))-v(z)}\\
		&\leq \frac{16}{7}\max_{h\in J_{0}}\sup_{z\in B(q_{0},L_{1})}\abs{v(h(z))-v(z)}=\frac{16}{7}\max_{h\in J_{0}}\sup_{z\in B(q_{0},L_{1})}\abs{u((gh)(z))-u(g(z))}\\
		&\leq \frac{16}{7}\max_{h\in J_{0}}\sup_{z\in B(x,3r)}\abs{u((ghg^{-1})(z))-u(z)}\ \text{ (since $B(g(q_{0}),L_{1})\subset B(x,3r)$)} \\
		&\lesssim \frac{M_{0}}{r}\fint_{B(x,3A_{*}r)}\abs{u}\dif m \ \text{ (using $\hyperlink{PH}{P(A_{*},1)}$, $\delta(ghg^{-1})=\delta(h)$ and \eqref{e.EHI7.1})}.\label{e.EHI10.1}
	\end{align}
	Let $\rho:=(12A_{0})^{-1}((L_{2}-L_{1})\wedge1)$. Let $N\in\bN$ and $\{q_{j}\}_{j=1}^{N}\subset K$ such that $K\subset \bigcup_{j=1}^{N}B(q_{j},\rho)$. Therefore \begin{equation}\label{e.EHI11}
		B(q_{j},A_{0}\rho )\subset B(q_{0},L_{2})\overset{\eqref{e.EHI8}}{\subset} \Omega.
	\end{equation}
	Since $\{B(q_{j},\rho)\}_{j=1}^{N}$ covers $K$, the inclusion in \eqref{e.EHI10} implies  \begin{equation}
		B(x,r)\subset \bigcup_{g\in J(x,r)}\bigcup_{j=1}^{N}g(B(q_{j},\rho)).
	\end{equation}
Therefore, using \eqref{e.G-LGH} with $u-u(g(q_{j}))$ in place of $u$, we have \begin{align}
	&\phantom{\ \leq}\norm{\wgrad u}_{L^{\infty}(B(x,r),\medm)}\\
	&\leq \max_{g\in J(x,r)}\max_{j\in\{1,\ldots,N\}}\norm{\wgrad (u-u(g(q_{j}))\one_{\ambient})}_{L^{\infty}(g(B(q_{j},\rho)),\medm)}\\
	&\overset{\eqref{e.G-LGH}}{\leq} \max_{g\in J(x,r)}\max_{j\in\{1,\ldots,N\}}\frac{C}{\rho}\fint_{B(g(q_{j}),A_{0}\rho)}\abs{u-u(g(q_{j}))}\dif\meas\\
	&\leq \frac{C}{\rho}\max_{g\in J(x,r)}\max_{j\in\{1,\ldots,N\}}\osc_{g(B(q_{j},A_{0}\rho ))}u\overset{\eqref{e.EHI11}}{\leq} \frac{C}{\rho}\max_{g\in J(x,r)}\osc_{g(\Omega)}u\\
	&\overset{\eqref{e.EHI10.1}}{\leq} \frac{C_{1}}{\rho}\frac{M_{0}}{r}\fint_{B(x,3A_{*}r)}\abs{u}\dif m.\label{e.G-LGH2}
\end{align}
	\item Suppose $r\in((6A_{0})^{-1},(\frac{A+2}{A_{3}-1}L_{2})\vee L_{1}\vee A_{*}\vee M_{0})$. Let $\rho_{1}=(12A_{0})^{-1}$. Then there exist $M\in\bN$ and $\{y_{j}\}_{j=1}^{M}\subset B(x,r)$ such that $B(x,r)\subset \bigcup_{j=1}^{M}B(y_{j},\rho_{1})$.
Then \begin{align}
	&\phantom{\leq}\norm{\wgrad u}_{L^{\infty}(B(x,r),\medm)}\\
	&\leq \max_{j\in\{1,\ldots, M\}}\norm{\wgrad u}_{L^{\infty}(B(y_{j},\rho_{1}),\medm)}\overset{\eqref{e.G-LGH}}{\leq}\frac{C}{\rho_{1}}\max_{j\in\{1,\ldots, M\}}\fint_{B(y_{j},A_{0}\rho_{1})}\abs{u}\dif\meas\\
	&\overset{\UVG}{\lesssim} \frac{\volume((A_{0}+1)((\frac{A+2}{A_{3}-1}L_{2})\vee L_{1}\vee A_{*}\vee M_{0}))}{ C^{-1}\rho_{1}\volume(A_{0}\rho_{1})}\fint_{B(x,(A_{0}+1)r)}\abs{u}\dif \meas.\label{e.G-LGH3}
\end{align}
\end{itemize}
Combining the above three cases \eqref{e.G-LGH}, \eqref{e.G-LGH2} and \eqref{e.G-LGH3}, and using $\UVG$, we obtain $\RH$. By Theorem \ref{t.RH>Gd} and \ref{it.G-HKE}, we obtain $\Grad$.
\qedhere
\end{enumerate} 
\end{proof}

\subsection{Example: alternating Vicsek fractafold}\label{Sec:examples}

In this section, we use Theorem \ref{t.Gp} to study the validity of $\R{p}$ and $\RR{p}$ for $p\in(1,\infty)$ for the alternating Vicsek fractafold, whose heat kernel is Gaussian at large scale, and is sub-Gaussian at small scale.

\begin{example}[Alternating Vicsek fractafold in $\bZ^{d}$]\label{ex.Vicsek}
Let $d\in\bN\cap[2,\infty)$. Let 
\begin{equation}
 L_d:=\Sett{\bfk\in\mathbb Z^d}{\bfk=(k_{1},\ldots,k_{d})\ \text{ with }\ k_1\equiv\cdots\equiv k_d\pmod 2}.
\end{equation}
Let $\mathbf{0}=(0,\ldots,0)\in\bZ^{d}$. For $\bfk\in L_d$, let $Q_{\bfk}:=\bfk+[0,1]^d$. Let $\bfc=(2^{-1},\ldots,2^{-1})\in\bR^{d}$ and let $A_d=\{0,1\}^d\cup\{\mathbf{c}\}$.  For $\bfa\in A_d$,
define
\begin{equation}
 F_{\bfa}(\mathbf{x}):=\bfa+\frac{1}{3}(\mathbf{x}-\bfa),\  \mathbf{x}\in\mathbb R^d.
\end{equation}
Let $K_{\mathbf{0}}$ be the unique non-empty compact subset of $Q_{\mathbf{0}}$, satisfying $K_{\mathbf{0}}=\bigcup_{\bfa\in A_d}F_{\bfa}(K_{\mathbf{0}})$. The existence and uniqueness of $K_{\mathbf{0}}$ follow from \cite[Section~3.1-(3)]{Hut81}. Let $\alpha_{d}:=\log_{3}(2^{d}+1)\in(1,d)$. Let $d_{{\mathbf{0}}}:K_{\mathbf{0}}\times K_{\mathbf{0}}\to[0,\infty)$ be the geodesic metric of $K_{\mathbf{0}}$ generated from the Euclidean metric. Let $K_{\bfk}:= \bfk+K_{\mathbf{0}}$ for each $\bfk\in L_{d}$. Then $d_{{\mathbf{k}}}(\cdot,\cdot):=d_{{\mathbf{0}}}(\cdot-\bfk,\cdot-\bfk)$ is a geodesic metric on $K_{\bfk}$. Let $m_{\mathbf{k}}:=\mathcal H^{\alpha_d}_{d_{\bfk}}$ be the $\alpha_d$-dimensional Hausdorff measure. There exists $C\in(1,\infty)$ such that \begin{equation}
	C^{-1}r^{\alpha_{d}}\leq m_{\bfk}(B_{K_{\bfk}}(x,r))\leq Cr^{\alpha_{d}}, \text{ for all }(x,r)\in K_{\bfk}\times(0,2] \text{ and all }\bfk\in L_{d}.
\end{equation} 
\begin{enumerate}[label=\textup{(\roman*)},align=right,leftmargin=*,topsep=5pt,parsep=0pt,itemsep=2pt]
\item Define $\ambient:=\bigcup_{\bfk\in L_{d}}K_{\bfk}$; see Figure \ref{fg.Vic-ffd} for the case $d=2$.
 \begin{figure}
 \centering
  \includegraphics[width=0.85\textwidth]{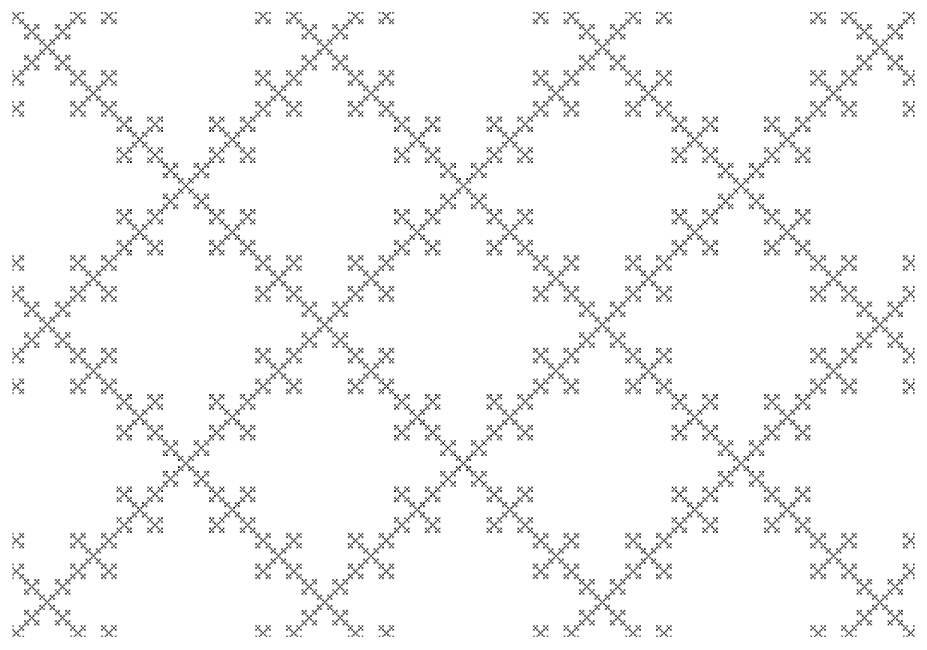}
 \caption{Alternating Vicsek fractafold in $\bZ^{2}$}
 \label{fg.Vic-ffd}
 \end{figure}
We define a function $\metric:\ambient\times\ambient\to[0,\infty)$ by \begin{equation}
\metric(x,y):=\inf\Biggl\{\sum_{j=0}^{n-1}d_{{\mathbf{k}_{j}}}(\mathbf{q}_{j},\mathbf{q}_{j+1})\Biggm|\begin{minipage}{230pt}
	$n\in\bN$, $(\bfk_{j})_{j=0}^{n-1}\subset L_{d}$, $x=\mathbf{q}_{0}\in K_{\bfk_{0}}$, $y=\mathbf{q}_{n}\in K_{\bfk_{n-1}}$, and  $\mathbf{q}_{j}\in K_{\bfk_{j-1}}\cap K_{\bfk_{j}}$ for all $j\in\{1,\ldots,n-1\}$ whenever $n\geq2$
	\end{minipage}
	\Biggr\}
\end{equation}
for all $(x,y)\in\ambient\times\ambient$. In other words, $\metric$ is the induced path metric on $\ambient$. We define a Borel measure $\meas$ on $\ambient$ by \begin{equation}
	\meas(E):=\sum_{\bfk\in L_{d}}m_{\bfk}(K_{\bfk}\cap E),\ \text{ for all Borel }E\subset \ambient.
\end{equation}
It is clear that $(\ambient,\metric)$ is an unbounded, proper, geodesic, and separable uniform local tree (with $\iota=4^{-1}$); $\meas$ is a Radon measure on $(\ambient,\metric)$ with full support, and the uniform volume growth condition $\UVG$ holds with \begin{equation}\label{e.VicUVG}
	\volume(r):=r^{\alpha_{d}}\vee r^{d},\ r\in(0,\infty).
\end{equation}
Let $(\form,\domain)$ be the canonical strongly local regular symmetric Dirichlet form on $L^{2}(\ambient,\meas)$ in Proposition \ref{p.energy}. Clearly
\begin{equation}\label{e.VicDF}
	\form(f,f)=\sum_{\bfk\in L_{d}}\int_{K_{\bfk}}\abs{\wgrad f}^{2}\dif\medm\  \text{ for all }f\in\domain.
\end{equation}

\item Note that $(L_{d},+)$ is an Abelian group with unit $\mathbf{0}$, and the inverse of $\bfk\in L_{d}$ is $-\bfk$. For $\bfk\in  L_{d}$ and $x\in\ambient$, the translation map $(\bfk,x)\mapsto \bfk+x$ gives an action of $L_{d}$ on $\ambient$. Let $S_{d}:=\{-1,1\}^{d}\subset L_{d}$. Then $L_{d}=\langle S_{d}\rangle$. It is evident that $L_{d}$ acts geometrically on $(\ambient,\metric)$, has polynomial growth, and 
\begin{equation}
	\sup_{(\bfk,x)\in S_{d}\times \ambient}\metric(x,\bfk+x)\leq C_d<\infty.
\end{equation}
By \eqref{e.VicDF}, we know that $L_{d}$ preserves $(\form,\domain)$. By Theorem \ref{t.Gp}, $(\ambient,\metric,\meas,\form,\domain)$ satisfies $\HKE$, $\RH$ and $\Grad$, with \begin{equation}
	\scale(r):=r^{\alpha_{d}+1}\wedge r^{2},\ r\in(0,\infty).
\end{equation}

\item It is clear that $\scale$ satisfies the local Dini-type condition \ref{e.near0int} for all $p\in(2,\infty)$. Theorem~\ref{t.main} therefore gives $\R{p}$ for all $p\in[2,\infty)$ and $\RR{q}$ for $q\in(1,2]$. By Theorem \ref{t.smallp} and \eqref{e.VicUVG}, these are the only ranges of $p\in(1,\infty)$ such that $\R{p}$ or $\RR{p}$ hold. In summary, for the alternating Vicsek fractafold, for $p\in(1,\infty)$, \begin{equation}
	\R{p}\ \text{ holds if and only if }p\in[2,\infty),\ \text{ and }\RR{p} \text{  holds if and only if } p\in(1,2].
\end{equation}
\end{enumerate}
\end{example}
\begin{remark}
The implication from $\HKE$ to $\Grad$ in \cite[condition~(1.3) and Theorem~2.1]{GY26} requires \emph{strong recurrence}, which is not satisfied by the alternating Vicsek fractafold in $\bZ^d$ for $d\geq2$.
\end{remark}

\appendix
\section{Useful facts}\label{s.appendix}

\begin{lemma}\label{l.phi}
Let $\scale:(0,\infty)\to(0,\infty)$ be a continuous, increasing bijection such that \eqref{e.scale}
 holds for some $C\in(1,\infty)$ and $1<\beta_{1}\leq\beta_{2}<\infty$. Define $\Phi$ by \eqref{e.phi}.
Then $\Phi$ is a $[0,\infty)$-valued lower semicontinuous function satisfying the following properties.
\begin{enumerate}[label=\textup{({\arabic*})},align=right,leftmargin=*,topsep=5pt,parsep=0pt,itemsep=2pt]
\item\label{it.phi1}
For every $t\in(0,\infty)$, the function $\Phi(\cdot,t)$ is non-decreasing, and $\Phi(R,t)>0$ for every $R\in(0,\infty)$.

\item\label{it.phi2}
For every $R\in(0,\infty)$, the function $\Phi(R,\cdot)$ is non-increasing, and
$a\Phi(R,t)\leq\Phi(aR,t)$ for every $a\in[1,\infty)$.

\item\label{it.phi3}
There exists a constant $C_{1}\in(1,\infty)$ depending only on $C$, $\beta_{1}$, and $\beta_{2}$ such that
\begin{equation}\label{e.phi1}
 C_{1}^{-1}\min_{k\in\{1,2\}}
 \left(\frac{\scale(R)}{t}\right)^{\frac{1}{\beta_{k}-1}}
 \leq \Phi(R,t)
 \leq C_{1}\max_{k\in\{1,2\}}
 \left(\frac{\scale(R)}{t}\right)^{\frac{1}{\beta_{k}-1}},
\end{equation}
and
\begin{equation}\label{e.phi2}
 C_{1}^{-1}\left(\frac{R}{r}\right)^{\frac{\beta_{2}}{\beta_{2}-1}}
 \leq \frac{\Phi(R,t)}{\Phi(r,t)}
 \leq C_{1}\left(\frac{R}{r}\right)^{\frac{\beta_{1}}{\beta_{1}-1}},
 \qquad 0<r\leq R<\infty.
\end{equation}
\end{enumerate}
\end{lemma}

\begin{proof}
The lower semicontinuity, assertions~\ref{it.phi1} and~\ref{it.phi2}, together with \eqref{e.phi1}, are proved in \cite[Lemma~5.7]{GK17}. Estimate~\eqref{e.phi2} is proved in \cite[Lemma~2.10]{Mur20}.
\end{proof}

\begin{lemma}\label{l.isom}
Let $E\subset\mathbb R$, and let $\tau:E\to\mathbb R$ satisfy
\begin{equation}
 |\tau(s)-\tau(t)|=|s-t|,
 \qquad s,t\in E.
\end{equation}
If $E$ contains at least two points, then there exist unique
$\varepsilon\in\{-1,1\}$ and $c\in\mathbb R$ such that $\tau(t)=c+\varepsilon t$ for all $t\in E$.
\end{lemma}

\begin{proof}
Choose $s_0,s_1\in E$ with $s_0<s_1$. Since $\tau$ preserves distances, there exists a unique $\varepsilon\in\{-1,1\}$ such that $\tau(s_1)-\tau(s_0)=\varepsilon(s_1-s_0)$. Fix $t\in E$. The two identities
\begin{equation}
 |\tau(t)-\tau(s_i)|=|t-s_i|,
 \text{ for } i\in\{0,1\},
\end{equation}
determine $\tau(t)$ uniquely. The number $\tau(s_0)+\varepsilon(t-s_0)$ satisfies both identities; hence $\tau(t)=\tau(s_0)+\varepsilon(t-s_0)$. Taking $c:=\tau(s_0)-\varepsilon s_0$ proves the assertion. Uniqueness is evdient.
\end{proof}

\begin{lemma}\label{l.measu}
Let $(X,\mathcal X)$ be a measurable space whose $\sigma$-algebra $\sX$ is generated by countably many measurable sets, let $\nu$ be a $\sigma$-finite measure on $(X,\mathcal X)$, and let $p\in[1,\infty)$. There exists a $\mathcal B(L^p(X,\nu))\otimes\mathcal X$-measurable function $\mathsf M:L^p(X,\nu)\times X\to\mathbb R$ such that, for every $f\in L^p(X,\nu)$, we have $\mathsf M(f,x)=f(x)$ for $\nu$-a.e. $x\in X$. Here, $L^p(X,\nu)$ is equipped with the Borel $\sigma$-algebra induced by its norm topology.
\end{lemma}

\begin{proof}
Choose a collection $\{G_k\}_{k\in\mathbb N}\subset\mathcal X$ such that
$\mathcal X=\sigma(\{G_k\}_{k\in\mathbb N})$. Since $\nu$ is $\sigma$-finite, there exists a disjoint measurable partition
$X=\bigsqcup_{j=1}^{\infty}X_j$ such that $X_j\in\mathcal X$ and
$\nu(X_j)<\infty$ for every $j\in\mathbb N$.

For each $k\in\mathbb N$, define $G_k^{(0)}:=G_k$ and
$G_k^{(1)}:=X\setminus G_k$. For each $n\in\mathbb N$, define
\begin{equation}
 \mathcal P_n:=
 \Sett{X_j\cap\bigcap_{k=1}^{n}G_k^{(\varepsilon_k)}}
 {j\in\mathbb N,\ (\varepsilon_1,\ldots,\varepsilon_n)\in\set{0,1}^n}.
\end{equation}
Then $\mathcal P_n$ is a countable measurable partition of $X$, after possibly discarding empty sets, and every $A\in\mathcal P_n$ satisfies $\nu(A)<\infty$. Let $\mathcal X_n:=\sigma(\mathcal P_n)$. Then
$\mathcal X_n\subset\mathcal X_{n+1}$ for every $n\in\mathbb N$, and $\sigma\left(\bigcup_{n\in\mathbb N}\mathcal X_n\right)=\mathcal X$.

For every $f\in L^p(X,\nu)$ and every $x\in X$, define
\begin{equation}\label{e.meas1}
 \mathsf M_n(f,x):=
 \sum_{\substack{A\in\mathcal P_n\\ \nu(A)\in(0,\infty)}}
 \left(\frac{1}{\nu(A)}\int_A f\dif\nu\right)\one_A(x).
\end{equation}
Since $\mathcal P_n$ is a partition, at most one summand in \eqref{e.meas1} is nonzero at each $x\in X$. We first show that the map
$(f,x)\mapsto\mathsf M_n(f,x)$ is measurable. Indeed, for every
$A\in\mathcal P_n$ with $0<\nu(A)<\infty$, H\"older's inequality gives
\begin{equation}
 \left|\int_A f\dif\nu\right|
 \leq \nu(A)^{1-1/p}\lVert f\rVert_{L^p(X,\nu)}.
\end{equation}
Thus $f\mapsto\int_A f\dif\nu$ is a continuous linear functional on
$L^p(X,\nu)$. Since $x\mapsto\one_A(x)$ is measurable, the map
\begin{equation}
 (f,x)\mapsto
 \left(\frac{1}{\nu(A)}\int_A f\dif\nu\right)\one_A(x)
\end{equation}
is jointly measurable. Because $\mathcal P_n$ is countable and the sum in
\eqref{e.meas1} has at most one nonzero term at each point,
$(f,x)\mapsto\mathsf M_n(f,x)$ is jointly measurable.

Moreover, since $\mathcal P_n$ is a partition into sets of finite measure, the measure space $(X,\mathcal X_n,\nu)$ is $\sigma$-finite, and
\begin{equation}
 \mathsf M_n(f,\cdot)=\mathbb E^{\nu}[f\mid\mathcal X_n]
 \qquad \nu\text{-a.e. on }X.
\end{equation}
Define
\begin{equation}
 A:=
 \bigcap_{k\in\mathbb N}
 \bigcup_{N\in\mathbb N}
 \bigcap_{m,n\geq N}
 \Sett{(f,x)\in L^p(X,\nu)\times X}
 {\left|\mathsf M_m(f,x)-\mathsf M_n(f,x)\right|<k^{-1}}.
\end{equation}
The set $A$ is measurable by the joint measurability of
$\{\mathsf M_n\}_{n\in\mathbb N}$. Define
\begin{equation}
 \mathsf M(f,x):=
 \begin{cases}
 \displaystyle\lim_{n\to\infty}\mathsf M_n(f,x), & (f,x)\in A,\\
 0, & (f,x)\notin A.
 \end{cases}
\end{equation}
Then $\mathsf M$ is measurable on $L^p(X,\nu)\times X$. Finally, since
$\mathcal X_n\uparrow\mathcal X$, the martingale convergence theorem on
$\sigma$-finite measure spaces \cite[Theorem~6.1.1]{Str25} gives, for every $f\in L^p(X,\nu)$,
\begin{equation}
 \lim_{n\to\infty}\mathsf M_n(f,x)=\lim_{n\to\infty}\mathbb E^{\nu}[f\mid\mathcal X_n](x)= f(x),
 \  \text{\ for $\nu$-a.e. }x\in X.
\end{equation}
Therefore, $x\mapsto\mathsf M(f,x)$ is a measurable representative of $f$.
\end{proof}

\begin{proposition}[{Cf. \cite[Proposition 3.2]{Mur24}}]\label{p.Whitney}
Let $\graph=(\vertex,\edge)$ be a connected locally finite graph for which there exists $C\in(1,\infty)$ such that
\begin{equation}
 \#B_{\graph}(v,2R)\leq C\#B_{\graph}(v,R),
 \qquad (v,R)\in\vertex\times(0,\infty).
\end{equation}
Let $\vartheta\in(0,1/2)$ and $\emptyset\neq\Omega\subsetneq\vertex$. For $x\in\Omega$, set $\delta_{\Omega}(x):=\metric_{\graph}(x,\vertex\setminus\Omega)$.
Then there exists a collection of balls
\begin{equation}
 \{B_i=B_{\graph}(x_i,r_i):x_i\in\Omega,\ r_i>0,\ i\in I\}
\end{equation}
satisfying the following properties.
\begin{enumerate}[label=\textup{(\roman*)},align=right,leftmargin=*,topsep=5pt,parsep=0pt,itemsep=2pt]
\item The balls $\{B_i:i\in I\}$ are pairwise disjoint.

\item For every $i\in I$, $r_i=\frac{\vartheta}{1+\vartheta}\delta_{\Omega}(x_i)$;

\item If $K_{\vartheta}:=2(1+\vartheta)\in(2,3)$, then $\Omega=\bigcup_{i\in I}B_{\graph}(x_i,K_{\vartheta}r_i)$.
\end{enumerate}
Such a collection of balls is called a \emph{$\vartheta$-Whitney covering of $\Omega$}.
\end{proposition}

\begin{lemma}\label{l.KerQ}
For any $s\in[0,\infty)$ and $h,\kappa\in[0,\infty)$, define
\begin{equation}\label{e.Qkern}
 Q_{s,h,\kappa}(t)
 :=\frac{1}{\pi^{\frac{1}{2}}}
 \left(
 \frac{\exp(-\kappa(t-s))}{(t-s)^{\frac{1}{2}}}\one_{(s,\infty)}(t)
 -\frac{\exp(-\kappa(t-s-h))}{(t-s-h)^{\frac{1}{2}}}\one_{(s+h,\infty)}(t)
 \right),
 \ t\in(0,\infty).
\end{equation}
Then the following assertions hold.
\begin{enumerate}[label=\textup{({\arabic*})},align=right,leftmargin=*,topsep=5pt,parsep=0pt,itemsep=2pt]
\item\label{it.series0}
If $t\in(0,s]$, then $Q_{s,h,\kappa}(t)=0$. If $t\in(s,s+h]$, then
\begin{equation}
 Q_{s,h,\kappa}(t)
 =\frac{\exp(-\kappa(t-s))}{\pi^{\frac{1}{2}}(t-s)^{\frac{1}{2}}}.
\end{equation}

\item\label{it.series1}
The Laplace transform of $Q_{s,h,\kappa}$ is
\begin{equation}\label{e.QLap}
 \int_{0}^{\infty}Q_{s,h,\kappa}(t)e^{-t\lambda}\dif t
 =(\lambda+\kappa)^{-\frac{1}{2}}e^{-s\lambda}(1-e^{-h\lambda}),
 \qquad \lambda\in[0,\infty).
\end{equation}
For $\lambda=\kappa=0$, we define the right hand side in \eqref{e.QLap} to be $0$ by continuous extension.
\item\label{it.series2}
There exists $C\in(0,\infty)$ such that
\begin{equation}\label{e.Qshif}
 \sup_{\kappa\in(0,\infty)}\lvert Q_{s,h,\kappa}(t)\rvert
 \leq C\frac{h}{(t-s)^{\frac{3}{2}}},
 \qquad t\in[s+2h,\infty)\cap (s,\infty).
\end{equation}

\item\label{it.series4}
For every $A,c\in(0,\infty)$, there exists
$C=C(A,c,\volume,\scale)\in(0,\infty)$ such that, for every
$R\in(0,\infty)$, every $h\in\left(0,\frac{\scale(R)}{2(A+2)}\right]$ and every $s\in(0,Ah)$, one has
\begin{equation}\label{e.Qweight}
 \sup_{\kappa\in(0,\infty)}
 \int_{0}^{\infty}
 \lvert Q_{s,h,\kappa}(t)\rvert\Lambda(t)
 \exp(-c\Phi(R,t))\dif t
 \leq C\frac{h}{\scale(R)}
 \frac{\scale(R)^{\frac{1}{2}}}
 {(R\wedge\scale(R))^{\frac{1}{2}}R^{\frac{1}{2}}
 \volume(R)^{\frac{1}{2}}}.
\end{equation}
Here $\Lambda$ is the function defined in \eqref{e.deflamb}.
\end{enumerate}
\end{lemma}

\begin{proof}
\begin{enumerate}[label=\textup{({\arabic*})},align=right,leftmargin=*,topsep=5pt,parsep=0pt,itemsep=2pt]
\item[\ref{it.series0}]
This follows immediately from the definition in \eqref{e.Qkern}. Indeed, both indicator functions vanish when $t\leq s$, whereas only the first indicator function is nonzero when $s<t\leq s+h$.

\item[\ref{it.series1}]
Fix $a\in[0,\infty)$ and $\lambda\in[0,\infty)$ with $\lambda\vee \kappa>0$. By the substitution $r=t-a$ and the identity $\Gamma(1/2)=\pi^{\frac{1}{2}}$,
\begin{equation}\label{e.Qshhf}
\frac{1}{\pi^{\frac{1}{2}}}\int_{a}^{\infty}
 \frac{\exp(-\kappa(t-a))}{(t-a)^{\frac{1}{2}}}
 e^{-t\lambda}\dif t=\frac{e^{-a\lambda}}{\pi^{\frac{1}{2}}}
 \int_{0}^{\infty}e^{-(\lambda+\kappa)r}r^{-\frac{1}{2}}\dif r
 =(\lambda+\kappa)^{-\frac{1}{2}}e^{-a\lambda}.
\end{equation}
Applying \eqref{e.Qshhf} first with $a=s$ and then with $a=s+h$, and subtracting the resulting identities, gives
\begin{align}
 \int_{0}^{\infty}Q_{s,h,\kappa}(t)e^{-t\lambda}\dif t
 =(\lambda+\kappa)^{-\frac{1}{2}}
 \bigl(e^{-s\lambda}-e^{-(s+h)\lambda}\bigr)=(\lambda+\kappa)^{-\frac{1}{2}}e^{-s\lambda}(1-e^{-h\lambda}).
\end{align}
This proves \eqref{e.QLap} in the case $\lambda\vee\kappa>0$.

If $\lambda=\kappa=0$, the identity also holds for $h=0$ because $Q_{s,0,0}\equiv0$. Suppose $h>0$. Then
\begin{equation}\label{e.revnew-zero-laplace}
\begin{aligned}
\int_0^{\infty} Q_{s,h,0}(t)\dif t=\lim_{T\to\infty}\int_0^T Q_{s,h,0}(t)\dif t=\lim_{T\to\infty}\frac{2}{\sqrt\pi}\left(\sqrt{T-s}-\sqrt{T-s-h}\right)=0.
\end{aligned}
\end{equation}

\item[\ref{it.series2}]
For $h=0$ and $t>s$, both sides of \eqref{e.Qshif} are zero. Suppose $h>0$. For $r\in(0,\infty)$, put $F_{\kappa}(r):=e^{-\kappa r}r^{-\frac{1}{2}}$.
If $t\geq s+2h$ and $x:=t-s$, then $x\geq2h$ and
\begin{equation}\label{e.Q1df}
 \pi^{\frac{1}{2}}Q_{s,h,\kappa}(t)
 =F_{\kappa}(x)-F_{\kappa}(x-h)
 =\int_{x-h}^{x}F_{\kappa}'(r)\dif r.
\end{equation}
For $r\in(0,\infty)$, we have $ F_{\kappa}'(r)
 =-e^{-\kappa r}
 (\kappa r^{-\frac{1}{2}}+\frac{1}{2}r^{-\frac{3}{2}})$. Since $\sup_{u\in[0,\infty)}ue^{-u}<\infty$, we have $\sup_{\kappa\in(0,\infty)}\lvert F_{\kappa}'(r)\rvert
 \leq Cr^{-\frac{3}{2}}$ for all $r\in(0,\infty)$.
Moreover, $r\in[x-h,x]$ implies $r\geq x/2$. Therefore,
\begin{align}
 \sup_{\kappa\in(0,\infty)}\lvert Q_{s,h,\kappa}(t)\rvert\overset{\eqref{e.Q1df}}{\leq}
 \frac{1}{\pi^{\frac{1}{2}}}
 \int_{x-h}^{x}Cr^{-\frac{3}{2}}\dif r
 \leq Chx^{-\frac{3}{2}}=C\frac{h}{(t-s)^{\frac{3}{2}}}.
\end{align}
This proves \eqref{e.Qshif}.

\item[\ref{it.series4}]
Put $T:=\scale(R)$. We first record an estimate that will be used repeatedly. Fix $t\in(0,T]$ and set $r:=\scale^{-1}(t)$. Then $r\leq R$. From the definition of $\Lambda$ in \eqref{e.deflamb}, the doubling property \eqref{e.VD1}, and the lower scaling bound for $\scale$, we obtain
\begin{align}
 \frac{\Lambda(t)}{\Lambda(T)}
 &=\left(\frac{R\wedge T}{r\wedge t}\right)^{\frac{1}{2}}
 \left(\frac{R}{r}\right)^{\frac{1}{2}}
 \left(\frac{\volume(R)}{\volume(r)}\right)^{\frac{1}{2}}\leq C
 \max\left(\frac{R}{r},\frac{T}{t}\right)^{\frac{1}{2}}
 \left(\frac{R}{r}\right)^{\frac{1+d_{2}}{2}}\\
 &\leq C\left(\frac{T}{t}\right)^{a},
 \ \text{ with } a:=\frac{1}{2}+\frac{1+d_{2}}{2\beta_{1}}.
 \label{e.lbdcpr}
\end{align}
By the lower estimate for $\Phi$ implied by the scaling condition on $\scale$, and since $T/t\geq1$,
\begin{equation}\label{e.plowq}
 \Phi(R,t)
 \geq c_{0}\left(\frac{T}{t}\right)^{\frac{1}{\beta_{2}-1}}.
\end{equation}
Consequently, for every $N\in(0,\infty)$, there exists $C_{N}\in(0,\infty)$ such that
\begin{equation}\label{e.lbdpdc}
 \Lambda(t)e^{-c\Phi(R,t)}
 \leq C_{N}\Lambda(T)\left(\frac{t}{T}\right)^{N},
 \ \text{ for all }t\in (0, T].
\end{equation}
Indeed, this follows from \eqref{e.lbdcpr}, \eqref{e.plowq}, and $\sup_{u\in[1,\infty)}u^{a+N}
 \exp\left(-cu^{\frac{1}{\beta_{2}-1}}\right)<\infty$.
Set $L:=2(A+2)h$. The assumption on $h$ gives $L\leq T$. Set
\begin{align}
 I_{1}&:=\int_{0}^{L}\lvert Q_{s,h,\kappa}(t)\rvert
 \Lambda(t)e^{-c\Phi(R,t)}\dif t,\notag\\
 I_{2}&:=\int_{L}^{T}\lvert Q_{s,h,\kappa}(t)\rvert
 \Lambda(t)e^{-c\Phi(R,t)}\dif t,\notag\\
 I_{3}&:=\int_{T}^{\infty}\lvert Q_{s,h,\kappa}(t)\rvert
 \Lambda(t)e^{-c\Phi(R,t)}\dif t.
\end{align}

For $I_{1}$, the definition of $Q_{s,h,\kappa}$ and \eqref{e.lbdpdc} with $N=1/2$ give
\begin{align}
 I_{1}\leq C\frac{\Lambda(T)}{T^{\frac{1}{2}}}
 \left(
 \int_{s}^{L}(t-s)^{-\frac{1}{2}}t^{\frac{1}{2}}\dif t
 +\int_{s+h}^{L}(t-s-h)^{-\frac{1}{2}}t^{\frac{1}{2}}\dif t
 \right)\leq C\frac{\Lambda(T)L}{T^{\frac{1}{2}}}
 \leq C_{A}\frac{h\Lambda(T)}{T^{\frac{1}{2}}}.
 \label{e.q-near}
\end{align}
Here and below, an integral is understood to be zero whenever its lower endpoint is not smaller than its upper endpoint.

If $t\geq L$, then, because $s<Ah$, we have $ t-s\geq L-s>(A+4)h\geq2h$ and $t-s\geq\frac{1}{2}t$. Thus \eqref{e.Qshif} and \eqref{e.lbdpdc} with $N=3/2$ imply
\begin{align}
 I_{2}
 &\leq Ch\int_{L}^{T}t^{-\frac{3}{2}}
 \Lambda(T)\left(\frac{t}{T}\right)^{\frac{3}{2}}\dif t
 \leq C\frac{h\Lambda(T)}{T^{\frac{1}{2}}}.
 \label{e.q-middle}
\end{align}

Finally, $s<Ah\leq A(2(A+2))^{-1}T<T/2$. Hence, for $t\geq T$, one has $t-s\geq t/2$. Since $\Lambda$ is non-increasing,
\begin{align}
 I_{3}
 &\leq Ch\Lambda(T)\int_{T}^{\infty}t^{-\frac{3}{2}}\dif t
 \leq C\frac{h\Lambda(T)}{T^{\frac{1}{2}}}.
 \label{e.q-far}
\end{align}
Combining \eqref{e.q-near}, \eqref{e.q-middle}, and \eqref{e.q-far}, we obtain
\begin{equation}
 \sup_{\kappa\in(0,\infty)}
 \int_{0}^{\infty}\lvert Q_{s,h,\kappa}(t)\rvert
 \Lambda(t)e^{-c\Phi(R,t)}\dif t
 \leq C\frac{h\Lambda(T)}{T^{\frac{1}{2}}}.
\end{equation}
Since $T=\scale(R)$ and $\scale^{-1}(T)=R$, $ \Lambda(T)
 ={(R\wedge\scale(R))^{-\frac{1}{2}}R^{-\frac{1}{2}}
 \volume(R)^{-\frac{1}{2}}}$. Therefore,
\begin{align}
 \frac{h\Lambda(T)}{T^{\frac{1}{2}}}
 &=\frac{h}{\scale(R)}
 \frac{\scale(R)^{\frac{1}{2}}}
 {(R\wedge\scale(R))^{\frac{1}{2}}R^{\frac{1}{2}}
 \volume(R)^{\frac{1}{2}}},
\end{align}
which proves \eqref{e.Qweight}.
\qedhere
\end{enumerate}
\end{proof}
	\begin{proposition}\label{p.equiH1}
Let $p\in(2,\infty)$, and let $\Psi:(0,\infty)\to(0,\infty)$ be a continuous increasing bijection. Then
\[
 \int_0^1 t^{-\frac12-\frac1p}\Psi^{-1}(t)^{-1+\frac2p}\dif t<\infty
 \quad\Longleftrightarrow\quad
 \int_0^{\Psi^{-1}(1)}\left(\frac{\Psi(r)}{r^2}\right)^{\frac12-\frac1p}\frac{\dif r}{r}<\infty.
\]
\end{proposition}
\begin{proof}
Put $a=\frac12-\frac1p>0$ and $r_0=\Psi^{-1}(1)$. For $0<t<1$,
\[
 \Psi^{-1}(t)^{-2a}
 =r_0^{-2a}+2a\int_{\Psi^{-1}(t)}^{r_0}r^{-2a-1}\dif r.
\]
Multiply by $t^{a-1}$ and integrate over $(0,1)$. Tonelli's theorem and the equivalence $\Psi^{-1}(t)\le r\Longleftrightarrow t\le\Psi(r)$ give the identity of extended nonnegative integrals
\begin{align}
 \int_0^1t^{a-1}\Psi^{-1}(t)^{-2a}\dif t
 &=\frac{r_0^{-2a}}a
 +2a\int_0^{r_0}r^{-2a-1}\left(\int_0^{\Psi(r)}t^{a-1}\dif t\right)\dif r\\
 &=\frac{r_0^{-2a}}a+2\int_0^{r_0}\Psi(r)^ar^{-2a-1}\dif r.
\end{align}
The asserted equivalence follows.
\end{proof}

\noindent \textbf{Acknowledgments.} F.B.~is partially supported by grant 10.46540/4283-00175B from the Independent Research Fund Denmark, by the Villum Investigator grant \emph{Stochastic Analysis in Aarhus}, and by the European Research Council (ERC) under the European Union's Horizon Europe research and innovation programme (RanGe project, Grant Agreement No.~101199772). A.C.~is partially funded by the Villum Investigator grant \emph{Stochastic Analysis in Aarhus}. L.C.~is partially supported by grant 10.46540/4283-00175B from the Independent Research Fund Denmark.

\vspace{0.4cm}
\noindent \textbf{AI disclosure statement.} Generative artificial intelligence tools were used during the initial exploratory phase of this work to assist in considering possible directions and organizing preliminary ideas, and during the final preparation phase to assist with linguistic and stylistic polishing. The authors take full responsibility for the content of the manuscript.

\vspace{10pt}

\noindent Fabrice~Baudoin:
\vspace{-3pt}

\noindent Department of Mathematics, Aarhus University, 8000 Aarhus C, Denmark
\vspace{-3pt}

\noindent \texttt{fbaudoin@math.au.dk}
\vspace{5pt}

\noindent Aobo~Chen:
\vspace{-3pt}

\noindent Department of Mathematics, Aarhus University, 8000 Aarhus C, Denmark
\vspace{-3pt}

\noindent \texttt{aobochen.math@hotmail.com} / \texttt{aobochen@math.au.dk}

\vspace{5pt}

\noindent Li~Chen:
\vspace{-3pt}

\noindent Department of Mathematics, Aarhus University, 8000 Aarhus C, Denmark
\vspace{-3pt}

\noindent \texttt{lchen@math.au.dk}

\end{document}